%% file: main.tex
\documentclass[format=smallextended]{svjour3} 

\usepackage[utf8]{inputenc}
\usepackage{authblk}
\usepackage{graphicx}
\usepackage{subcaption}
\usepackage{multirow}
\usepackage{amsmath,amssymb,amsfonts}
\usepackage{mathrsfs}
\usepackage[title]{appendix}
\usepackage{xcolor}
\usepackage{textcomp}
\usepackage{manyfoot}
\usepackage{booktabs}
\usepackage{algpseudocode}
\usepackage{listings}
\usepackage{hyperref}
\usepackage[maxcitenames=1,style=authoryear-comp]{biblatex}

\usepackage{bm}
\usepackage{soul}
\usepackage{epstopdf}
\usepackage[font=small]{caption}
\usepackage{comment}
\usepackage{algorithm}
\usepackage{placeins}
\usepackage{units}

\usepackage{tabularx}
\usepackage{tcolorbox}
\usepackage{tikz}
\usepackage{listings}
\usetikzlibrary{shapes.geometric, arrows.meta, positioning, backgrounds}
\usepackage{gensymb}
\pgfmathsetseed{\number\pdfrandomseed}

\newtheorem{assumption}{Assumption}

\newenvironment{ldescription}[1]
  {\begin{list}{}
   {\renewcommand\makelabel[1]{##1\hfill}
   \settowidth\labelwidth{\makelabel{#1}}
   \setlength\leftmargin{\labelwidth}
   \addtolength\leftmargin{\labelsep}}}
  {\end{list}}

\definecolor{purple}{rgb}{0.6, 0.2, 0.8}
\definecolor{darkred}{rgb}{0.7, 0.0, 0.0}
\definecolor{darkgreen}{rgb}{0.0, 0.5, 0.0}
\definecolor{lightblue}{rgb}{0.68, 0.85, 0.9}

\title{Towards Reliability-Aware Active Distribution System Operations: A Sequential Convex Programming Approach}
\titlerunning{Towards Reliability-Aware Active Distribution System Operations}
\author{Gejia Zhang \thanks{gejia.zhang@rutgers.edu}and Robert Mieth \thanks{robert.mieth@rutgers.edu}}
\institute{Department of Industrial and System Engineering, Rutgers University, 96 Frelinghuysen Road, Piscataway, 08854, NJ, USA}
\date{}
\authorrunning{G. Zhang and R. Mieth}

\begin{document}

\tcbset{colframe=black,colback=white,boxrule=0.05mm,arc=2mm,outer arc=2mm}

\maketitle

\begin{abstract}
The increasing demand for electricity and the aging infrastructure of power distribution systems have raised significant concerns about future system reliability. Failures in distribution systems, closely linked to system usage and environmental factors, are the primary contributors to electricity service interruptions. The integration of distributed energy resources (DER) presents an opportunity to enhance system reliability through optimized operations. 
This paper proposes a novel approach that explicitly incorporates both decision- and context-dependent reliability into the optimization of control setpoints for DERs in active distribution systems. The proposed model captures how operational decisions and ambient temperature impact the likelihood of component failures, enabling a balanced approach to cost efficiency and reliability. By leveraging a logistic function model for component failure rates and employing a sequential convex programming method, the model addresses the challenges of non-convex optimization under decision-dependent uncertainty. Numerical case study on a modified IEEE $33$-bus test system demonstrates the effectiveness of the model in dynamically adjusting power flows and enhancing system robustness under varying environmental conditions and operational loads. The results highlight the potential of DERs to contribute to distribution system reliability by efficiently managing power flows and responding to fluctuating energy demands.
\end{abstract} \\

\noindent \textbf{Keywords:} Decision-dependent uncertainty, non-convex optimization, optimal power flow, power system reliability

\newpage

\section{Introduction}

The global demand for electricity is projected to increase to approximately 150\% by 2035 (\cite{emmanuel2017}), driven by economic growth and the concurrent electrification of traditionally fossil-fueled loads such as cars and building heating (\cite{griffith2022electrify,zhang2019data,andrews2022data}).
This anticipated surge in electricity demand in combination with aging infrastructure has raised concerns about the reliability and adequacy of power distribution systems in the future (\cite{njbpu2022gridmod}). 
This concern is amplified by the fact that failures in distribution systems account for the majority of interruptions in electricity service with a distinct correlation to system usage and environmental factors (\cite{richard_epdr_3,campbell2012weather}).

The reliability of electric distribution systems has been considered predominantly from a planning perspective, i.e., long-term investments in protective equipment (\cite{izadi2019optimal,billinton1996optimal}), power quality control equipment (\cite{baran1989optimal}), and line hardening (\cite{milad2023,wang2017robust}), or from the perspective of system \textit{resilience}, i.e., preparation for, mitigation of, and recovery from extreme events (\cite{kim2018enhancing,wang2016resilience,hanchen2022,bajpai2016novel}). 
At the same time, efforts to abate fossil fuel-based electric power generation has introduced a growing fleet of distributed energy resources (DERs) at the distribution level, including
wind, solar, and other controllable energy systems like distributed generators (DGs), battery storage, and demand response (DR).
This development has sparked the interest of utilities to explore ``\textit{DER enabled reliability}'' (\cite{njbpu2022gridmod}) as an option for distributed resources to positively impact system reliability at the operational stage.  
This paper explores this option under the lens of decision-dependent and contextual reliability for active distribution system operations.

Most research on reducing service interruptions in the distribution system focuses on planning system reliability (e.g., \cite{izadi2019optimal,billinton1996optimal,baran1989optimal,milad2023}) or on system resilience in anticipation of and during extreme events (e.g., \cite{wang2017robust,kim2018enhancing}). 
In addition to the \textit{planning} tools described in the references above,
there are some approaches that discuss event-driven reliability enhancements in the \textit{operational} stage.
For example, \cite{kim2018enhancing} and \cite{bajpai2016novel} propose tools for the optimal deployment of DGs and mobile Battery Energy Storage Systems (BESSs), respectively, so that the expected cost of unserved energy during an anticipated extreme event is minimized.
As a different approach, \cite{wang2016resilience} and \cite{hanchen2022} discuss a combination of preventive generator dispatch, load shedding, and topology switching to minimize unserved energy during some critical scenarios.

Apart from the use of DERs for emergency services as in \cite{kim2018enhancing} or \cite{bajpai2016novel}, the connection between DER operation and distribution system reliability has previously been considered more implicitly by ensuring that technical safety limits, including voltage and thermal equipment limits, are maintained. 
This objective has gained importance over the past years because the generation of electric power at distribution level, especially from weather-dependent wind and solar resources complicates maintaining voltage limits (\cite{nasif2016,dall2017chance}), accelerates equipment degradation (\cite{andrianesis2021optimal}), and amplifies the effect of temperature-dependent resource properties on system operations (\cite{moein2016}).
Various options to ensure the safe operation of distribution systems with a high penetration of (renewable) DERs have been proposed, including the optimal placement and control of BESSs (\cite{yuqing2018}), controllable DGs (\cite{mieth2018data}), and DR programs (\cite{mieth2019online,jamshid2013}).
However, keeping system operations within predefined safety limits as in the references mentioned above offers no explicit insight into relevant reliability metrics and how they are affected by system operations.
Such considerations have been discussed for transmission systems, e.g., using weather-dependent equipment failure rates to improve unit commitment (\cite{miguel2016}), analyze resilience during extreme weather events (\cite{mathaios2016,abdin2019}), or optimize decision-dependent outage risk (\cite{mieth2022risk}).

The approach proposed in this paper \textit{explicitly} considers distribution system reliability when optimizing control setpoints for DERs by considering component failure rate as a function of operational context. 
Contextual reliability has been discussed by previous works with a focus on how component failure rates depend on different weather variables. 
For example,
\cite{billinton2005consideration} and \cite{billinton2006distribution} propose forecasting models for component failure and repair rates under varying weather conditions via a Markov approach.  
Similarly, \cite{hani2023} introduce a weather-aware predictive reliability assessment method for distribution systems. 
Using empirical data, \cite{luca2022} propose a classifier for defining heatwaves based on changes in the reliability of distribution system components during high-temperature events. 
The study models the impact of heatwaves on the component reliability, along with other covariates, using formulations adopted from the Cox proportional hazards model~(\cite{cox1972}). 

In addition to context from external factors such as weather, in this paper we propose a novel approach that also considers \textit{decision-dependent reliability}, i.e., component failure rate as a function of their utilization. 
To the best of our knowledge, this is the first paper to explore this approach.
Similarly to \cite{luca2022}, we use logistic regression that results in a component reliability model that is equivalent to the Cox proportional hazards model within a reasonable approximation. This allows us to internalize the contextual reliability of the components of a radial distribution system in the objective of a decision-making problem that co-optimizes its cost and reliability.
Solving the resulting problem is challenging because it is a non-convex optimization problem with decision-dependent uncertainty. 
To overcome this, we use a sequential convex programming approach that iterates between the exact non-convex expression of component reliability and its convex linear approximation.
Considering the impact of ambient temperature and net power demand through each bus and line on system reliability, our model ensures robustness under varying environmental conditions and operational loads. It demonstrates its ability to balance operational and the cost of expected-energy-not-served.
Furthermore, with its multiple time-step analysis, the model allows the distribution system to dynamically adjust power flows according to different times of the day and varying power demands as well as varying ambient temperature, enhancing the system's ability to meet fluctuating energy requirements efficiently. 

The paper is structured as follows. Section~\ref{sec:2_notation} lists the required notations.
Section~\ref{sec:3_math_model} describes our model of an active distribution grid with an AC power flow formulation and a baseline cost-minimizing objective.
Next, it introduces our approach to modeling and internalizing reliability into the model objective and outlines the proposed iterative solution algorithm.
Section~\ref{sec:4_case_study} details our case study, including data description, parameter estimation, and results.
Section~\ref{sec:conclusion} concludes the paper and provides future research directions. 

\section{Nomenclature} \label{sec:2_notation}

\noindent \textit{A. Superscripts}
\begin{ldescription}{$\rm{Pr}_{0}
(t)/\rm{Pr}_{i}^{\diamond}(t)xxx$}
\setlength{\itemsep}{0pt}
    \item [$\text{B}$] BESS.
    \item [$\text{B,c}$] BESS charging.
    \item [$\text{B,d}$] BESS discharging.
    \item [$\text{b}$] Bus.
    \item [$\text{DG}$] Distributed generation unit.
    \item [$\text{DR}$] Demand response.
    \item [$\text{l}$] Line.
\end{ldescription}

\noindent \textit{B. Sets and Indices}
\begin{ldescription}{$\rm{Pr}_{0}(t)/\rm{Pr}_{i}^{\diamond}(t)xxx$}
\setlength{\itemsep}{0pt}
    \item [$\mathcal{A}_i$] Singleton set with upstream (ancestor) bus of bus $i$.
    \item [$\mathcal{MCS}_{{b}_i}$] Minimum cut-set of lines connection the substation and bus $i$.
    \item [$\mathcal{N}$] Set of buses defined as $\mathcal{N} = \{0, 1, \cdots, N\}$.
    \item [$\mathcal{N}^+$] Set of buses without substation, $\mathcal{N}^+ = \mathcal{N}\setminus\{0\}$.
    \item [$\mathcal{N}_{\diamond}$] Subset of buses that host $\diamond = \{ \text{BESS}, \text{DG}, \text{DR} \}$.
    \item [$\mathcal{S}c_i$] Set of downstream (children) buses of bus $i$.
    \item [$\mathcal{T}$] Set of modeled time steps, indexed by $t$. (Time is modeled discretely and any value related to a time step $t$ represents the value at the end of the discrete time step $t$.)
\end{ldescription}

\noindent Indices $i \in \mathcal{N}^{+}$ refer to non-substation buses and lines towards these buses. \\
Indices $t\in\mathcal{T}$ enumerate time steps.\\

\noindent \textit{C. Variables} \\
\begin{ldescription}{$\rm{Pr}_{0}(t)/\rm{Pr}_{i}^{\diamond}(t)xxx$}
\setlength{\itemsep}{0pt}
    \item [$f_{i, t}^p/f_{i, t}^q$] Active/reactive power flow from bus $\mathcal{A}_i$ to bus $i$ (p.u.).
    \item [$l_{i, t}$] Square of current magnitude in line $i$ (p.u.).
    \item [$p_{0, t}/q_{0, t}$] Active/reactive power provided by substation (p.u.).
    \item [$p_{i, t}^{\diamond}/q_{i, t}^{\diamond}$] Active power provided by $\diamond$, $\diamond = \{\text{DG}, \text{(B,c)}, \text{(B,d)}, \text{DR} \}$  (p.u.).
    \item [$p_{i, t}^{\diamond}/q_{i, t}^{\diamond}$] Reactive power provided by $\diamond$, $\diamond = \{\text{DG}, \text{DR} \}$  (p.u.).
    \item [$\tilde{p}_{0, t}^{\rm b}$] Auxiliary variable for modeling $\left| p_{0, t} \right|$.
    \item [$\tilde{p}_{i, t}^{\rm b}$] Auxiliary variable for modeling $\left| - p_{i, t}^{\rm c} + p_{i, t}^{\rm DG} - p_{i, t}^{\rm B, c} + p_{i, t}^{\rm B, d} + p_{i, t}^{\rm DR} \right|$.
    \item [$\bm{p}_{i,t}$] Active power defined as a column vector $\bm{p}_{i, t} \!\! = \!\! (p_{i, t}^{\text{DG}}\!, p_{i, t}^{\text{B,c}}\!, p_{i, t}^{\text{B,d}}\!, p_{i, t}^{\text{DR}})^T$.
    \item [${\rm {SOC}}_{i, t}^{\text{B}}$] BESS state-of-charge (\%).
    \item [$u_{i, t}^{\text{B}}$] Binary BESS charging indicator. \newline
    (If $u_{i, t}^{\text{B}} = 1$ then $p_{i, t}^{\text{B,c}},q_{i, t}^{\text{B,c}}\ge 0$ and $p_{i, t}^{\text{B,d}},q_{i, t}^{\text{B,d}}= 0$; If $u_{i, t}^{\text{B}} = 0$ then $p_{i, t}^{\text{B,c}},q_{i, t}^{\text{B,c}} = 0$ and $p_{i, t}^{\text{B,d}},q_{i, t}^{\text{B,d}} \ge 0$.)
    \item [$u_{i, t}^{\text{DR}}$] Binary DR usage indicator. \newline
    (If $u_{i, t}^{\text{DR}} = 1$ then $p_{i, t}^{\text{DR}},q_{i, t}^{\text{DR}}\geq 0$; if $u_{i, t}^{\text{DR}} = 0$ then $p_{i, t}^{\text{DR}},q_{i, t}^{\text{DR}} = 0$.)
    \item [$v_{0, t}$] Square of voltage magnitude at substation (p.u.). 
    \item [$v_{i, t}$] Square of voltage magnitude (p.u.).
    \item [$\bm{v}$] Vector of all CM decision variables.
    \item [$\tilde{\bm{v}}$] Vector of all CRM-APPX decision variables.
\end{ldescription}

\noindent \textit{D. Parameter}
\begin{ldescription}{$\rm{Pr}_{0}(t)/\rm{Pr}_{i}^{\diamond}(t)xxx$}
\setlength{\itemsep}{0pt}
    \item [$\rm{AT}_t$] Ambient temperature $(\text{C}^{\circ})$.
    \item [$c_{0, t}$] Per-unit energy cost at substation (\$).
    \item [$c_{i, t}^{\diamond}$] Per-unit energy cost from distributed resource; $\diamond = \{\text{DG}, (\text{B,c}),$ \\
    $(\text{B,d}), \text{DR}\}$ (\$).
    \item [$\bm{c}_{i, t}$] Per-unit cost energy cost from each resource type defined as a column vector $\bm{c}_{i, t} = (c_{i, t}^{\text{DG}}, c_{i, t}^{\text{B,c}}, c_{i, t}^{\text{B,d}}, c_{i, t}^{\text{DR}})^{T}$.
    \item [$G_{i}/B_{i}$] Shunt conductance/susceptance of the line $i$ (p.u.).
    \item [$p_{i, t}^{\rm c}/q_{i, t}^{\rm c}$] Active/reactive power consumption (p.u.).
    \item [$P_{i}^{\rm \diamond, min}$]
    Lower bound on $p_{i,t}^\diamond$; $\diamond = \{\text{DG}, (\text{BESS}), \text{DR}\}$ (p.u.).
    \item [$P_{i}^{\rm \diamond, max}$]
    Upper bound on $p_{i,t}^\diamond$; $\diamond = \{\text{DG}, (\text{BESS}), \text{DR}\}$ (p.u.).
    \item [$\Omega_{0, t}^{\text{b}}$] Cost of asset failure at substation (\$).
    \item [$\Omega_{i, t}^{\text{b}}$] Cost of asset failure (\$).
    \item [$Re_{0}^{\text{b}}(t)$] \textit{Reliability} of substation up to time $t$.
    \item [$Re_{i}^{\diamond}(t)$] \textit{Reliability} of $\diamond$ up to time $t$, $\diamond = \{\text{b}, \text{l} \}$.
    \item [$F_{0}^{\text{b}}(t)$] \textit{Unreliability} of substation up to time $t$.
    \item [$F_{i}^{\diamond}(t)$] \textit{Unreliability} of $\diamond$ up to time $t$, $\diamond = \{ \text{b}, \text{l} \}$.
    \item [$\rm{Pr}_{0,\textit{t}}^{\text{b}}$] \textit{Interval unreliability} of substation.
    \item [$\rm{Pr}_{\textit{i,t}}^{\diamond}$] \textit{Interval unreliability} of $\diamond$, $\diamond = \{\text{b}, \text{l} \}$.
    \item [$R_{i}/X_{i}$] Resistance/reactance of line $i$ (p.u.).
    \item [$S_{i}$] Apparent power flow limit on line $i$ (p.u.).
    \item [$V_{0}^{\text{con}}$] Constant voltage at substation (p.u.).
    \item [$V_{i}^{\diamond}$] Bus voltage limits; $\diamond = \{\text{min}, \text{max} \}$ (p.u.).
    \item [$w_{0}$] Per unit cost of load-shedding at substation (\$).
    \item [$w_{i}^{\rm c}$] Per unit cost of power consumption (\$).
    \item [$w_{i}^{\diamond}$] Per unit cost of $\diamond$; $\diamond = \{\text{DG},\ (\text{B,c}),\ (\text{B,d}),\ \text{DR}\}$ (\$).
    \item [$\bm{w}_{i}$] Per unit cost of de-energizing consumers, DGs, BESSs defined as a column vector $\bm{w}_{i} = (w_{i}^{\rm DG}, w_{i}^{\rm B, c}, w_{i}^{\rm B, d}, w_{i}^{\rm DR})$.
    \item [$\Delta P_{i, \text{AT}_t}^{\diamond, \text{corr}}$] Maximum capacity correction factor for $\diamond$ due to ambient temperature $T$ at time $t$, $\diamond = \{\text{DG}, \text{B}, \text{l} \}$.
    \item [$\lambda_{\text{b}, 0}, \beta_{\kappa, 0}^{\text{b}}$] Coefficients of generalized logistic regression model for modelling the interval unreliability of substation, $\kappa = 0, 1, 2$.
    \item [$\lambda_{\diamond, i}, \beta_{\kappa, i}^{\diamond}$] Coefficients of generalized logistic regression model for modelling the interval unreliability of buses and lines, $\kappa = 0, 1, 2$, $\diamond = \{\text{b}, \text{l} \}$.
    \item [$\delta_i$] BESS Self-discharge rate (\%).
    \item [$\eta_i^{\rm B, c}/\eta_i^{\rm B, d}$] BESS Charging/discharging efficiency (\%).
    \item [${\rm {SOC}}_{i}^{\rm B, min}$] Minimum BESS state-of-charge (\%).
    \item [${\rm {SOC}}_{i}^{\rm B, max}$] Maximum BESS state-of-charge (\%).
\end{ldescription}

\newpage
\section{Model} \label{sec:3_math_model}

We model an active distribution system that is operated by a central distribution system operator and that hosts various controllable resources, i.e., DG, BESS, DR.
Section~\ref{sec:consts} introduces this model.
We use the established \textit{DistFlow} (or \textit{Branch Flow}) AC optimal power flow formulation for radial systems to model power flow and voltages (\cite{baran1989optimal,low2014convex}).
Building on the results in \cite{moein2016}, the model also captures dynamic resource capacity ratings that depend on ambient temperature. 
As a result, the model adjusts the maximum capacity of DGs and lines by introducing temperature-dependent coefficients. 

We then formulate two objectives of the distribution system operator using the model. First, in the \textit{cost-only model} (CM) the distribution system operator aims to control all resources such that the cost of supplying system demand is minimized and distributed generation can be utilized as much as possible (Section~\ref{sec:3_cm}).
Section~\ref{sec:3_crem} then extends the CM to additionally account for the cost of expected energy-not-served (EENS) in case of the failure of a bus or line (e.g., due to equipment failure or overload triggering protection equipment.)
The objective of this \textit{cost-and-reliability model} (CRM), is non-convex because the probability of component failure failure depends on the system control setpoints (which define the power flows, local generation, and voltages that, in turn, affect failure rate).
We propose a computationally efficient iterative solution approach using sequential convex optimization to solve the problem in Section~\ref{sec:3_comp}.

\subsection{Active distribution grid model}
\label{sec:consts}

We consider a radial distribution system with a set $\mathcal{N}$ of buses indexed by $i=0,1,...,N$. 
Index $i=0$ corresponds to the substation at the root of the network.  
Due to the radial network configuration, each bus $i$ has exactly one ancestor (upstream) bus $\mathcal{A}_i$ with $\left| \mathcal{A}_i \right| = 1$. 
This allows to uniquely identify each line with the index $i\in\mathcal{N}^+ = \mathcal{N}\setminus\{0\}$ as the line connecting bus $i$ to bus $\mathcal{A}_i$. 
The number of children (downstream) buses of bus $i$, collected in set $\mathcal{S}c_i$, is  $\left| \mathcal{S}c_i \right| \geq 1$ for internal buses $i$ and $\left| \mathcal{S}c_i \right| = 0$ for leaf buses. 
Consequently, there is a unique path between each bus and the substation.
Each bus can host distributed resources, i.e., DG (at all buses in $\mathcal{N}_{\rm DG}$), DR (at all buses in $\mathcal{N}_{\rm DR}$), or BESS (at all buses in $\mathcal{N}_{\rm BESS}$). 
As per the the \textit{DistFlow} model, the voltage at each bus is captured by the voltage magnitude squared $v_{i,t}$ and current flowing over each line by the current magnitude squared $l_{i,t}$.
We model active and reactive power losses due to the shunt conductance at bus~$i$ as $G_{i}v_{i,t}$ and $B_{i}v_{i,t}$, and losses due to line resistance and reactance as $R_{i}l_{i,t}$ and $X_{i}l_{i,t}$.
Fig.~\ref{fig:radial_network} illustrates the central elements of our notation and
the complete model is given by the following set of constraints:
\begin{figure}
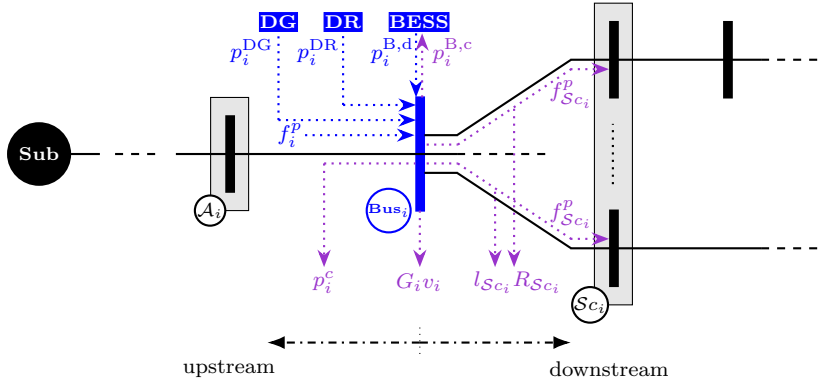

    \include{tikzfig_radial_network}
\caption{Schematic illustration of central \textit{DistFlow} notations for active power in radial networks.}
\label{fig:radial_network}
\end{figure}

\allowdisplaybreaks
\begin{align}
    & \forall i \in \mathcal{N}^+, \forall t \in \mathcal{T}: \left( f_{i, t}^p \right)^2 + \left( f_{i, t}^q \right)^2 \leq \Delta P_{i, \text{AT}_t}^{\text{l,corr}} S_{i}^2 \label{eq:const_line_1}\\
    & \forall i \in \mathcal{N}^+, \forall t \in \mathcal{T}: \left( f_{i, t}^p + l_{i, t} R_{i} \right)^2 + \left( f_{i, t}^q + l_{i, t} X_{i} \right)^2 \leq \Delta P_{i, \text{AT}_t}^{\text{l,corr}} S_{i}^2 \label{eq:const_conic_1}\\
    & \forall i \in \mathcal{N}^+, \forall t \in \mathcal{T}: v_{i, t} + 2\left(R_{i}f_{i, t}^p + X_{i}f_{i, t}^q \right) + l_{i, t}\left( R_{i}^2 + X_{i}^2 \right) = v_{\mathcal{A}_i, t} \label{eq:const_conic_2}\\
    & \forall i \in \mathcal{N}^+, \forall t \in \mathcal{T}: \frac{\left( f_{i, t}^p \right)^2 + \left( f_{i, t}^q \right)^2}{l_{i, t}} \leq v_{i, t} \label{eq:const_conic_3}\\
    & \forall i \in \{0\}, \forall t \in \mathcal{T}: p_{i, t} = \sum_{j \in \mathcal{S}c_0}\left( f_{j, t}^p + l_{j, t}R_{j} \right) + G_{i}v_{i, t} \label{eq:const_sub_bal_real}\\
    & \forall i \in \{0\}, \forall t \in \mathcal{T}: q_{i, t} = \sum_{j \in \mathcal{S}c_0}\left( f_{j, t}^q + l_{j, t}X_{j} \right) + B_{i}v_{i, t} \label{eq:const_sub_bal_reactive}\\
    & i \in \mathcal{N}^+, \forall t \in \mathcal{T}: \label{eq:const_bus_bal_real}\\
    & f_{i, t}^p = \sum_{j \in \mathcal{S}c_i}\left( f_{j, t}^p + l_{j, t} R_{j} \right) + p_{i, t}^c - p_{i, t}^{\rm DG} + p_{i, t}^{\rm B, c} - p_{i, t}^{\rm B, d} - p_{i, t}^{\rm DR} + G_{i}v_{i, t} \notag \\
    & i \in \mathcal{N}^+, \forall t \in \mathcal{T}: \label{eq:const_bus_bal_reactive}\\
    & f_{i, t}^q = \sum_{j \in \mathcal{S}c_i}\left( f_{j, t}^q + l_{j, t} X_{j} \right) + q_{i, t}^c - q_{i, t}^{\rm DG} - q_{i, t}^{\rm DR} + B_{i}v_{i, t} \notag \\
    & \forall i \in \{0\}, \forall t \in \mathcal{T}: v_{i, t} = V_{i}^{\text{con}}\label{eq:const_sub_voltage_equal} \\
    & \forall i \in \mathcal{N}^{+}, \forall t \in \mathcal{T}: V_{i}^{\text{min}} \leq v_{i, t} \leq V_{i}^{\text{max}} \label{eq:const_bus_voltage_lim} \\
    & \forall i \in \{0\}, \forall t \in \mathcal{T}: p_{i, t} \leq P_{i}^{\rm max} \label{eq:const_sub_lim_real}\\
    & \forall i \in \{0\}, \forall t \in \mathcal{T}: q_{i, t} \leq Q_{i}^{\rm max} \label{eq:const_sub_lim_reactive}\\
    & \forall i \in \mathcal{N}_{\text{DG}}, \forall t \in \mathcal{T}: P_{i}^{\text{DG,min}} \leq p_{i, t}^{\rm DG} \leq P_{i}^{\text{DG,max}} \Delta P_{i, \text{AT}_t}^{\text{DG,corr}} \label{eq:const_dg_lim_real}\\
    & \forall i \in \mathcal{N}_{\text{DG}}, \forall t \in \mathcal{T}: Q_{i}^{\text{DG,min}} \leq q_{i, t}^{\rm DG} \leq Q_{i}^{\text{DG,max}} \label{eq:const_dg_lim_reactive}\\   
    & \forall i \in \mathcal{N}_{\text{DR}}, \forall t \in \mathcal{T}: P_{i}^{\text{DR,min}} u_{i, t}^{\text{DR}} \leq p_{i, t}^{\rm DR} \leq P_{i}^{\text{DR,max}} u_{i, t}^{\rm DR} \label{eq:const_dr_lim_real}\\
    & \forall i \in \mathcal{N}_{\text{DR}}, \forall t \in \mathcal{T}: Q_{i}^{\text{DR,min}} u_{i, t}^{\text{DR}} \leq q_{i, t}^{\rm DR} \leq Q_{i}^{\text{DR,max}} u_{i, t}^{\rm DR} \label{eq:const_dr_lim_reactive}\\
    & \forall i \in \mathcal{N}_{\text{BESS}}, \forall t \in \mathcal{T}: \text{SOC}_{i, t}^{\rm B} = \label{eq:const_soc_bal} \\
    & \left( 1 - \delta_i \right) \text{SOC}_{i, t - 1}^{\rm B} + \left( \frac{p_{i, t}^{\rm B, c}\eta_i^{\rm B, c}}{P_{i}^{\rm B, max} \Delta P_{i, \text{AT}_t}^{\text{B,corr}}} \right) - \left( \frac{p_{i, t}^{\rm B, d}}{\eta_i^{\rm B, d}P_{i}^{\rm B,max} \Delta P_{i, \text{AT}_t}^{\text{B,corr}}} \right) \notag \\
    & \forall i \in \mathcal{N}_{\rm BESS}, \forall t \in \mathcal{T}: \text{SOC}_{i}^{\rm B, min} \leq \text{SOC}_{i, t}^{\rm B} \leq \text{SOC}_{i}^{\rm B, max} \label{eq:const_soc_lim}\\
    & \forall i \in \mathcal{N}_{\rm BESS}, \forall t \in \mathcal{T}: P_{i}^{\rm B, min} u_{i, t}^{\rm B} \leq p_{i, t}^{\rm B, c} \leq P_{i}^{\rm B, max} u_{i, t}^{\rm B} \label{eq:const_battery_c_lim_real} \\
    & \forall i \in \mathcal{N}_{\rm BESS}, \forall t \in \mathcal{T}: P_{i}^{\rm B, min} \left( 1 - u_{i, t}^{\rm B} \right) \leq p_{i, t}^{\rm B, d} \leq P_{i}^{\rm B, max} \left( 1 - u_{i, t}^{\rm B} \right)\label{eq:const_battery_d_lim_real} 
\end{align}
\allowdisplaybreaks[0]%
Constraints \eqref{eq:const_line_1}, \eqref{eq:const_conic_1} limit the real and reactive power flow ($f^p_{i,t}$ and $f^q_{i,t}$) between buses, ensuring they do not exceed the maximum allowable line capacity, which is adjusted for the ambient temperature by the maximum capacity correction factor $(\Delta P_{i, \text{AT}_t}^{\text{l, max}})$, as suggested in \cite{moein2016}.
Constraints \eqref{eq:const_conic_2}, \eqref{eq:const_conic_3} model the bus voltages and line currents.
Constraints \eqref{eq:const_sub_bal_real}, \eqref{eq:const_sub_bal_reactive} are the real and reactive power flow balance equations for the substation ($i=0$) and constraints \eqref{eq:const_bus_bal_real}, \eqref{eq:const_bus_bal_reactive} are the real and reactive power flow balance equations for all other buses $i\in\mathcal{N}^+$.
Constraints \eqref{eq:const_sub_voltage_equal}, \eqref{eq:const_bus_voltage_lim} establish the bounds for the squared voltage magnitudes across all buses.
Constraints \eqref{eq:const_sub_lim_real}, \eqref{eq:const_sub_lim_reactive} limit the active and reactive power throughput at the substation. 
Constraints \eqref{eq:const_dg_lim_real},  \eqref{eq:const_dg_lim_reactive} set the lower and upper bounds for the real and reactive power of the DG at bus $i$. 
For real power generation, the maximum power correction factor $(\Delta P_{i, \text{AT}_t}^{\text{DG, corr}})$ for each DG at $i$ 
due to ambient temperature is added (\cite{moein2016}). 
Constraints \eqref{eq:const_dr_lim_real}, \eqref{eq:const_dr_lim_reactive} set the lower and upper bounds for the real and reactive power of the DR resource.
Constraint \eqref{eq:const_soc_bal} models BESS state-of-charge (SOC).
The SOC at time step $t$ is determined by several factors, including the previous SOC at time step $t - 1$, adjusted by self-discharge rate $\delta_i$. 
Charging power $p_{i,t}^{\rm B,c}$ is weighted by charging efficiency $\eta_{i}^{\rm B, c}$ and normalized by the product of the battery's maximum power $P_{i}^{\rm B, max}$ and a power correction factor $(\Delta P_{i, \text{AT}_t}^{\text{B,corr}})$
that depends on temperature $\text{AT}_t$ (\cite{moein2016}). 
Discharging power $p_{i,t}^{\text{B,d}}$ is also adjusted by discharging efficiency $\eta_i^{\rm B, d}$ and normalized by $P_{i}^{\rm B, max}$ and temperature-dependent power correction factor $\Delta P_{i, \text{AT}_t}^{\text{B,corr}}$. 
Constraint \eqref{eq:const_soc_lim} sets the lower and upper bounds for the SOC and constraints \eqref{eq:const_battery_c_lim_real} \eqref{eq:const_battery_d_lim_real} limit the BESS charging and discharging power. 
\subsection{Cost-only objective (CM)}\label{sec:3_cm}

For concise notation, we collect all decision variables from \eqref{eq:const_line_1}--\eqref{eq:const_battery_d_lim_real} in the vector~$\bm{v}$.
The objective of CM is to minimize the distribution system operational cost: 
\begin{equation}
    Obj_{\rm  CM}: \min_{\bm{v}} \left[ \sum_{t \in \mathcal{T}} \left( c_{0, t} p_{0, t} + \sum_{i \in \mathcal{N}^+} \bm{c}^{T}_{i, t} \bm{p}_{i, t} \right) \right] = \min_{\bm{v}} \left[ \sum_{t \in \mathcal{T}}\mathcal{C}_t^{\text{op}}({\bm{v}}) \right]. \label{eq:obj_baseline}
\end{equation}
The first term in \eqref{eq:obj_baseline} captures the cost of buying energy at the substation from the transmission grid. 
Hence, cost $c_{0,t}$ are, e.g., the wholesale electricity price at time $t$. 
The second term in \eqref{eq:obj_baseline} reflects the cost of utilizing distributed resources (DG, DR, BESS charging and discharging).

\subsection{Cost-and-reliability model (CRM)}\label{sec:3_crem}

The objective of CRM is to minimize the operational costs of the power system ($\mathcal{C}^{\text{op}}$ as defined in \eqref{eq:obj_baseline} above) and \textit{in addition} to CM consider the expected cost of energy not served $\mathcal{C}^{\text{eens}}$, which depend on system operations.
\begin{align}
    Obj_{\rm CRM}: \min_{\bm{v}} \left[ \sum_{t \in \mathcal{T}} \mathcal{C}^{\text{op}}_{t}(\bm{v}) + \sum_{t \in \mathcal{T}} \mathcal{C}_{t}^{\text{eens}}(\bm{v}) \right].
\label{eq:CRM_objective}
\end{align}
We define $\mathcal{C}_{t}^{\text{eens}}(\bm{v})$ as:
\begin{equation}
    \mathcal{C}_{t}^{\text{eens}}(\bm{v}) = \left(\Omega_{0, t}^{\rm b}(\bm{v}) \text{Pr}_{0,\textit{t}}^{\rm b}(\bm{v})\right) +\! \sum_{i \in \mathcal{N}^{+}}\! \Omega_{i, t}^{\rm b}(\bm{v}) \left(\!1\! - \!\left(\! 1 - \text{Pr}_{\textit{i,t}}^{\rm b}(\bm{v}) \right)\!\!\!\!\! \prod_{j \in \rm{MCS}_{b_i}} \!\!\!\!\!\left( 1 - \text{Pr}_{\textit{j,t}}^{\rm l}(\bm{v}) \right)\!\!\right), \label{eq:reli_measure}
\end{equation}
where $\rm{Pr}_{\textit{i,t}}^{\rm b}$, $\rm{Pr}_{\textit{i,t}}^{\rm l}$ is the interval unreliability (i.e., the probability that a bus or line will fail in time step $t$) and $\Omega_{i, t}^{\rm b}$ are resulting cost in the event of a failure.

The detailed expression of $\mathcal{C}_{t}^{\text{eens}}(\bm{v})$ is based on some central modeling definitions and assumptions outlined below.
\begin{definition}[Component failure]
    A component (bus or line, including transformers) is considered \emph{failed} if they must be de-energized for any reason. Such reasons may include triggered protection equipment (e.g., due short circuits or ground faults) or actual equipment failure (e.g., mechanical failures, transformer damage).  A component that is not failed is considered \emph{operational}.
\end{definition}

Any bus $i$ is considered operational at time step $t$ if and only if bus $i$ itself has not failed during $t$ and no line in the minimum cut-set $\mathcal{MCS}_{\text{b}_i}$ of the network flow from the substation to the bus has failed.
In a radial network, $\mathcal{MCS}_{\text{b}_i}$ consists of all upstream buses that directly connect the substation to bus $i$. 
See also Fig.~\ref{fig:radial_network_bus_i_status} for illustration.

\begin{assumption}
    The failure events of buses and lines \emph{as a component} are treated as mutually independent across all components of the system. 
\end{assumption}

\begin{figure}[b]
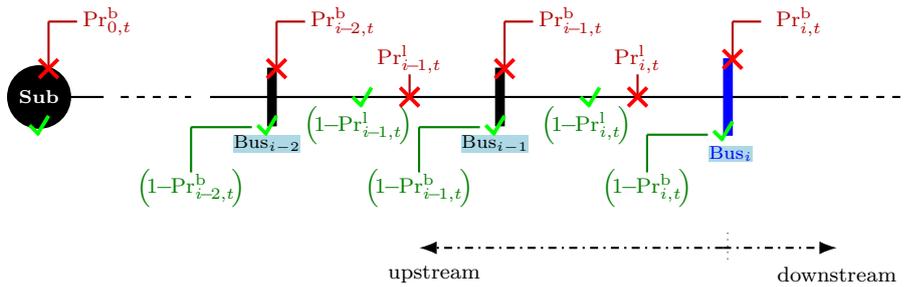

    \include{tikzfig_reliability_model}
\vspace{-2em}    
\caption{Reliability Modeling for Radial Network}
\label{fig:radial_network_reliability_modelling}
\end{figure}

\begin{figure}
    \centering
    
    \begin{subfigure}[b]{0.45\textwidth}
        \centering
        \include{tikzfig_radial_network_working_a}
        \vspace{-2em}    
        \caption{}
    \label{fig:radial_network_bus_working_ex_1}
    \end{subfigure}
    \hfill
    \begin{subfigure}[b]{0.45\textwidth}
        \centering
        \include{tikzfig_radial_network_working_b}
        \vspace{-2em}    
        \caption{}
    \label{fig:radial_network_bus_working_ex_2}
    \end{subfigure}
    \caption{Radial Network Properly Working Up To Bus $i$ Examples: Fig. \ref{fig:radial_network_bus_working_ex_1} shows that Bus $i$ is operational, along with all other buses and lines in the system, while Fig. \ref{fig:radial_network_bus_working_ex_2} demonstrates that Bus $i$ remains operational despite disruptions to some upstream buses.}
    \label{fig:radial_network_bus_working_ex}
    
    \vspace{1cm}
    
    \begin{subfigure}[b]{0.45\textwidth}
        \centering
        \include{tikzfig_radial_network_failure_a}
        \vspace{-2em} 
        \label{fig:bus_not_working_line}
        \caption{}
    \label{fig:radial_network_bus_not_working_ex_1}
    \end{subfigure}
    \hfill
    \begin{subfigure}[b]{0.45\textwidth}
        \centering
    \include{tikzfig_radial_network_failure_b}
    \vspace{-2em}
    \caption{}
\label{fig:radial_network_bus_not_working_ex_2}
    \end{subfigure}
    \caption{Radial Network Failure Towards Bus $i$ Examples: Fig. \ref{fig:radial_network_bus_not_working_ex_1} illustrates the failure of Bus $i$ due to a disruption in the lines connecting the substation to Bus $i$, while Fig. \ref{fig:radial_network_bus_not_working_ex_2} demonstrates the failure of Bus $i$ itself.}
    \label{fig:radial_network_bus_not_working_ex}
    
    \label{fig:radial_network_bus_i_status}
    
\end{figure}

\begin{definition}[Reliability/Unreliability]
The \emph{reliability} of a component at time $t$ is defined as the probability of the component not failing during the time interval $[0, t]$. 
The \emph{unreliability} of buses and lines are defined as the probability that the component will fail \emph{by time $t$},  
\begin{align*}
    F_i^{\text{b}}(t) &= 1 - Re_i^{\text{b}}(t) \quad i  \in \mathcal{N},  \\
    F_j^{\text{l}}(t) &= 1 - Re_j^{\text{l}}(t) \quad j \in \mathcal{N}^+.
\end{align*}
\end{definition}

\begin{definition}[Interval unreliability]
The \textit{interval unreliability} is the probability that a component fails \emph{in} time step $t$, i.e., between the end of $t-1$ and the end of $t$. Formally:
\begin{align*}
    \text{Pr}_i^{\text{b}}(t) &= F_i^{\text{b}}(t) - F_i^{\text{b}}(t-1) = Re_i^{\text{b}}(t-1) - Re_i^{\text{b}}(t),\ i \in \mathcal{N}\\
    \text{Pr}_j^{\text{b}}(t) &= F_j^{\text{b}}(t) - F_j^{\text{b}}(t-1) = Re_j^{\text{b}}(t-1) - Re_j^{\text{b}}(t),\ j \in \mathcal{N}^+.
\end{align*}
\end{definition}
Using these definitions and assumptions, the term $\big( 1 - \rm{Pr}_{\textit{i,t}}^{\rm b} \big)$ in \eqref{eq:reli_measure} is the the probability that a bus $i$ itself is operational (not failed) during  time step $t$.
Term $\left( 1 - \rm{Pr}_{\textit{j,t}}^{\rm l} \right)$ is the probability that upstream line $j$ of bus $i$ operational during time step $t$.
As a result, $\prod_{j \in \rm{MCS}_{b_i}} \left( 1 - \rm{Pr}_{\textit{j,t}}^{\rm l} \right)$ is the probability that all upstream lines that serve bus $i$ are operational during time step $t$.
Finally,  $\left( 1 - \rm{Pr}_{\textit{i,t}}^{\rm b} \right) \times \prod_{j \in \rm{MCS}_{b_i}} \left( 1 - \rm{Pr}_{\textit{j,t}}^{\rm l} \right)$ is the resulting probability that bus $i$ is operational \textit{and} has an operational connection to the substation during time step $t$.
We also refer to Fig.~\ref{fig:radial_network_reliability_modelling} for illustration of the model notation. 

\begin{remark}
We note the we do not explicitly model fault isolation and service restoration via tie line switches (TS). 
While this is a common operational practice in distribution system operations, we opt to not include TS switching events in our model for two main reasons: 
(i) Our model remains useful because even if a component failure can be isolated and service can be restored through TS switching, local service will be interrupted for a brief period of time (\cite{milad2023}), which leads to the potentially costly de-energizing of equipment. Some DG may not be able to immediately and automatically re-connect to the grid, even if service is restored.
(ii) The focus of this paper is the modeling, optimization, and parameter estimation of the context- and decision-aware reliability functions. We consider the additional complexity in modeling and solving the resulting problem to be beyond the scope of the intention of this paper. We discuss future reserach avenues that include TS operation in Section~\ref{sec:conclusion} below.
\end{remark}

The values of $\Omega_{i, t}^{\rm b}$, $\rm{Pr}_{\textit{i,t}}^{\rm b}$, and $\rm{Pr}_{\textit{j,t}}^{\rm l}$, $i \in \mathcal{N}, j \in \mathcal{N}^{+}$, $t \in \mathcal{T}$ in \eqref{eq:reli_measure} depend on the system state.
Failure cost $\Omega_{i, t}^{\rm b}$ are the expected cost resulting from the system's inability to serve load and/or accommodate net-generation from distributed resources during in the event of component failure.
Hence, $\Omega_{i, t}^{\rm b}$ depend on the resources dispatch as: 
\begin{align*}
    \Omega_{0, t}^{\rm b} &= w_{\text{b}_0}  p_{0, t}, \\
    \Omega_{i, t}^{\rm b} &= w_{i}^{\rm c} p_{i, t}^c + \bm{w}_{i}^{T} \bm{p}_{i, t} \nonumber \\
    &= w_{i}^{\rm c} p_{i, t}^c + w_{i}^{\rm DG} p_{i, t}^{\rm DG} + w_{i}^{\rm B, c} p_{i, t}^{\rm B, c} + w_{i}^{\rm B, d} p_{i, t}^{\rm B, d} + w_{i}^{\rm DR} p_{i, t}^{\rm DR}. 
\end{align*}
The failure cost at all other buses $i\in\mathcal{N}^+$ are modeled as a weighted sum of power consumption $p_{i,t}^{\rm c}$ and generation $p_{i,t}^{\diamond},\ \diamond=\{\rm DG,\ (B,c),\ (B,d),\ DR\}$.
Cost $w_i^{\diamond},\ \diamond=\{\rm DG,\ (B,c),\ (B,d),\ DR\}$ can be chosen to reflect value-of-lost-load with nodal resolution (e.g., to model higher-risk loads at buses with vulnerable demographics or critical infrastructure (\cite{eckstrom2022outing}) or model penalty payments for unused generation potential. 

The probability of component failures depend on both ambient and system conditions (\cite{billinton2005consideration,billinton2006distribution,hani2023,luca2022}). For analyzing the failure time of a component when the failure time is dependent on some time-independent covariates, the Cox proportional hazards model (PH) is suitable as detailed in \cite{john2023fifty}. By taking advantage of the fact that our time steps are discrete, it can be shown that, given the probability of a failure event in each time step is small, e.g., less than $0.1$ (\cite{abbott1985logistic}), logistic regression can be adapted for the analysis. This adaptation leads to parameter estimates close to those obtained by the PH model, which we explain in detail in Appendix \ref{appendix:log_reg_cox}.
Give that generalized logistic regression provides a probability measure,
we model this dependency through a logistic function as: 
\begin{equation} \label{eq:logistic_ori}
            \rm{Pr} = \frac{1}{1 + e^{-(\beta_0 + \beta_1 X_1 + \beta_2 X_2)}}, 
\end{equation}
where $\beta_0$ is the intercept, $\beta_1$ and $\beta_2$ are coefficients for the context variables $X_1$ and $X_2$ respectively.
The values of $\beta_0$, $\beta_1$ and $\beta_2$ can be estimated through maximum likelihood estimation (MLE) and Hamiltonian Markov Chain (HMC) (\cite{duane1987hybrid}), detailed in Section \ref{sec:4_case_study}. 
We rewrite \eqref{eq:logistic_ori} as
\begin{equation*}
    \rm{Pr} = \frac{1}{1 + \lambda e^{-(\beta_1 X_1 + \beta_2 X_2)}}, 
\end{equation*}
where $\lambda = e^{-\beta_0}$ can be interpreted as the inverse of the baseline odds. 
Specifically, it represents the ratio of the probability of the component not failing to the probability of the component failing when all covariates are zero.

In our model, we set variables $X_1$ to the net power demand ($\tilde{p}^{\text{b}}$) or current magnitude ($l$) of each component and variable $X_2$ to the the ambient temperature.
We note that this a modeling choice for this paper. The model can be modified to incorporate other covariates to component reliability as required by the modeler or decision-maker.
We assume that $\rm{Pr}_{*}^{*}(t)$ is monotonically increasing with ambient temperature and net power demand ($\tilde{p}^{\text{b}}$) or current magnitude ($l$) of the components. 
The resulting component failure probabilities are:
\begin{itemize}
\item Interval unreliability of substation:
    \begin{equation}
        \rm{Pr}_{0,\textit{t}}^{\rm b} = \frac{1}{1 + \lambda_{\text{b}, 0}  e^{-\left(\beta_{1, 0}^{\rm b} \left|p_{0, t}\right| + \beta_{2, 0}^{\rm b} {\rm AT}_t\right)}}. \label{eq:fail_sub}
    \end{equation}
    \item Interval unreliability of bus $i$, $i \in \mathcal{N}^{+}$:
        \begin{equation}
            \rm{Pr}_{\textit{i,t}}^{\rm b} = \frac{1}{1 + \lambda_{\text{b}, i}  e^{-\left(\beta_{1, i}^{\rm b}  \left|- p_{i, t}^{\rm c} + p_{i, t}^{\rm DG} - p_{i, t}^{\rm B, c} + p_{i, t}^{\rm B, d} + p_{i, t}^{\rm DR}\right| + \beta_{2, i}^{\rm b}  {\rm AT}_t\right)}}. \label{eq:fail_bus}
        \end{equation}
    \item Interval unreliability of line $i$, $i \in \mathcal{N}^{+}$:
        \begin{equation}
            \rm{Pr}_{\textit{i,t}}^{\rm l} = \frac{1}{1 + \lambda_{\text{l}, i}  e^{-\left(\beta_{1, i}^{\rm l} l_{i, t} + \beta_{2, i}^{\rm l}  {\rm AT}_t\right)}}. \label{eq:fail_line}
        \end{equation}
\end{itemize}

\subsection{Computational approach}\label{sec:3_comp}

The CRM objective is non-linear and non-convex due to the logistic functions in \eqref{eq:reli_measure}. We adopt sequential convex programming (SCP), a local optimization method to handle the non-convex problems (\cite{schittkowski1995sequential}). It initially addresses the convex parts of the problem precisely, then approximates the non-convex parts using convex functions that are sufficiently accurate within a local region. The process is repeated until the convergence conditions is met. To solve the CRM, we use the solution from the CM as an initial guess. The convex portion of the problem is then obtained by linearizing the objective of the CRM, as illustrated in Fig.~\ref{fig:scp_iterative_approach}. 

\begin{figure}
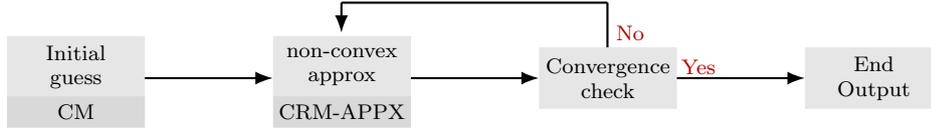

    \include{tikzfig_scp_iterative_approach}
\caption{SCP Iterative Approach}
\label{fig:scp_iterative_approach}
\end{figure}

In the following sections we provide details for the sequential convex programming approach. 
Section~\ref{sec:3_comp_abs} first discusses handling the absolute value functions in \eqref{eq:fail_sub}, \eqref{eq:fail_bus}.
The discussed approach also guarantees that the objective function remains differentiable throughout the open domain of the variables.
We use this property as discussed in Section~\ref{sec:3_comp_lin} to create a linearized approximation of the non-convex function.
Section \ref{sec:3_comp_alg} then summarizes our procedure to numerically solve CRM.

\subsubsection{Handling absolute value functions}\label{sec:3_comp_abs}
In CRM, during each time step $t$, function $\mathcal{C}_{t}^{\text{eens}}$ includes $\rm{Pr}_{i, t}^{\rm b},\ i \in \mathcal{N}$, which contains absolute value functions as per \eqref{eq:fail_sub}, \eqref{eq:fail_bus}. 
Specifically, these values are the absolute net power demand of the substation at time $t$ ($| p_{0, t} |$) and the absolute net power demand of bus $i\in\mathcal{N}^+$ at time $t$ ($|\! - p_{i, t}^{\rm c} + p_{i, t}^{\rm DG} - p_{i, t}^{\rm B, c} + p_{i, t}^{\rm B, d} + p_{i, t}^{\rm DR} |$), respectively. 
We introduce auxiliary variables for these values defined as:
\begin{align} 
    \tilde{p}_{0, t}^{\rm{b}} &= \left| p_{0, t} \right|, \label{eq:addition_var_sub}\\ 
    \tilde{p}_{i, t}^{\rm{b}} &= \left| - p_{i, t}^{\rm c} + p_{i, t}^{\rm DG} - p_{i, t}^{\rm B, c} + p_{i, t}^{\rm B, d} + p_{i, t}^{\rm DR} \right|. \label{eq:addition_var_bu}
\end{align}
Using \eqref{eq:addition_var_sub}, \eqref{eq:addition_var_bu} we rewrite \eqref{eq:fail_sub}, \eqref{eq:fail_bus} as: 
\begin{equation*} 
    {\tilde{\rm{Pr}}_{\textit{i,t}}^{\rm{b}}} = \frac{1}{1 + \lambda_{\text{b}, i} e^{-\left(\beta_{1, i}^{\rm b} \tilde{p}_{i, t}^{\rm{b}} + \beta_{2, i}^{\rm b} {\rm AT}_t\right)}},
\end{equation*}
where $i \in \mathcal{N}$, and $\mathcal{C}_{t}^{\text{eens}}(\bm{v})$ as 
\begin{equation*}
    \tilde{\mathcal{C}}_{t}^{\text{eens}}(\tilde{\bm{v}}) = \left(\Omega_{0, t}^{\rm b}(\tilde{\bm{v}}) \tilde{\text{Pr}}_{0,\textit{t}}^{\rm b}(\tilde{\bm{v}})\right) +\! \sum_{i \in \mathcal{N}^{+}}\! \Omega_{i, t}^{\rm b} \left(\!1\! - \!\left(\! 1 - \tilde{\text{Pr}}_{\textit{i,t}}^{\rm b}(\tilde{\bm{v}}) \right)\!\!\!\!\! \prod_{j \in \rm{MCS}_{b_i}} \!\!\!\!\!\left( 1 - \text{Pr}_{\textit{j,t}}^{\rm l}(\tilde{\bm{v}} \right) \!\!\right).
\end{equation*}
Because $\mathcal{C}_{t}^{\text{eens}}$ is monotonically increasing with respect to $\tilde{p}_{i, t}^{\rm{b}}$, for $i \in \mathcal{N}$ and CRM is a minimization problem,
we can replace equality constraints \eqref{eq:addition_var_sub}, \eqref{eq:addition_var_bu} with their epigraph formulation as:
\begin{align} 
    \tilde{p}_{0, t}^{\rm{b}} &\geq p_{0, t}, \label{eq:addition_var_sub_const_p}\\
    \tilde{p}_{0, t}^{\rm{b}} &\geq - p_{0, t},  \label{eq:addition_var_sub_const_n}\\
    \tilde{p}_{i, t}^{\rm{b}} &\geq - p_{i, t}^{\rm c} + p_{i, t}^{\rm DG} - p_{i, t}^{\rm B, c} + p_{i, t}^{\rm B, d} + p_{i, t}^{\rm DR},  \label{eq:addition_var_bu_const_p}\\
    -\tilde{p}_{i, t}^{\rm{b}} &\geq p_{i, t}^{\rm c} - p_{i, t}^{\rm DG} + p_{i, t}^{\rm B, c} - p_{i, t}^{\rm B, d} - p_{i, t}^{\rm DR}. \label{eq:addition_var_bu_const_n}
\end{align}

\subsubsection{Linearization}\label{sec:3_comp_lin}

\begin{claim}\label{claim:non_convex}
    The expression $\mathcal{C}^{\text{eens}}(\bm{v})$ is non-convex (non-concave) in $\bm{v}$.
\end{claim}
\noindent
We provide a formal discussion of Claim \ref{claim:non_convex} in Appendix \ref{appendix:proof_non_convex}. 
To efficiently solve CRM, we employ a sequential approach that utilizes a linearization of $\tilde{\mathcal{C}}_{t}^{\text{eens}}(\tilde{\bm{v}})$ which we define as CRM-APPX:
\begin{equation*}
    \text{Obj}_{\text{CRM-APPX}}: \min \left[ \sum_{t \in \mathcal{T}}\mathcal{C}(t) + \sum_{t \in \mathcal{T}} \mathcal{L}\big[\tilde{\mathcal{C}}_{t}^{\text{eens}}(\tilde{\bm{v}})|\tilde{\bm{v}}^{\star}\big] \right], 
\end{equation*}
where $\mathcal{L}\left[*\right|\bm{v}]$ indicates the linearization around $\bm{v}$.
Specifically, we use multivariate first-order Taylor expansion.
For a given expansion point $\tilde{\bm{v}}_t^{\star} = \left(\tilde{p}_{0, t}^{\rm{b}^{\star}}, \tilde{\bm{v}}_{i, t}^{\star}, i \in \mathcal{N}^{+} \right)$, where
\begin{equation*}
    \tilde{\bm{v}}_{i, t}^{\star} = \left( {p_{i, t}^{\rm DG}}^{\star}, {p_{i, t}^{\rm B, c}}^{\star}, {p_{i, t}^{\rm B, d}}^{\star}, {p_{i, t}^{\rm DR}}^{\star}, \tilde{p}_{i, t}^{\rm{b}^{\star}}, l^{\star}_{i, t} \right) = ({\tilde{v}_{i, t}^{j^{\star}}}, j = 1, 2, \cdots, 6).
\end{equation*}
we can write
$\mathcal{L}\left[ \tilde{\mathcal{C}}_{t}^{\text{eens}}(\tilde{\bm{v}})|\tilde{\bm{v}}_t^{\star}\right]$ at point $\tilde{\bm{v}}_t^{\star}$ as:
\begin{align*} 
    &\tilde{\mathcal{C}}_{t}^{\text{eens}}(\tilde{\bm{v}})|\tilde{\bm{v}}_t^{\star} \approx \mathcal{L}\left[\tilde{\mathcal{C}}_{t}^{\text{eens}}(\tilde{\bm{v}})|\tilde{\bm{v}}^{\star}_t\right] = \tilde{\mathcal{C}}_{t}^{\text{eens}}\left( \tilde{\bm{v}}_t^{\star} \right)\notag \\
    +& \left[ \sum_{i = 1}^{i = N}\sum_{j = 1}^{j = 6} \frac{\partial \tilde{\mathcal{C}}_{t}^{\text{eens}}}{\partial \tilde{v}_{i,t}^{j}}\left( \tilde{\bm{v}}^{\star}_t \right) \left( \tilde{v}_{i,t}^{j} - \tilde{v}_{i, t}^{j^{\star}} \right) \right] + \left[ \frac{\partial {\tilde{\mathcal{C}}_{t}^{\text{eens}}}}{\partial \tilde{p}_{0, t}^{\rm{b}}}\left(\tilde{\bm{v}}^{\star}_t\right) \left(\tilde{p}_{0, t}^{\rm{b}} - \tilde{p}_{0, t}^{\rm{b}^{\star}}\right) \right],  
\end{align*}
where
\begin{equation*}
    \tilde{\bm{v}}_{i, t} = \left( p_{i, t}^{\rm DG}, p_{i, t}^{\rm B, c}, p_{i, t}^{\rm B, d}, p_{i, t}^{\rm DR}, \tilde{p}_{i, t}^{\rm{b}}, l_{i, t} \right) = (\tilde{v}_{i, t}^{j}, j = 1, 2, \cdots, 6). 
\end{equation*}

\noindent The following claim establishes that the linearization of the objective function of the CRM remains consistent regardless of whether we expand around the point $\tilde{\bm{v}}^*$ with respect to $\bm{v}$ or $\tilde{\bm{v}}$.
\begin{claim}\label{claim:add_var_equal}
    Given ${\tilde{p}_{0, t}^{\rm{b}}} = \left|p_{0, t} \right|$ and $\tilde{p}_{i, t}^{\rm{b}} = \left| - p_{i, t}^{\rm c} + p_{i, t}^{\rm DG} - p_{i, t}^{\rm B, c} + p_{i, t}^{\rm B, d} + p_{i, t}^{\rm DR} \right|$:
    \begin{equation*}
        \mathcal{L}\left[\mathcal{C}_{t}^{\text{eens}}(\bm{v})|\bm{v}_t^{\star}\right] = \mathcal{L}\left[ \tilde{\mathcal{C}}_{t}^{\text{eens}}(\tilde{\bm{v}})|\tilde{\bm{v}}_t^{\star}\right]. 
    \end{equation*}
\end{claim}

\begin{proof}
    See Appendix~\ref{appendix:proof_add_var_equal}
\end{proof}

\subsubsection{Iterative Approach} \label{sec:3_comp_alg}

Once $\rm{Obj}_{\rm CRM}$ is linearized to $\rm{Obj}_{\rm CRM-APPX}$, we use an iterative approach to refine the solution. The initial solution from the CM serves as a starting point. To handle the convex portion, the objective function of CRM is linearized at the current solution point. 
The process is repeated until convergence.
To control the step size in each iteration, we add a regularization $\text{Reg}(\tilde{\bm{v}})$ term to  $\rm{Obj}_{\rm CRM-APPX}$ and we denote the resulting objective as:
\begin{equation*}
    \text{Obj}_{\text{CRM-ITE}}: \min_{\tilde{\bm{v}}} \left[ \sum_{t \in \mathcal{T}}\mathcal{C}_t(\tilde{\bm{v}}) + \sum_{t \in \mathcal{T}} \mathcal{L}\left[\mathcal{C}_{t}^{\text{eens}}(\tilde{\bm{v}})\right] + \text{Reg}(\tilde{\bm{v}})\right].
\end{equation*}
Regularization $\text{Reg}(\tilde{\bm{v}})$ penalizes large changes in decision variables $\tilde{\bm{v}}$ between iterations and, as a result,
promotes gradual and controlled variable updates between iterations.
We define $\text{Reg}(\tilde{\bm{v}})$ as $\phi(k)||\tilde{\bm{v}}^* - \tilde{\bm{v}}||_2^2$, 
where $\phi(k)$ is a monotonically increasing function w.r.t. the $k^{th}$ iteration. 

Algorithm~\ref{alg:ite_approach} itemizes the steps of our computational approach.
We introduce three convergence criteria, where $k$ indicates the $k^{th}$ iteration:
\begin{enumerate}
    \item Total variable difference: Terminate if the total sum of the differences in variable values between the current iteration and the previous iteration is below a given threshold:
        \begin{equation} \label{eq:conv_var_diff}
            \left( \tilde{\bm{v}}[k] - \tilde{\bm{v}}[k-1] \right)^T \left( \tilde{\bm{v}}[k] - \tilde{\bm{v}}[k-1] \right) \le \epsilon_1.
        \end{equation}
    \item Total linearization deviation: Terminate if the total sum of differences between $\mathcal{C}^{\text{eens}}(\tilde{\bm{v}})$ at current iteration and $\mathcal{C}^{\text{eens}}(\tilde{\bm{v}})$ at previous iteration is smaller then a given threshold across all modelled time steps $t$:
        \begin{equation} \label{eq:conv_lin_diff}
            \left| \left( \text{Obj}_{\text{CRM-APPX}}[k] - \text{Obj}_{\text{CM}}[k]\right) - \left( \text{Obj}_{\text{CRM}}[k - 1] - \text{Obj}_{\text{CM}}[k - 1] \right) \right| \le \epsilon_2.
        \end{equation}
        
    \item Relative objective improvement: Terminate if the relative objective improvement between the previous and current iterations 
    are below a given threshold:
        \begin{equation} \label{eq:conv_ratio_diff}
            \frac{\left| \text{Obj}_{\text{CREM-APPX}}[k] - \text{Obj}_{\text{CREM-APPX}}[k-1] \right|}{\text{Obj}_{\text{CREM-APPX}}[k]} \leq \epsilon_3.
        \end{equation}
\end{enumerate}
Algorithm~\ref{alg:ite_approach} terminates when one of these criteria is met. 

\begin{algorithm}
\caption{Iterative Linearization Approach for Solving CRM}
\label{alg:ite_approach}
\begin{algorithmic}[1]
    \Statex \hspace*{-1.6em} \textbf{Input:} CM parameter set, CRM-ITE parameter set, maximum number of iteration: $k^{\text{max}}$
    \Statex \hspace*{-1.6em} \textbf{Output:} Optimized variable values: $\tilde{\bm{v}}_{\text{CRM}}$ and objective value: $\rm{Obj}_{\rm{CRM}}$
    \State Solve CM \Comment{\textit{initialization}}
    \State Obtain $\bm{v}_{\text{CM}}$ and $\text{Obj}_{\rm{CM}}$ 
    \State Obtain initial variable guess: $\tilde{\bm{v}}_{\text{CRM-ITE}}[0] \leftarrow \bm{v}_{\text{CM}}$ 
    \State Obtain initial objective value: 
    \Statex \qquad $\text{Obj}_{\rm{CRM-ITE}}[0] \leftarrow \text{Obj}_{\rm{CM}}$
    \Statex \qquad $\text{Obj}_{\rm{CRM-APPX}}[0] \leftarrow \text{Obj}_{\rm{CM}}$
    \For{$k=1,...,k^{\text{max}}$}  \Comment{\textit{solving iteratively}}
        \State linearize $\mathcal{C}^{\text{eens}}$ around $\tilde{\bm{v}}_{\text{CRM-ITE}}[k - 1]$
        \State Solve CRM-ITE
        \State Obtain $\tilde{\bm{v}}_{\text{CRM-ITE}}[k]$ and $\text{Obj}_{\rm{CRM-ITE}}[k]$
        \State Use $\tilde{\bm{v}}_{\text{CRM-ITE}}[k]$ to calculate the updated objective values:
        \Statex \qquad $\text{Obj}_{\text{CM}}[k]$, $\text{Obj}_{\text{CRM}}[k]$ and $\text{Obj}_{\text{CRM-APPX}}[k]$
        \If {$\text{total variable difference \eqref{eq:conv_var_diff}}  \leq \epsilon_1$ \textbf{or} \\ \hspace*{2.3em} $\text{total linearization deviation \eqref{eq:conv_lin_diff}} \leq \epsilon_2$ \textbf{or} \\
        \hspace*{2.3em} $\text{linearization ratio deviation \eqref{eq:conv_ratio_diff}} \leq \epsilon_3$} \Comment{\textit{check convergence}}
            \State \textbf{break}
        \EndIf
    \EndFor
    \State {\textbf{return}} $\tilde{\bm{v}}_{\text{CRM}}$, $\text{Obj}_{\text{CRM}}$
\end{algorithmic}
\end{algorithm}

\section{Case Study} \label{sec:4_case_study}
We now apply the proposed method to a detailed case study. 
Section~\ref{sec:4_exp_setting} details the case study dataset. 
Section~\ref{sec:4_para_est} discusses the process of estimating the coefficients for the generalized logistic regression model \eqref{eq:fail_sub}, \eqref{eq:fail_bus}, \eqref{eq:fail_line} that computes the failure probability of each component at each time step as a function of the relevant covariates.
Section~\ref{sec:case_study_result_discussion} presents and discusses the numerical results of the case study.

\subsection{Experiment description and parameters}\label{sec:4_exp_setting}
We consider the a one-day operation of an active distribution system. 
We assume that all system components in the power system are all in proper working condition and, for our scope of modeling a single day, disregard the effects of component aging. 
We use a modified version of the radial IEEE $33$-bus test feeder.
We use data provided by pandapower (\cite{thurner2018pandapower}) and manually add 
DERs.
Fig.~\ref{fig:33_bus_power_radial_system} in Appendix~\ref{appendix:sys_layout} illustrates the system in detail.  
We extend the single-period load data from pandapower by scaling it with the load profile of New York City obtained from the New York Independent System Operator (zone ``N.Y.C.'') for the same day (\cite{nyiso}) as:
\begin{equation} \label{eq:power_load_scal}
    p_{i, t}^c = p_{i}^c \times \frac{\rm{NYC~load~at~time~step~}t}{\rm{maximum~NYC~load~across~all~time~steps}}, 
\end{equation}
where $p_{i}^c$ is the load for bus $i$ from the original dataset. 
The same scaling is applied to compute $q_{i, t}^c$. 
Fig.~\ref{fig:temp_power_daily} shows the load and temperature profile for the studied day and we provide a complete parameter list in Appendix \ref{appendix:para_values}.

The substation voltage and voltage limits $V_{0}^{\text{con}}, V_{i}^{\text{min}}, V_{i}^{\text{max}}$ are set to $1.03$ p.u., $0.9$ p.u. and $1.1$ p.u., respectively. 
We take the temperature-dependent limit corrections from \cite{moein2016} as:
\begin{align*}
    &\Delta P_{i, \text{AT}_t}^{\text{l, max}} = -0.77 \text{AT}_t + 119.45 ,\\
    &\Delta P_{i, \text{AT}_t}^{\text{DG, corr}} = -0.47 \text{AT}_t + 111.60 ,\\
    &\Delta P_{i, \text{AT}_t}^{\text{B,corr}} = -0.016 \left( \text{AT}_t \right)^2 + 1.97 \text{AT}_t + 60.75.
\end{align*}

\begin{figure}
    \centering    
    \includegraphics[width=0.65\textwidth]{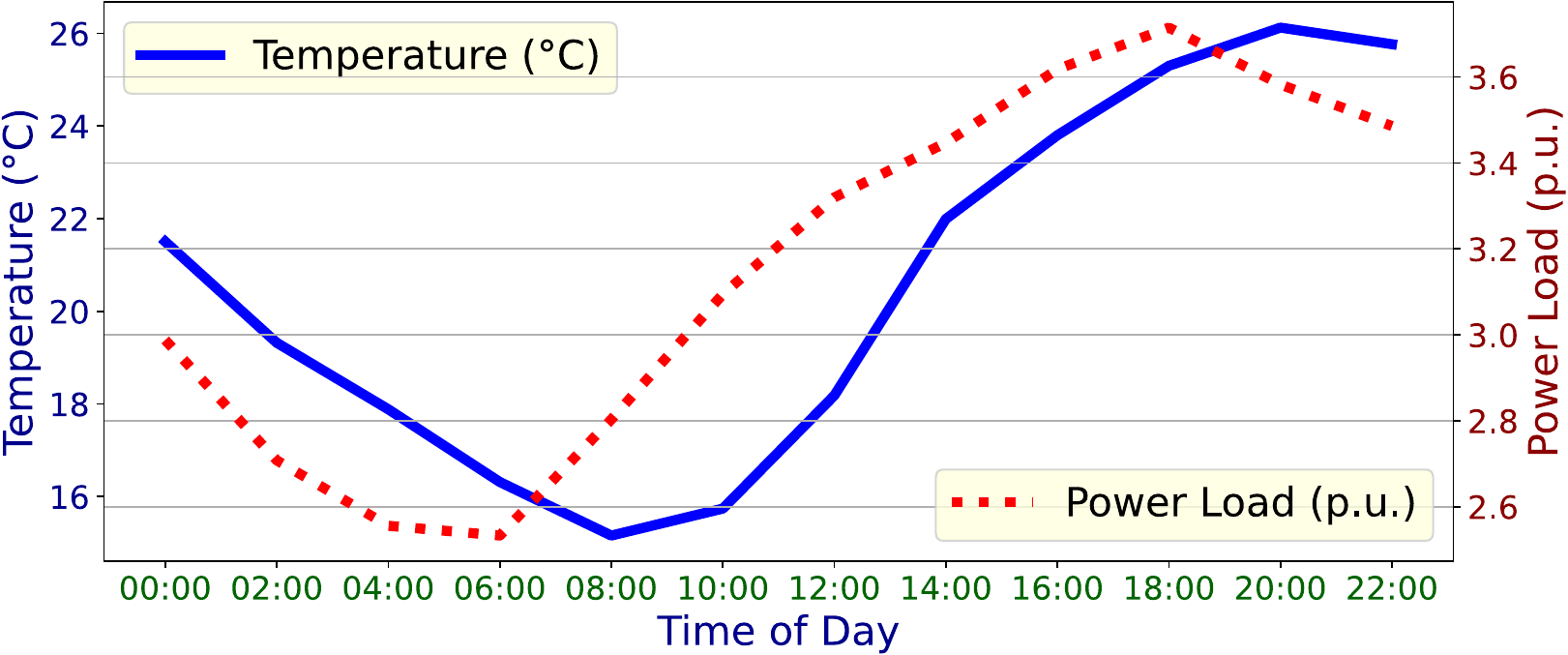}
    \caption{Temperature and load profile of the considered day. Load is total load across all buses: $\sum_{i\in\mathcal{N}}p_{i,t}^{\rm c}\ \forall t$.}
    \label{fig:temp_power_daily}
\end{figure}

\subsection{Regression Model Coefficients Estimation}\label{sec:4_para_est}
This section outlines a blueprint for system operators to estimate the required failure rate models. 
For this case study we use a data-generating simulation that creates database for our regression model. 
While real-world system operators can use real data, they may also benefit from a similar simulation approach if data is scarce.

\subsubsection{Bus and line failure rate}
We simulate the baseline time to failure of all components through an 
exponential distribution with constant failure rate (\cite{richard_epdr_4}).
In particular:

\begin{enumerate}
    \item \textbf{Substation and bus failure rate:}
    For the substation and buses we adopt the failure rate of transformers from \cite{richard_epdr_4} under the following conditions:
    \begin{enumerate}
            \item ambient temperature is $22$ $^{\circ}$C
            \item average winding temperature rise is $65$ $^{\circ}$C
            \item additional hot spot temperature rise is $15$ $^{\circ}$C.
        \end{enumerate}
    The resulting mean time to failure (MTTF) of the each transformer is:
    \begin{equation*}
        \rm{MTTF}_{\rm b} = 10^{\frac{K_1}{\left( 273 + ^{\circ}C  \right)}+K_2} \rm{hours}, 
    \end{equation*}
    where constants $K_1 = 6328.80$ and $K_2 = -11.269$. 
    The resulting time step failure rate (TSFR) can then be calculated as the reciprocal of the MTTF as: (\cite{pham_sre_1}):
    \begin{align*}
        \text{TSFR}_{\rm{b}}= \frac{1}{10^{\frac{K_1}{\left( 273 + ^{\circ}C  \right)}+K_2}} 
        = \frac{1}{2\!\times\! 10^{\frac{6328.80}{\left( 273 + (22+65+15)^{\circ}C  \right)}+(-11.269)}} 
        = 4.93 \times 10^{-6}. 
    \end{align*}
    
    \item \textbf{Line failure rate:} We assume the constant failure rate follows an exponential distribution and adopt the yearly line failure rate as the typical yearly failure rate of overhead lines: $\text{Year}_{\rm{l}} = 0.05$ (\cite{richard_epdr_4, brown2004failure}). Converting into the TSFR:
    \begin{equation*}
        \text{TSFR}_{\rm{l}} = \frac{0.05}{365\text{ days} \times 12\text{ time steps}} = 1.14 \times 10^{-5}.
    \end{equation*}
\end{enumerate}

\subsubsection{Bus and line total failure probability during a single time step}

Because the time scope in the experimental setting is a single day, we do not consider component degradation in this model.
Hence, the failure probability during each time step is independent of the time step for all components.
Without loss of generality, for all components the failure probability during a single time step can be calculated as the probability of failure within the time interval $(0, 1]$. Specifically:
\begin{align}
    &\forall i \in \mathcal{N}: \text{Pr}_{i}^{\text{b}}(1) = F_{i}^{\text{b}}(1) - F_{i}^{\text{b}}(0) = 1 - e^{-\text{TSFR}_{\text{b}}} = 1 - e^{-5.56\times 10^{-6}} = 4.93\times 10^{-6},  \label{eq:case_study_bus_int_unrel}\\
    &\forall i \in \mathcal{N}^{+} \!\!: \text{Pr}_{i}^{\text{l}}(1) = F_{i}^{\text{l}}(1) - F_{i}^{\text{l}}(0) = 1 - e^{-\text{TSFR}_{\text{l}}} = 1 - e^{-1.14\times 10^{-5}} = 1.14\times 10^{-5}. \label{eq:case_study_line_int_unrel}
\end{align}

\subsubsection{Sampling experimental data}
\label{ssec:data_sampling}

Because component failures are relatively rare, i.e., the interval unreliability is small as shown in \eqref{eq:case_study_bus_int_unrel} and \eqref{eq:case_study_line_int_unrel}, a large data sample is necessary to accurately estimate the parameters of the logistic regression model proposed in Section~\ref{sec:3_crem} (\cite{heidelberger1995fast}). 
While such datasets may be available for system operators with multiple years of data records, for our experiment we generate a smaller scale synthetic dataset and employ weighted bootstrapping to effectively enlarge the sample size (\cite{jens1993exchangeably}).
This approach can also be used to expand existing datasets with rare failure events to ensure more accurate parameter estimation.

\paragraph{Ambient temperature data.}
We use $2$-meter temperature data (i.e., air temperature measured at $2$ meters above the surface) provided by Copernicus Climate Change (\cite{copernicus}) for our case study. 
Our analysis is based on data collected for New York City (specifically, between latitudes $40.47^\circ$N to $40.91^\circ$N and longitudes $73.70^\circ$W to $74.25^\circ$W) over a span of $122$ days, covering the period from April $1$, $2024$ to July $31$, $2024$. 
The temperature readings were recorded every $2$ hours, starting from 00:00 and continuing through 22:00 (in total of $12$ time steps) each day.

\paragraph{Power flow and line current magnitude data.}
We use load data from New York City for the same time period (April $1$, $2024$ to July $31$, $2024$), as provided by the New York Independent System Operator (\cite{nyiso}).
By applying the scaling method detailed in \eqref{eq:power_load_scal}, we extend the single-period load data from pandapower to create a dataset of $1464$ data points, corresponding to $122$ days with $12$ time steps per day. 
Running the CM for each of these $1464$ time steps, we generate $1464$ reference data points for net power demand and currents in the system.

\paragraph{Component failure data.}
We generate synthetic failure data for our case study by assigning failure probabilities based on key context variables, reflecting observations made, e.g., in \cite{zhang2019data} and \cite{luca2022} for distribution systems and mirroring generator reliability modeling under adversarial conditions discussed in \cite{miguel2016,mieth2022risk}. 
For buses, the likelihood of failure increases with higher net power demand and elevated ambient temperatures. Similarly, line failures are more probable when the current magnitude and ambient temperature are high. 
(As discussed in Section~\ref{sec:3_crem} above, other context variables can be used in the model.)

Using the 1464 samples of the generated data as described above, for each bus $i$, $i \in \mathcal{N}$ and line $j$, $j \in \mathcal{N}^{+}$, we calculate the mean ($\mu_{\text{b}_i}$ for bus $i$, $\mu_{\text{l}_j}$ for line $j$) and standard deviation ($\sigma_{\text{b}_i}$ for bus $i$, $\sigma_{\text{l}_j}$ for line $j$) of net power demand $\tilde{p}^{\text{b}}$ (for buses) and current square magnitude $l_{i,t}$ (for lines).
Further, we calculate the mean ($\mu_{\text{AT}}$) and standard deviation ($\sigma_{\text{AT}}$) of ambient temperature from the previously generated datasets. 
We then define thresholds for net power demand, current magnitude, and ambient temperature as: $\tau_{\text{b}} = \mu_{\text{b}} + \sigma_{\text{b}}$; $\tau_{\text{l}} = \mu_{\text{l}} + \sigma_{\text{l}}$; $\tau_{\text{AT}} = \mu_{\text{AT}} + \sigma_{\text{AT}}$. 
Assuming that each variable follows a normal distribution, setting the threshold at one standard deviation above the mean corresponds to capturing values that are higher than approximately $84.2\%$ of the distribution.
Next, we count the number of samples, $s_i^{\text{b}}$ for bus $i$ and $s_i^{\text{l}}$ for line $j$, where the net power demand (or current magnitude) exceeds the respective thresholds $\tau_{\text{b}}$ (or $\tau_{\text{l}}$) and the ambient temperature exceeds $\tau_{\text{AT}}$. 
For each bus and line, we designate $80\%$ of these samples as failure cases. Within these failure cases, $90\%$ correspond to instances where both the net power demand (or current magnitude) and ambient temperature are above their respective thresholds, while the remaining $10\%$ correspond to instances where both values are below their thresholds. 

The failure data is represented by binary indicators, where $1$ denotes failure and $0$ denotes normal operations. 
Initially, all failure indicators are set to $0$. 
For each bus and line, starting from the first of the $1464$ data points, we assign a failure indicator of $1$ to the data points where the net power demand exceeds $\tau_{\text{b}}$ (or the current magnitude exceeds $\tau_{\text{l}}$) and the ambient temperature exceeds $\tau_{\text{AT}}$. 
We continue this process until the total number of assigned failures reaches $s_i^{\text{b}} \times 80\% \times 90\%$ for bus $i$ (or $s_i^{\text{l}} \times 80\% \times 90\%$ for lines). 
Subsequently, we restart from the beginning of the $1464$ data points and assign a failure indicator of $1$ to data points where both the net power demand (or current magnitude) and ambient temperature are below their respective thresholds, until $s_i^{\text{b}} \times 80\% \times 10\%$ for bus $i$ (or $s_i^{\text{l}} \times 80\% \times 10\%$ for lines). 
We denote the resulting synthetic dataset, which, to summarize, has a sample size of $1464$, and includes the context variables (covariates) net power demand for buses (current magnitude for lines) and ambient temperature, and response variables indicating failure, as $D_{\text{b}}$ for buses and $D_{\text{l}}$ for lines.

\subsubsection{Generalized logistic regression model parameter estimation}

Traditional logistic regression using MLE is known to be unsuitable for rare events due to its tendency to underestimate event probabilities as shown in \cite{king2001logistic} and  \cite{mccullagh_glm_4}.
To avoid this, we utilize the HMC method, with prior distributions derived from MLE obtained from a logistic regression model, which we apply to the expanded versions of our datasets $D_{\text{b}}$ and $D_{\text{l}}$ to provide a blueprint for real-world applications.
HMC is preferred over traditional Markov Chain Monte Carlo (MCMC) methods for several reasons. Although MCMC is theoretically guaranteed to converge to the target distribution, determining a sharp bound on the convergence rate is often difficult, and terminating the chain after a finite time can lead to errors in approximating the stationary distribution, as discussed in \cite{guanyang2022exact, michael2018conceptual}. HMC accelerates the convergence rate by exploiting the geometry of the target distribution through the use of gradients. By introducing auxiliary momentum variables, HMC ensures that exploration is guided more efficiently toward the typical set (\cite{yuan2024markov, robert2018accelerating}). Furthermore, HMC is more robust when dealing with target distributions that exhibit pathological behavior, which is especially useful in our case since the prior guess may not be perfect (\cite{michael2018conceptual}).
The parameter estimation procedure is consistent for all buses and lines in the network. 
We illustrate the procedure in detail using the bus indexed as $i=1$ as an example.

\paragraph{Weighted bootstrapping and logistic regression model fitting using MLE.}
Given the rare occurrence of bus failure events, the sample size in $D_{\text{b}}$ is insufficient for a reasonable parameter estimation using the logistic regression model with MLE. 
To address this, we bootstrap $10^{7}$ samples from $D_{\text{b}}$. 
Specifically, we use weighted bootstrapping to ensure that the number of failures in the expanded sample is consistent with the failures rates from \eqref{eq:case_study_bus_int_unrel}, \eqref{eq:case_study_line_int_unrel}.
The sampling weight $w_{i}$ for each sample $i$ in $D_{\text{b}}$ is calculated as follows:
\begin{equation}
    w_i = 
        \begin{cases} 
            \text{Pr}_{1}^{\text{b}}/\frac{s_1^{\text{b}}}{1464} & \text{if sample \( i \) is a failure sample,} \\ 
            \big(1-\text{Pr}_{1}^{\text{b}}\big)/\frac{1464 - s_1^{\text{b}}}{1464} & \text{if sample \( i \) is a non-failure sample.} 
        \end{cases}
\label{eq:sampling_weight_defintion}
\end{equation}
After bootstrapping, with a total of $10^{7}$ samples, we fit the model using the \texttt{LogisticRegression} class from the \texttt{scikit-learn} Python library (\cite{scikitlearn}). We denote the estimated parameters as $\hat{\beta}_{0, 1}^{\text{b,prior}}$, $\hat{\beta}_{1, 1}^{\text{b,prior}}$, $\hat{\beta}_{2, 1}^{\text{b,prior}}$ and $\hat{\lambda}_{\text{b}, 1}^{\text{prior}} = e^{-\hat{\beta}_{0, 1}^{\text{b,prior}}}$. 

\paragraph{HMC.} \label{sec:para_hmc}
To estimate the parameters of our logistic regression model, we apply HMC using the \texttt{TensorFlow Probability} library (\cite{tensorflow2015whitepaper}). We model the prior as a multivariate normal distribution with mean $[\hat{\beta}_{0, 1}^{\text{b,prior}}, \hat{\beta}_{1, 1}^{\text{b,prior}}, \hat{\beta}_{2, 1}^{\text{b,prior}}]$ and a standard deviation of 10 for each parameter, reflecting our initial uncertainty. 
According to Bayes' theorem~(\cite{hoff_afcbsm_10}):
\begin{equation*}
    \text{Posterior}(\boldsymbol{\beta} \mid \text{data}) \propto \text{Weighted Likelihood}(\text{data} \mid \boldsymbol{\beta}) \times \text{Prior}(\boldsymbol{\beta}). 
\end{equation*}
Explicitly, the posterior is:
\begin{equation*}
    \text{Pr}(\boldsymbol{\beta} \! \mid \! \text{data}) \! \propto \!\left[\prod_{i=1}^{n_{\text{mc}}} \left(\text{pr}_i^{y_i} (1 - \text{pr}_i)^{1 - y_i}\right)^{w_i} \! \right] \! \times \prod_{j=0}^2 \frac{1}{\sqrt{2\pi\sigma_j^2}} \exp \! \left( \!\! -\frac{(\beta_{j, 1} - \mu_j)^2}{2\sigma_j^2} \! \right),  
\end{equation*}
The log posterior is:
\begin{equation*}
    \log \text{Pr}(\boldsymbol{\beta} \! \mid \! \text{data}) \! \propto \! \sum_{i=1}^{n_{\text{mc}}} w_i \left[ y_i \log(\text{pr}_i) + (1 - y_i) \log(1 - \text{pr}_i) \right] - \sum_{j=0}^2 \left(\frac{(\beta_{j, 1} - \mu_j)^2}{2\sigma_j^2}\right), 
\end{equation*}
where
\begin{enumerate}
    \item[-] $\boldsymbol{\beta} = (\beta_{0, 1}^{\text{b}}, \beta_{1, 1}^{\text{b}}, \beta_{2, 1}^{\text{b}})$.
    \item[-] $y_i$ is the binary value indicates the observed outcome (1 for failure, 0 for non-failure).
    \item[-] $\text{pr}_i = \frac{1}{1 + \exp(-\bm{x}_i^\top \boldsymbol{\beta})}$ is the predicted probability for the $i^{th}$ observation, $\bm{x} = (x_0, x_1, x_2)^{T}$ is the covariates matrix with  $x_0 = 1$, $x_1$ corresponding to power flow and $x_2$ corresponding to ambient temperature.
    \item[-] $w_i$ is the sampling weight applied to the $i^{th}$ observation as per \eqref{eq:sampling_weight_defintion}
    \item[-] $\mu_j$ and $\sigma_j^2$ are the mean and variance of the prior distribution for~$\beta_{j, 1}$.
\end{enumerate}
The HMC algorithm treats the parameter vector $\boldsymbol{\beta}$ as particles in a physical system, where the negative log-posterior serves as the potential energy. 
An auxiliary momentum variable $\bm{p}$ is introduced, typically drawn from a Gaussian distribution $\mathcal{N}(0, M)$, where $M$ is the mass matrix, which is set to the identity matrix in our implementation. The Hamiltonian function $H(\boldsymbol{\beta}, \bm{p})$ combines potential and kinetic energy and is defined to be (\cite{gelman_hmcmc_5, martin_bmacip_11}):
\begin{equation*}
    H(\boldsymbol{\beta}, \bm{p}) = -\log \text{Pr}(\boldsymbol{\beta} \mid \text{data}) + \frac{1}{2} \bm{p}^\top M^{-1} \bm{p}. 
\end{equation*}
HMC uses leapfrog integration to simulate the trajectories of $\boldsymbol{\beta}$ and $\mathbf{p}$ over time (\cite{gelman_hmcmc_5, martin_bmacip_11}). A Metropolis-Hastings correction step ensures that the proposed new state is accepted with the correct probability (\cite{gelman_hmcmc_5, martin_bmacip_11}):
\begin{equation*}
    \text{Accept}(\boldsymbol{\beta}^*) = \min\left(1, \exp\left(H(\boldsymbol{\beta}, \bm{p}) - H(\boldsymbol{\beta}^*, \bm{p}^*)\right)\right). 
\end{equation*}
We apply the Dual Averaging Step Size Adaptation to dynamically adjust the step size for the leapfrog (\cite{tensorflow2015whitepaper}). This adaptation aims for an acceptance rate of 0.65, and values between 0.41 and 0.65 are considered acceptable for stable exploration of the parameter space (\cite{gelman_hmcmc_5, martin_bmacip_11}). The HMC sampling is conducted for $10^5$ iterations, including a burn-in period of $2\times 10^4$ steps. The final estimates of the parameters: $\hat{\beta}_{0, 1}^{\text{b}}$, $\hat{\beta}_{1, 1}^{\text{b}}$ and $\hat{\beta}_{2, 1}^{\text{b}}$ are taken as the mean values from the posterior samples.

\subsection{CRM  Regularization}\label{sec:case_study_ite_alg}
As detailed in Section~\ref{sec:3_comp_alg} above, the iterative optimization process of CRM incorporates predefined convergence and regulation parameters.
For our case study we set the convergence thresholds as $\epsilon_1=10^{-3}$, $\epsilon_2=10^{-1}$, and $\epsilon_3=2 \times 10^{-5}$. The maximum number of iterations is $k^{\text{max}} = 100$. 
Finally, we define the regularization term as:
\begin{equation*}
    \frac{10^{5}\times \left(\left(\mathbf{v}[k]-\mathbf{v}[k-1]\right)^{T}\left(\mathbf{v}[k]-\mathbf{v}[k-1]\right)\right)}{0.85^{k+5}}, 
\end{equation*}
where $k$ indicates the iteration number.

\subsection{Results and discussion} \label{sec:case_study_result_discussion}

\paragraph{Cost.}
Fig.~\ref{fig:obj_crem_cremappx_vs_ite} illustrates the objective values for both CRM-APPX and CRM at each iteration. We show detailed values and computational details in Appendix~\ref{appendix:obj_values}. 
The algorithm required 48 iterations to converge and was terminated by the relative objective improvement ($\epsilon_3$).
At iteration $0$, both models display the result of the initial CM, i.e., resources are dispatched only with regard to operational cost.  
As expected, the initial CM shows lower cost, as it does not account for expected cost from failures (point (i) in Fig.~\ref{fig:obj_crem_cremappx_vs_ite}). 
Point (ii) in Fig.~\ref{fig:obj_crem_cremappx_vs_ite}, i.e., the first iteration of the algorithm, then shows the resulting total cost, i.e., including expected cost of failure as per CRM objective \eqref{eq:CRM_objective}, of the previous reliability-myopic DER control setpoints defined by CM in the initial iteration of the algorithm. 
As iterations progress, CRM co-optimizes the cost of operation and expected failure cost, resulting in point (iii) in Fig.~\ref{fig:obj_crem_cremappx_vs_ite} and achieving a \unit[11.7]{\%} reduction of expected cost of unserved energy.
Operational cost similarly increase by \unit[12.6]{\%}, however at an absolute scale that is two orders of magnitude lower.

\begin{figure}[b]
    \centering
    \begin{subfigure}[b]{0.495\textwidth}
        \centering
        \includegraphics[width=\textwidth]{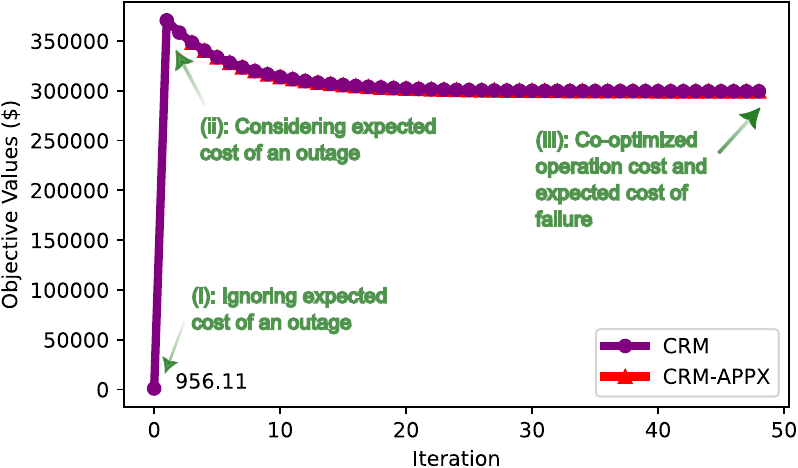}
        \caption{CRM \& CRM-APPX objective vs Iteration}
        \label{fig:obj_crem_and_approx}
    \end{subfigure}
    \hfill
    \begin{subfigure}[b]{0.495\textwidth}
        \centering
        \includegraphics[width=\textwidth]{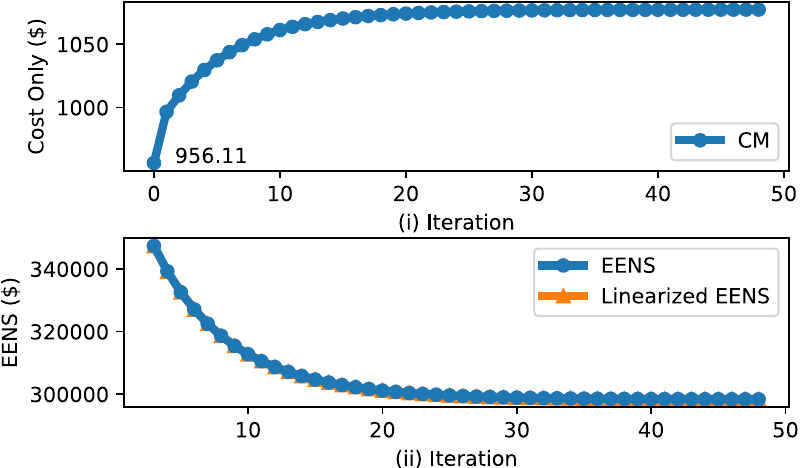}
        \caption{(i) $\sum_{t\in\mathcal{T}} C_t^{\text{op}}$\!\!  \& (ii) $\sum_{t\in\mathcal{T}} C_t^{\text{eens}}$\!\!  vs Iteration}
        \label{fig:obj_op_and_eens}
    \end{subfigure}
    \caption{Objective values for CRM and CRM-APPX in each iteration}
    \label{fig:obj_crem_cremappx_vs_ite}
\end{figure}

\paragraph{DER dispatch.} 
Fig.~\ref{fig:pf_cm_crem} shows the resulting active power dispatch of all DERs for the CM and CRM.
Figs.~\ref{fig:crem_voltage} to \ref{fig:crem-cm_current} illustrate the resulting voltages and currents in the system and highlight the respective differences when applying CRM.
Additionally, Fig.~\ref{fig:v_cs} illustrates the resulting voltages and currents for the CRM and the differences between the CM and CRM results in a spatial resolution for the 18:00, i.e., the time step with the highest temperature.

\begin{figure}
    \centering
    \begin{subfigure}[b]{0.495\textwidth}
        \centering
        \includegraphics[width=\textwidth]{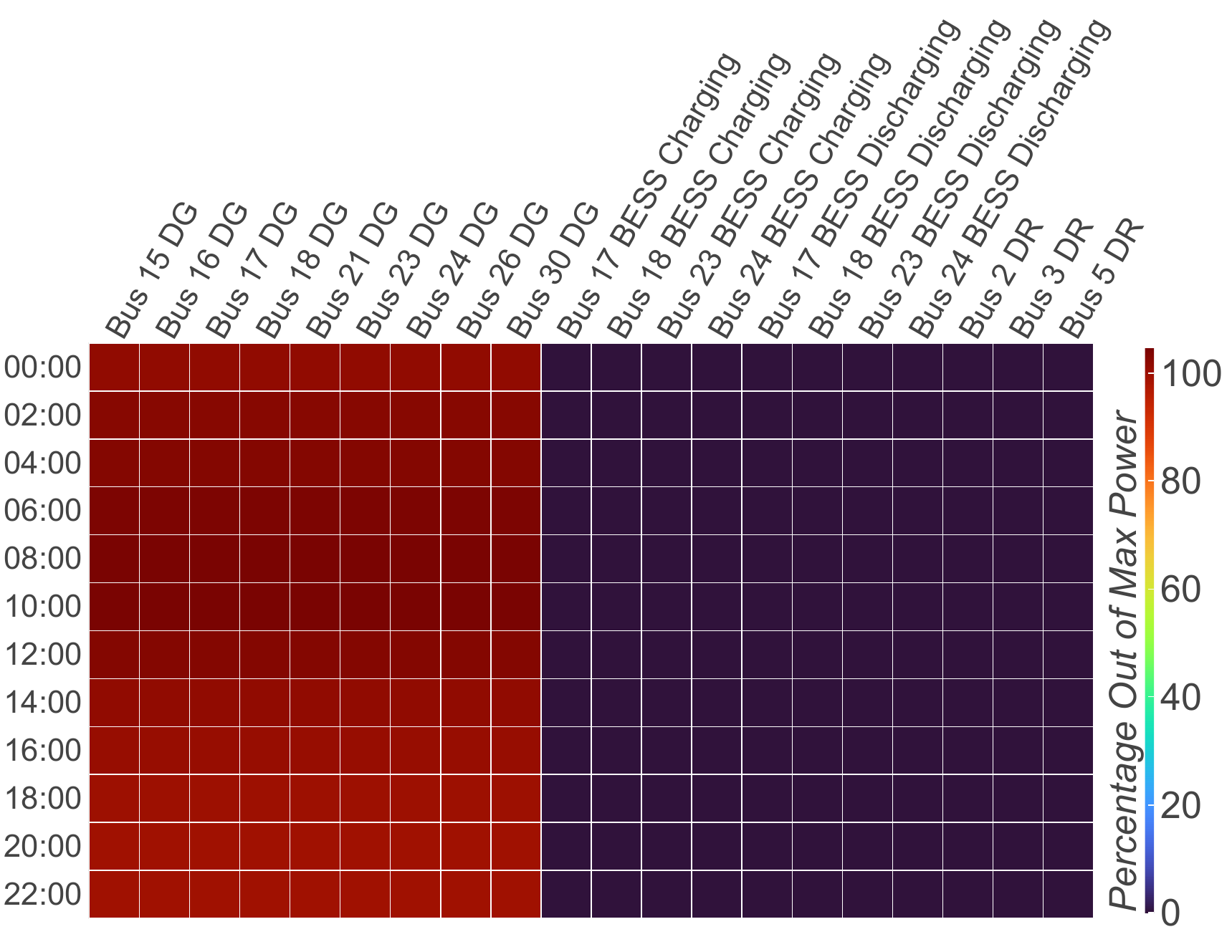}
        \caption{CM}
        \label{fig:pf_cm}
    \end{subfigure}
    \hfill
    \begin{subfigure}[b]{0.495\textwidth}
        \centering
        \includegraphics[width=\textwidth]{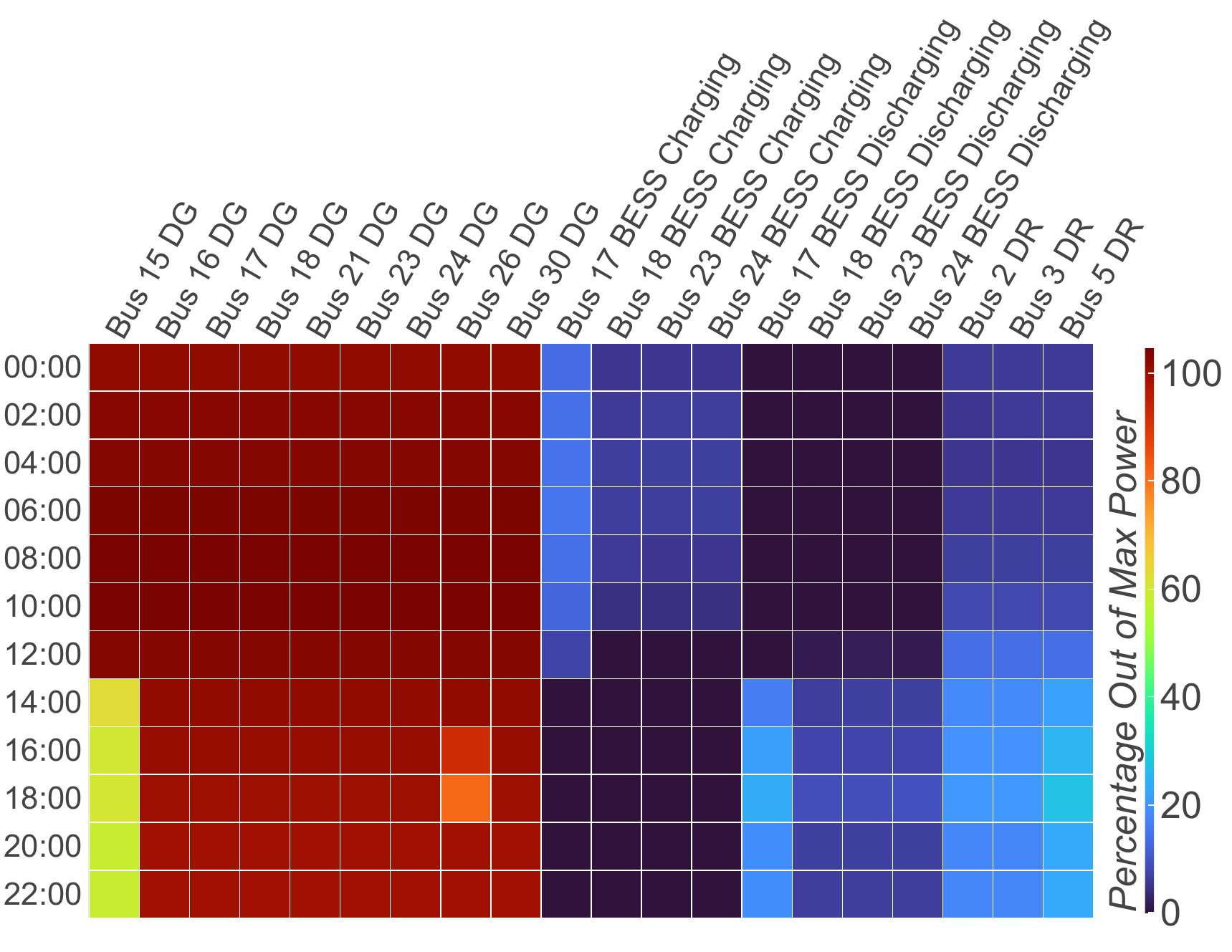}
        \caption{CRM}
        \label{fig:pf_crem}
    \end{subfigure}
    \caption{DER active power dispatch in percent utilization of maximum power limit.}
    \label{fig:pf_cm_crem}
\end{figure}

We observe from the dispatch results in Fig.~\ref{fig:pf_cm} that in the reliability-myopic CM only energy supplied from the substation and cheap local DG generation is required. 
The CM does not consider battery discharging (which would require prior charging) or the relatively costly utilization of DR, because the system is able to ensure that demand is met and all technical constraints, as defined in \eqref{eq:const_line_1}--\eqref{eq:const_battery_d_lim_real}, are satisfied. 
Fig.~\ref{fig:crem_voltage}, for example, shows that all voltage magnitudes remain within the technical limits of \unit[0.9]{p.u.} to \unit[1.1]{p.u.} (for squared voltage magnitudes).
When switching to CRM, the model now aims to find DER control setpoints that decrease the expected cost of EENS by optimally decreasing the power- and current-dependent failure rate of the system components. 
To this end, more DER is utilized as shown in Fig.~\ref{fig:pf_crem}.
We observe that CRM decides to charge batteries at lower load hours and discharge them at higher load hours (see also Fig.~\ref{fig:temp_power_daily} for the studied load profile) to reduce net bus net power demand.
For the same reason, and in contrast to CM, CRM utilizes DR during high load hours.
As a result, voltage levels increase at the relevant buses and time steps as shown in Fig.\ref{fig:crem-cm_voltage}. 

Net power demand reduction at buses with DERs is particularly impactful on line currents and their resulting failure rate. Fig.~\ref{fig:crem-cm_current} shows that during high load and high temperature hours (time steps $16$ to $22$), currents are reduced. During lower load hours Fig.~\ref{fig:crem-cm_current} shows higher currents to accommodate the higher load from battery charging as per Fig.~\ref{fig:pf_crem}. The impact of using CRM on line currents is reduced with increasing distance from the substation because of the radial topology of the network.

\begin{figure}[b]
    \centering    
    \includegraphics[width=1.0\linewidth]{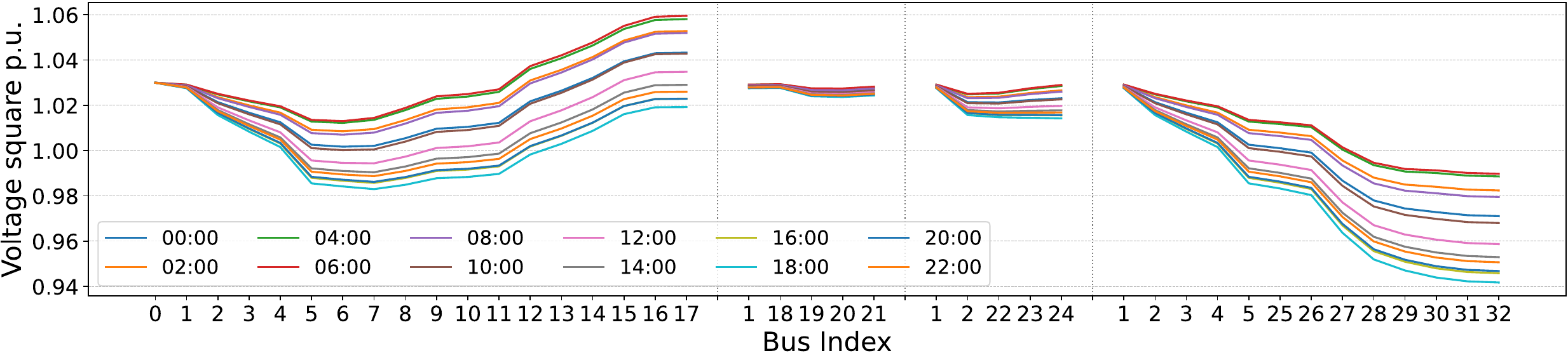}
    \caption{Voltage square ($v_{i,t}$ in p.u.) for CM.}
    \label{fig:crem_voltage}
\end{figure}

\begin{figure}
    \centering    
    \includegraphics[width=1.0\textwidth]{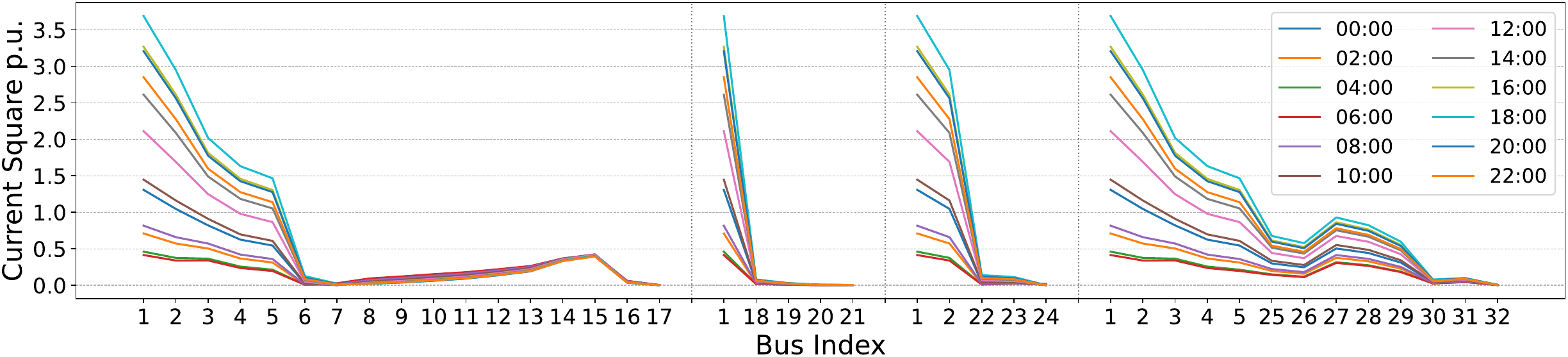}
    \caption{Current square ($l_{i,t}$ in p.u.) for CM}
    \label{fig:crem_current}
\end{figure}

\begin{figure}
    \centering    
    \includegraphics[width=1.0\textwidth]{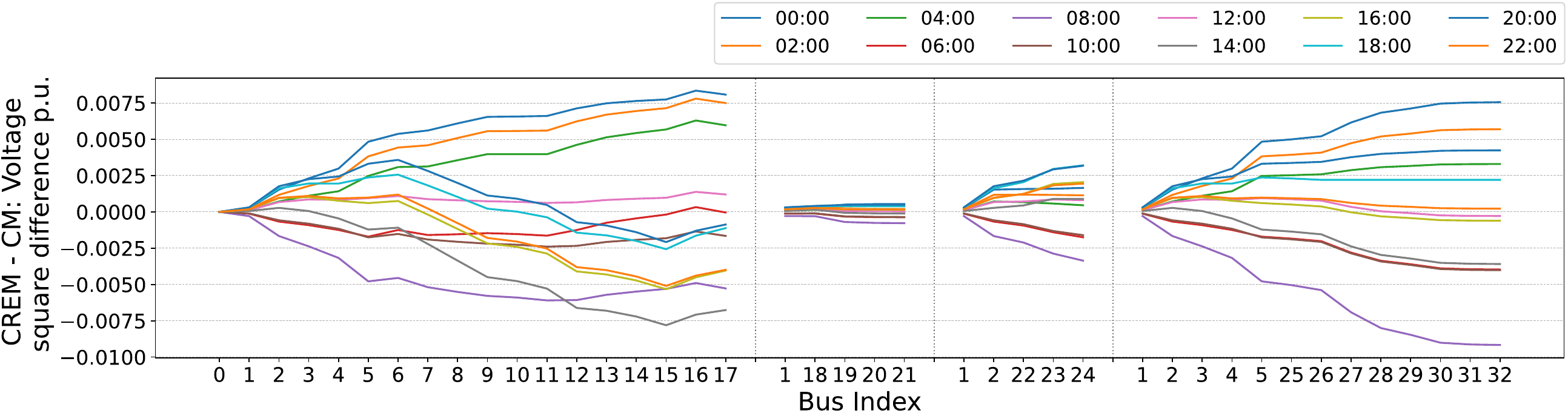}
    \caption{Voltage square difference ($v_{i,t}$ in p.u.) between CRM and CM. Positive values indicate an increase in CRM compared to CM.}
    \label{fig:crem-cm_voltage}
\end{figure}

\begin{figure}
    \centering    
    \includegraphics[width=1.0\textwidth]{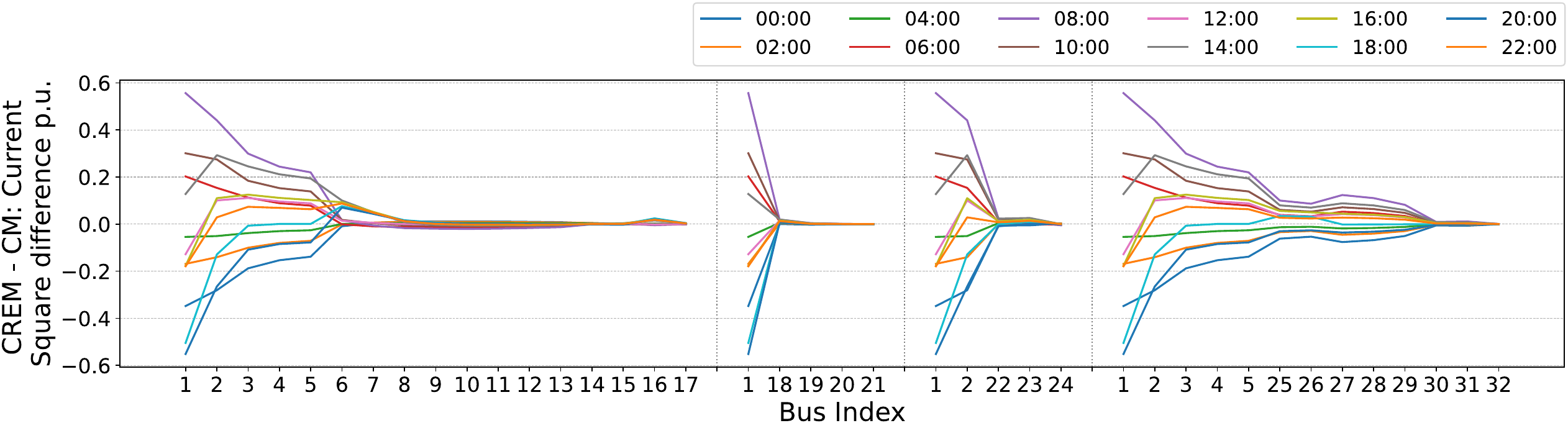}
    \caption{Current square difference ($l_{i,t}$ in p.u.) between CRM and CM. Positive values indicate an increase in CRM compared to CM.}
    \label{fig:crem-cm_current}
\end{figure}

\begin{figure}
    \centering
    \begin{subfigure}[b]{0.495\textwidth}
        \centering
        \includegraphics[width=\textwidth]{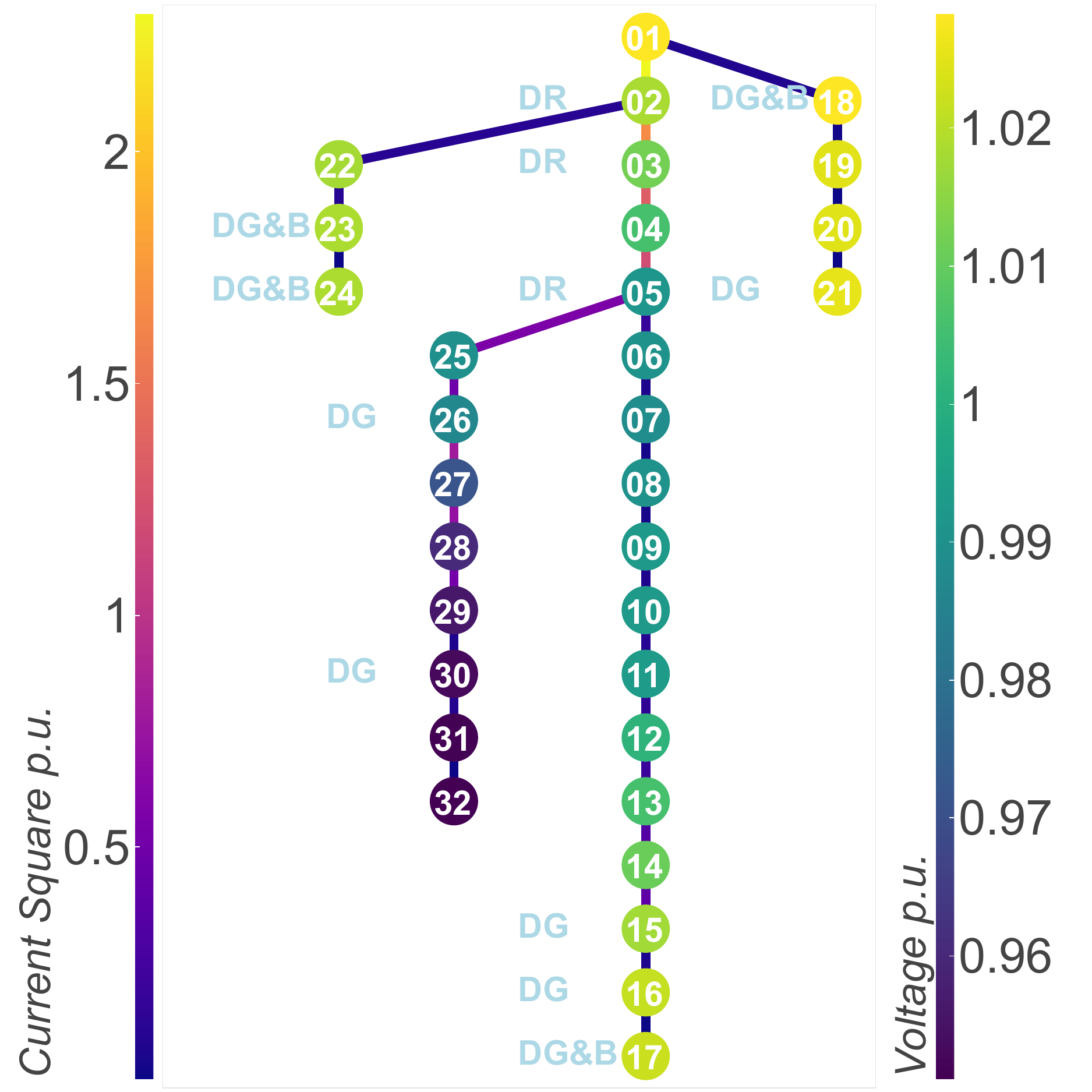}
        \caption{CRM voltage and current}
        \label{fig:v_cs_crem}
    \end{subfigure}
    \hfill
    \begin{subfigure}[b]{0.495\textwidth}
        \centering
        \includegraphics[width=\textwidth]{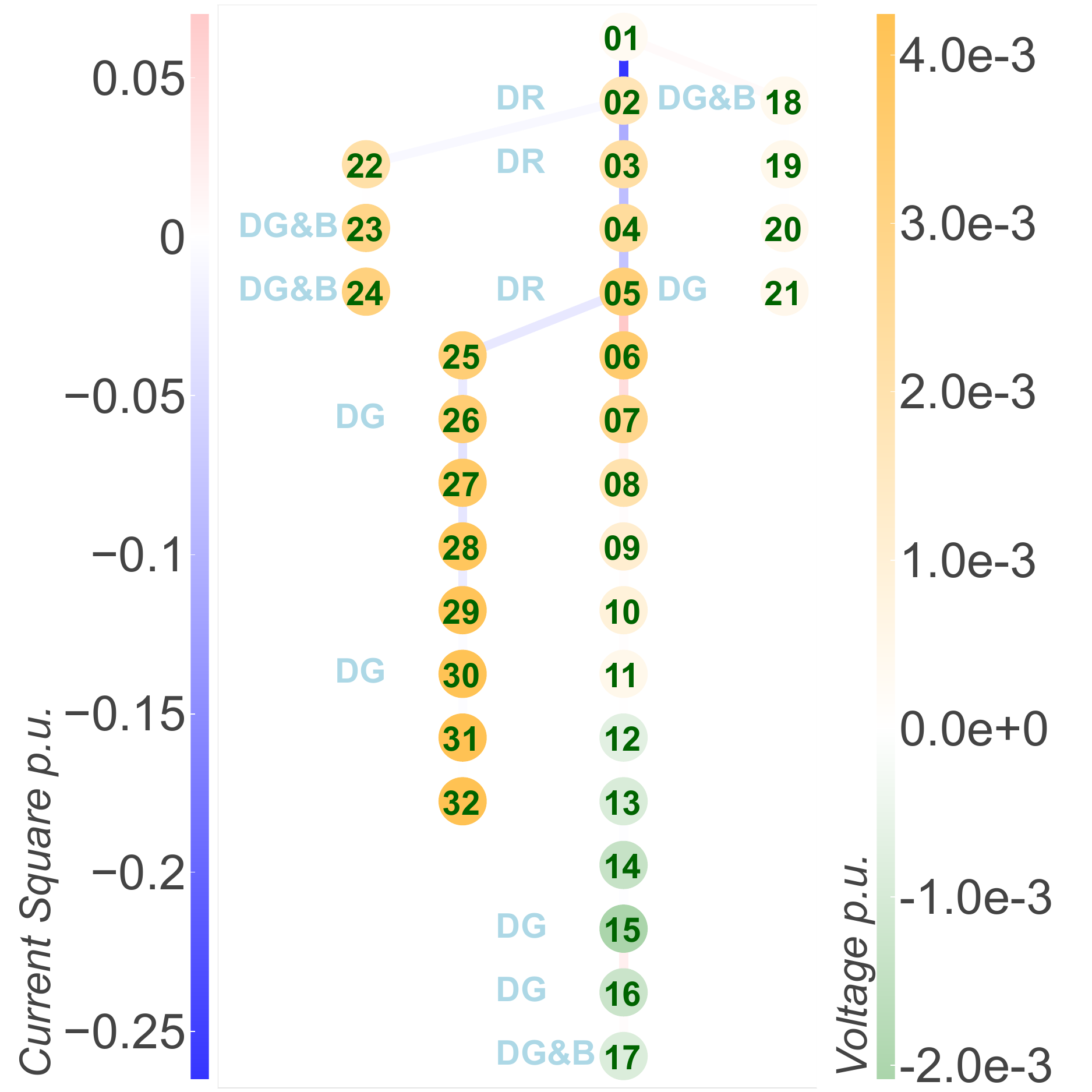}
        \caption{Voltage and current difference CRM-CM}
        \label{fig:v_cs_crem_cm}
    \end{subfigure}
    \caption{Spatial illustration of voltage and current results for CRM (absolute in (a) and relative to CM in (b)).}
    \label{fig:v_cs}
\end{figure}

\paragraph{Failure rates.}

Figs.~\ref{fig:bus_line_fail_prob}, \ref{fig:bus_line_fail_prob_600}, and \ref{fig:bus_line_fail_prob_1800} itemize the resulting change in line and bus failure rates that are achieved by CRM. 
Fig.~\ref{fig:bus_line_fail_prob} shows the average resulting failure rate change across all time steps for the studied day. While most failure probabilities are reduced on average, some buses and lines show a slight failure probability increase. 
This observed increase is an acceptable trade-off made by the CRM as to reduce the overall expected cost of EENS.
From Fig.~\ref{fig:bus_line_fail_prob}, we see that this trade-off made spatially by the CRM, i.e., across the buses, and also across time as the inspection of Figs.~\ref{fig:bus_line_fail_prob_600} and \ref{fig:bus_line_fail_prob_1800} reveals.
At 06:00 (Fig.~\ref{fig:bus_line_fail_prob_600}) the CRM accepts a higher failure probability for some lines to accommodate BESS charging as discussed above.
As a result, it can achieve an over-proportionally reduced failure probability for lines 1 and 2 at 18:00 (Fig.~\ref{fig:bus_line_fail_prob_1800}) enabled by BESS discharging and DR activation. 
Additionally, due to the radial topology of the grid, the reliability improvements at lines close to the substation improve the reliability of the entire system (see Section~\ref{sec:3_crem} above). 

\begin{figure}
    \centering
    \begin{subfigure}[b]{0.495\textwidth}
        \centering
        \includegraphics[width=\textwidth]{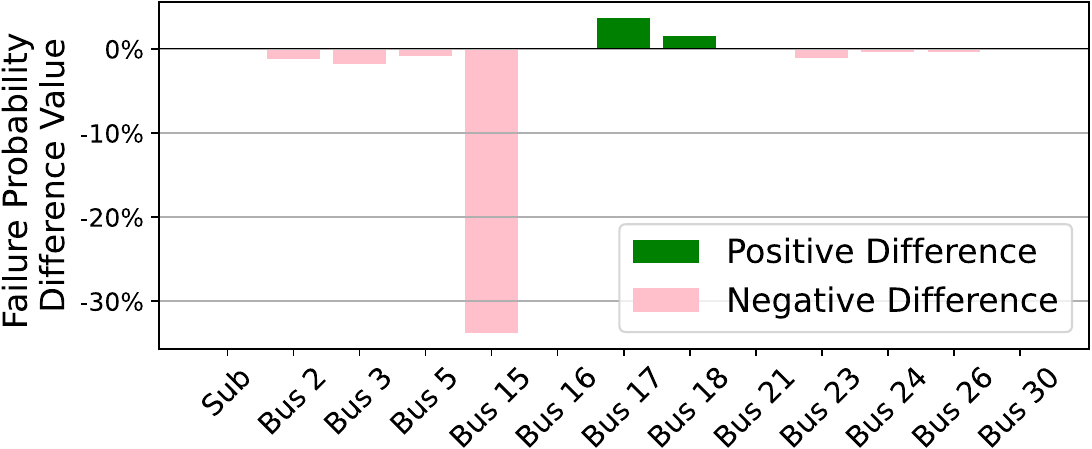}
        \caption{Relative change in line failure probability}
        \label{fig:bus_fail_prob_diff}
    \end{subfigure}
    \hfill
    \begin{subfigure}[b]{0.495\textwidth}
        \centering
        \includegraphics[width=\textwidth]{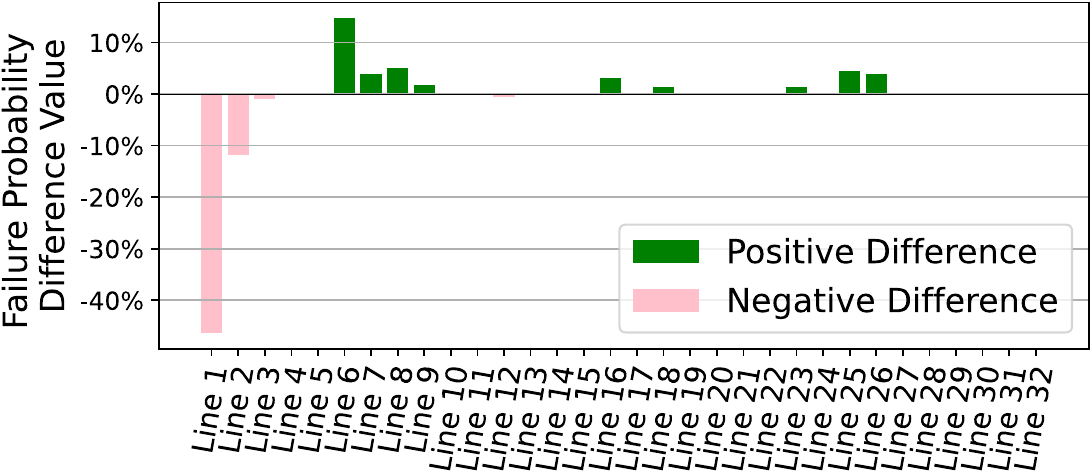}
        \caption{Relative change in line failure probability}
        \label{fig:line_fail_prob_diff}
    \end{subfigure}
    \caption{Relative change in bus and line failure probability (CRM relative to CM in percent) on average for all time steps.}
    \label{fig:bus_line_fail_prob}
\end{figure}

\begin{figure}
    \centering
    \begin{subfigure}[b]{0.495\textwidth}
        \centering
        \includegraphics[width=\textwidth]{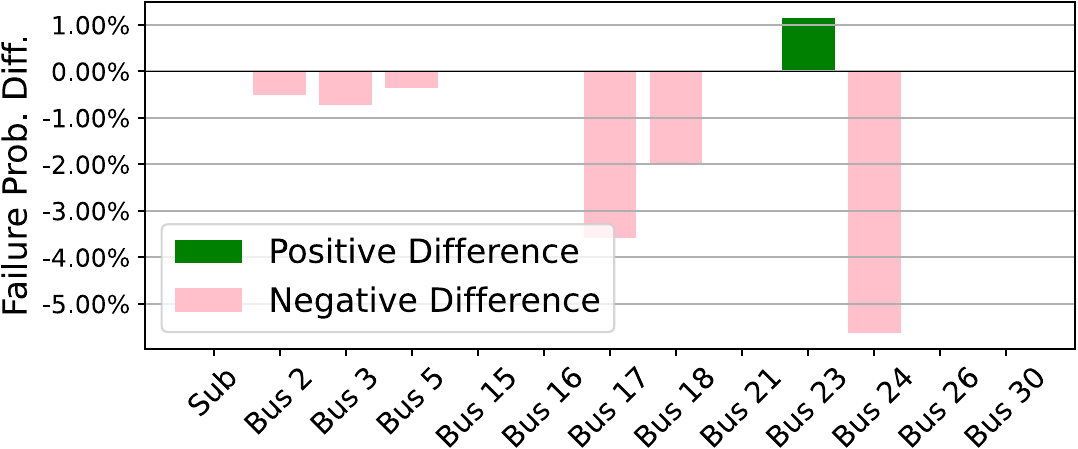}
        \caption{Relative change in bus failure probability}
        \label{fig:bus_fail_prob_diff_600}
    \end{subfigure}
    \hfill
    \begin{subfigure}[b]{0.495\textwidth}
        \centering
        \includegraphics[width=\textwidth]{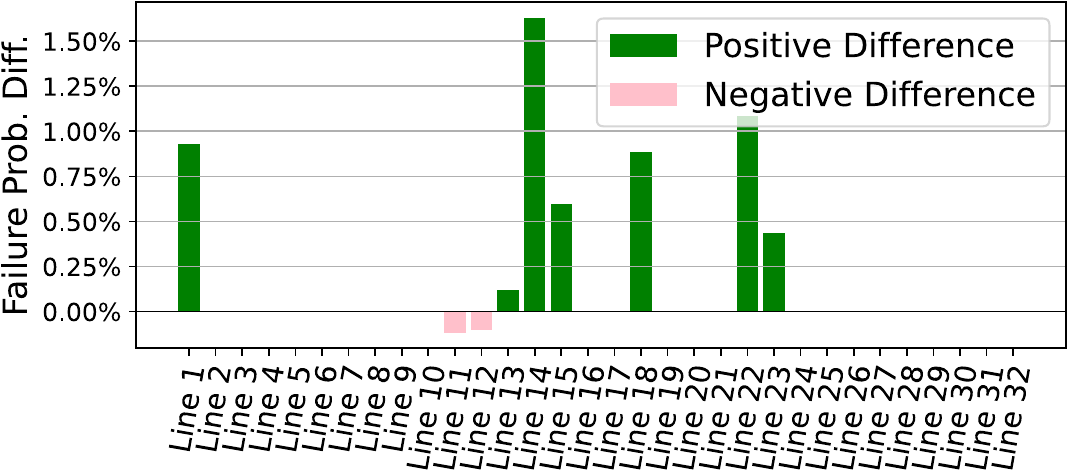}
        \caption{Relative change in line failure probability}
        \label{fig:line_fail_prob_diff_600}
    \end{subfigure}
    \caption{Relative change in bus and line failure probability (CRM relative to CM in percent) at time 06:00.}
    \label{fig:bus_line_fail_prob_600}
\end{figure}

\begin{figure}
    \centering
    \begin{subfigure}[b]{0.495\textwidth}
        \centering
        \includegraphics[width=\textwidth]{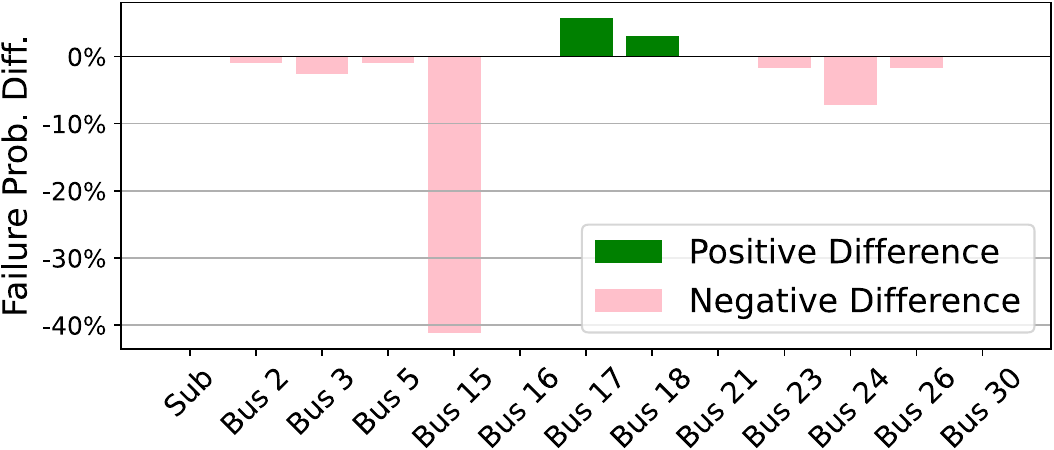}
        \caption{Relative change in bus failure probability}
        \label{fig:bus_fail_prob_diff_1800}
    \end{subfigure}
    \hfill
    \begin{subfigure}[b]{0.495\textwidth}
        \centering
        \includegraphics[width=\textwidth]{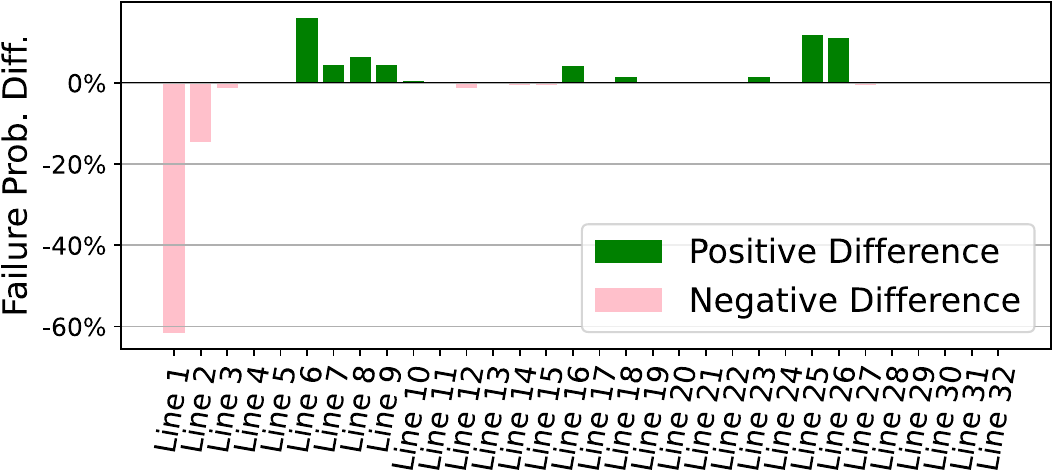}
        \caption{Relative change in line failure probability}
        \label{fig:line_fail_prob_diff_1800}
    \end{subfigure}
    \caption{Relative change in bus and line failure probability (CRM relative to CM in percent) at time 18:00.}
    \label{fig:bus_line_fail_prob_1800}
\end{figure}

\section{Conclusion}
\label{sec:conclusion}
Motivated by the potential of distributed energy resources (DERs) to contribute to distribution system reliability, this paper proposed methods to model and optimize decision- and context-dependent reliability in active distribution systems. 
We first established a reliability-myopic cost-only model (CM) of an active distribution system with controllable DERs as a baseline.
We then proposed a reliability-aware modification of CM, the cost-and-reliability model (CRM), that internalizes the cost of expected energy not served (EENS). We modeled EENS via a decision- and context-aware failure rate model using a logistic function model. This model reflected the dependency of component reliability model decisions (i.e., current and bus net power consumption as a result of DER control setpoints) and context (e.g., ambient temperature).
The resulting optimization problem was non-convex and we proposed an effective iterative approach to solve it. 
Additionally, we extensively discussed options to accurately estimate the parameters of the proposed reliability model. 
Building on a synthetic test dataset, we demonstrated and discussed our model and solution algorithm on a modified IEEE 33-bus test system. 
We observed that CRM efficiently manages power flows, utilizing available battery storage systems (BESSs) to store energy during low-load periods and discharge during high-load periods, and additional activating DR resources in response to increased demand. 

There are several future research directions that arise from the work presented in this paper. 
Our model currently does not consider tie line switches (TS) for fault isolation and service restoration. Introducing TS  will complicate the model of failure events due to the resulting option of optimal switching actions. 
We will extend our model to include TS operations alongside and extensive analysis of the scalability of the proposed approach.
Promising avenues for further improving model solution speed are (i) considering explicit parametric formulations of the linearized reliability functions and (ii) exploring options for machine-learning supported approximations of the reliability function.
Further work is also required on performance guarantees for the proposed sequential convex programming approach.

\newpage
\appendix
\section{Logistic Regression Model and Cox' PH Model Comparison for Discrete Time Interval}\label{appendix:log_reg_cox}

We compare the proposed logistic regression model and Cox proportional hazard (PH) model by obtaining reformulations of both models and show their equivalence for our studied problem.
Throughout the Appendix \ref{appendix:log_reg_cox}, we use $t^{\text{c}}$ to denote the continuous time index. The following analysis is a modified version of the work presented in \cite{abbott1985logistic}.

\paragraph{Proportional Hazard (PH) model.} 
The \textit{hazard rate function}, or \textit{failure rate function}, is defined as the limit of the failure rate as $\Delta t^{\text{c}}$ approaches $0$ (\cite{pham_sre_1}):
\begin{equation*}
    \lambda(t^{\text{c}}) = \lim_{\Delta t^{\text{c}} \rightarrow 0} \frac{Re(t^{\text{c}}) - Re(t^{\text{c}} + \Delta t^{\text{c}})}{\Delta t^{\text{c}} Re(t^{\text{c}})}.
\end{equation*}
In the PH model, the hazard rate is defined as (\cite{abbott1985logistic, cox1972}):
\begin{equation} \label{eq:appendix_cox_hazard}
    \lambda(t^{c}) = \lambda_0(t^{\text{c}}) \exp{\left( \sum \beta_k x_k \right)},
\end{equation}
where $\lambda_0(t^{\text{c}})$ is the baseline failure rate, variables $X_k$ are covariates and $\beta_k$'s are the coefficients of the $X_k$. 
The relationship between the hazard rate and reliability function is given by (\cite{abbott1985logistic, pham_sre_1}):
\begin{equation*}
    \lambda(t^{\text{c}}) = \frac{d \log\left[ Re(t^{\text{c}}) \right]}{dt^{\text{c}}}
\end{equation*}
which can also be written as (\cite{pham_sre_1}):
\begin{equation} \label{eq:appendix_def_rel}
    Re(t^{\text{c}}) = \exp{\left( -\int_{0}^{t^{\text{c}}} \lambda(z) \right)} dz
\end{equation}
Note that conditional probability of the event not happening during the time interval $(t^{\text{c}}, \Delta t^{\text{c}})$ given that the event has not happened before time $t^{\text{c}}$ is given as $Re(t^{\text{c}} + \Delta t^{\text{c}})/Re(t^{\text{c}})$. 

Define:
\begin{equation*}
    p_{t_i} = \text{Pr}\left( Y_{t_i} = 1 | Y_{t_j} = 0, t_j < t_i \right),
\end{equation*}
where $i$, $j$ are some non-negative integers. The relationship between the continuous time $t^{\text{c}}$ and the discrete time step $t$ is given by $t_i = \left(t_{i - \! 1}^{\text{c}}, t_{i - \! 1}^{\text{c}} \!\! + \! \delta  \right]$, where $=t_{i - \! 1}^{\text{c}} \in t^{\text{c}}$ and $\delta$ denotes the length of the time step (\cite{abbott1985logistic}). 

From the definition of conditional probability and \eqref{eq:appendix_def_rel}:
\begin{align} \label{eq:appendix_cox_log_conditional}
    &\text{Pr}\left( Y_{t_i} = 0 | Y_{t_j} = 0, t_j < t_i \right) = \frac{Re(t_{i - \! 1}^{\text{c}} + \delta)}{Re(t_{i - \! 1}^{\text{c}})} = \frac{Re(t_i)}{Re(t_{i - \! 1})} \notag \\
    =&\frac{\exp{\left[ - \int_0^{t_{i - \! 1}^{\text{c}} + \delta} \lambda(z) dz\right]}}{\exp{\left[ - \int_0^{t_{i - \! 1}^{\text{c}}} \lambda(z) dz\right]}} = \exp{\left[ - \int_{t_{i - \! 1}^{\text{c}}}^{t_{i - \! 1}^{\text{c}} + \delta} \lambda(z) dz\right]} \notag \\
    =& \exp{\left[ -\int_{t_{i - \! 1}^{\text{c}}}^{t_{i - \! 1}^{\text{c}} + \delta} \lambda_{0}(z) \exp{\left( \sum \beta_k x_k \right)} dz \right]} \notag \\
    =& \exp{\left[ -\exp{\left( \sum \beta_k x_k \right)\int_{t_{i - \! 1}^{\text{c}}}^{t_{i - \! 1}^{\text{c}} + \delta} \lambda_{0}(z) dz} \right]} \notag \\
    =& \exp{\left[ -\exp{\left( \sum \beta_k x_k \right) \cdot \exp{(\alpha_k)}} \right]} \notag \\
    =& \exp{\left[ -\exp{\left( \alpha_k + \sum \beta_k x_k \right)} \right]},
\end{align}
where $\alpha_k = \log{\left[ \int_{t_{i - \! 1}^{\text{c}}}^{t_{i - \! 1}^{\text{c}} + \delta} \lambda_{0}(z) dz \right]}$.
Next, we use use $p_{t_i}$ to rewrite \eqref{eq:appendix_cox_log_conditional}:
\begin{equation*}
    1 - p_{t_i} = \exp{\left[ -\exp{\left( \alpha_k + \sum \beta_k x_k \right)} \right]},
\end{equation*}
and write $1 - p_{t_i}$ in terms of a power series as:
\begin{equation*}
    1 - p_{t_i} = 1 - G + \frac{G^2}{2!} - \frac{G^3}{3!} + \cdots + (-1)^n \frac{G^n}{n!} + \cdots
\end{equation*}
where $G = \exp{\left( \alpha_{k} + \sum \beta_k x_k \right)}$. When the probability of an event in each interval is small, $G$ will be small, allowing higher-order terms of G to be neglected, leaving only $1 - G$ (\cite{abbott1985logistic}).  

\paragraph{Logistic Regression Model.}
By definition of the logit transform (\cite{mccullagh_glm_4}):
\begin{equation*}
    \log{\left[ \frac{p_{t_i}}{1 - p_{t_i}} \right]} = \alpha_{k} + \sum \beta_k x_k \Rightarrow 1 - p_{t_i} = \frac{1}{1 + \exp{\left[\alpha_{k} + \sum \beta_k \alpha_k\right]}}
\end{equation*}
In case of $|\exp{\left(\alpha_{k} + \sum \beta_k \alpha_k\right)}| < 1$, $1 - p_{t_i}$ can be expanded using geometric series as:
\begin{equation*}
    1 - p_{t_i} = 1 - G + G^2 - G^3 + \cdots + (-1)^n G^n + \cdots
\end{equation*}

\noindent Similarly, when the probability of an event in each interval is small, leaving only $1 - G$ (\cite{abbott1985logistic}). 

In our study, the event probability refers to the failure of a component, such as a bus or line, within a $2$-hour window. Based on historical data, these failure probabilities are small. Therefore, it is appropriate to approximate the Cox' PH model using generalized logistic regression (\cite{richard_epdr_4, brown2004failure}). 

\section{Clarification of Claim \ref{claim:non_convex}} \label{appendix:proof_non_convex}
We build this discussion on the fact that a twice differentiable function $f:\mathbf{D}\subset\mathbb{R}^n \rightarrow \mathbb{R}$ is convex (concave) if and only if its Hessian is positive (negative) semi-definite (\cite{boyd_co_3}), i.e.:
\begin{equation*}
    H(\bm{v}) = \left[ h_{i, j}(\bm{v}) \right] = \left[ \frac{\partial^2 f(\bm{v})}{\partial v_i \partial v_j} \right] \succeq (\preceq) \text{ } 0.
\end{equation*}
Consider the function $\mathcal{C}_{1}^{\text{eens}}\left( \bm{v}_{\text{b}1}\right)$, which represents the expected cost of energy not served by bus $i=1$ at time step $t=1$:
\begin{equation*}
    \mathcal{C}_{1}^{\text{eens}}\left( \bm{v}_{\text{b}1}\right) = Q_{1, 1}^{\rm b} \left[ 1 - \left( 1 - \rm{Pr}_{1, 1}^{\rm b} \right) \left( 1 - \rm{Pr}_{1, 1}^{\rm l} \right) \right], 
\end{equation*}
where
\begin{align*}
    Q_{1, 1}^{\rm b} &= w_{1}^{\rm c} p_{1, 1}^c + w_{1}^{\rm DG} p_{1, 1}^{\rm DG} + w_{1}^{\rm B, c} p_{1, 1}^{\rm B, c} + w_{1}^{\rm B, d} p_{1, 1}^{\rm B, d} + w_{1}^{\rm DR} p_{1, 1}^{\rm DR} \\
    \rm{Pr}_{1, 1}^{\text{b}} &= \frac{1}{1 + \lambda_{\text{b}, 1} \cdot e^{-\left(\beta_{1, 1}^{\rm b} \left|- p_{1, 1}^{\rm c} + p_{1, 1}^{\rm DG} - p_{1, 1}^{\rm B, c} + p_{1, 1}^{\rm B, d} + p_{1, 1}^{\rm DR}\right| + \beta_{2, 1}^{\rm b} {\rm AT}_1\right)}} \\
    \rm{Pr}_{1, 1}^{\text{l}} &= \frac{1}{1 + \lambda_{\text{l}, 1} \cdot e^{-\left(\beta_{1, 1}^{\rm l} l_{1, 1} + \beta_{2, 1}^{\rm l} {\rm AT}_1\right)}}.
\end{align*}

\noindent Note that the domain of all the variables, specifically in $\mathcal{C}_{1}^{\text{eens}}\left( \bm{v}_{\text{b}1}\right)$: $p_{1, 1}^{\rm B, c}$, $p_{1, 1}^{\rm B, d}$, $p_{1, 1}^{\rm DR}$, include the value $0$. For convenience of the below derivations and w.l.o.g., assume $p_{1, 1}^{\rm B, c} = 0$, $p_{1, 1}^{\rm B, d} = 0$ and $0 < p_{1, 1}^{\rm DG} \leq P_{1}^{\text{DG,max}}$. W.l.o.g. also assume that $- p_{1, 1}^{\rm c} + p_{1, 1}^{\rm DG} \geq 0$. Then:
\begin{align} \label{eq:eens_bus_1_full}
    &\mathcal{C}_{1}^{\text{eens}}\left( \bm{v}_{\text{b}1}\right) \notag \\
    =& (w_{1}^{\rm c} p_{1, 1}^c \! + \! w_{1}^{\rm DG} p_{1, 1}^{\rm DG}) \!\! \left[ \! 1 \! - \!\! \left( \!\! 1 \! - \! \frac{1}{1 + \lambda_{\text{b}, 1} e^{-\left(\beta_{1, \! 1}^{\rm b} p_{1, 1}^{\rm DG} \! + \beta_{2, 1}^{\rm b} {\rm AT}_1 \! \right)}} \!\! \right) \!\!\! \left( \!\! 1 \! - \! \frac{1}{1 + \lambda_{\text{l}, 1} e^{-\left(\beta_{1, \! 1}^{\rm l} l_{1, \! 1} \! + \beta_{2, \! 1}^{\rm l} {\rm AT}_1 \! \right)}} \!\! \right) \!\! \right].
\end{align}
Note that $w_{1}^{\rm c}$, $p_{1, 1}^c$, $w_{1}^{\rm DG}$, $\lambda_{\text{b}, 1}$, $\beta_{1, 1}^{\rm b}$, $\beta_{2, 1}^{\rm b}$, ${\rm AT}_1$, $\lambda_{\text{l}, 1}$, $\beta_{1, 1}^{\rm l}$ and $\beta_{2, 1}^{\rm l}$ are all parameters. For the remainder of this section, we use $\nu_{*} \in \mathbb{R}$ to denote some constant scalar terms. Rewrite \eqref{eq:eens_bus_1_full} as: 
\begin{equation*} 
    \mathcal{C}_{1}^{\text{eens}}\left( \bm{v}_{\text{b}1}\right) = (\nu_{1} + \nu_{2} p_{1, 1}^{\rm DG}) \left[ 1 - \left( 1 - \frac{1}{1 + \nu_{3} e^{-\left(\nu_{4} p_{1, 1}^{\rm DG}\right)}} \right) \left( 1 - \frac{1}{1 + \nu_{5} e^{-\left(\nu_{6} l_{1, 1}\right)}}\right)\right],
\end{equation*}
where $\nu_1 = \omega_1^{\text{c}}p_{1, 1}^{\text{c}} = 8043.2197$, $\nu_2 = \omega_1^{\text{DG}} = 2 \times 10^{4}$, $\nu_3 = \lambda_{\text{b}, 1}e^{-\beta_{2, 1}^{\text{b}}\text{AT}_{1}} = 5.0833 \times 10^{-12}$, $\nu_4 = \beta_{1, 1}^{\text{b}} = 0.3039$, $\nu_5 = \lambda_{\text{l}, 1}e^{-\beta_{2, 1}^{\text{l}}\text{AT}_{1}} = 1.1032 \times 10^{-11}$, $\nu_6 = \beta_{1,1}^{l} = 0.9501$. \\

\noindent The Hessian of $\mathcal{C}_{1}^{\text{eens}}\left( \bm{v}_{\text{b}1}\right)$ is
\begin{equation} \label{eq:hessian_r_1_1}
    H\left( p_{1, 1}^{\rm DG}, l_{1, 1} \right) = \begin{bmatrix}
        \frac{\partial^2 \mathcal{C}_{1}^{\text{eens}}\left( \bm{v}_{\text{b}1}\right)}{\partial {p_{1, 1}^{\rm DG}}^2} & \frac{\partial^2 \mathcal{C}_{1}^{\text{eens}}\left( \bm{v}_{\text{b}1}\right)}{\partial p_{1, 1}^{\rm DG} \partial l_{1, 1}} \\
        \frac{\partial^2 \mathcal{C}_{1}^{\text{eens}}\left( \bm{v}_{\text{b}1}\right)}{\partial l_{1, 1} \partial p_{1, 1}^{\rm DG}} & \frac{\partial^2 \mathcal{C}_{1}^{\text{eens}}\left( \bm{v}_{\text{b}1}\right)}{\partial {l_{1, 1}}^2}, 
    \end{bmatrix}
\end{equation}
where
\begin{align*}
    \frac{\partial^2 \mathcal{C}_{1}^{\text{eens}}\left( \bm{v}_{\text{b}1}\right)}{\partial {p_{1, 1}^{\rm DG}}^2} &= \frac{ 2\nu_2 \nu_3 \nu_4 \left( 1 - \frac{1}{\nu_5 e^{-\left(\nu_6 l_{1, 1}\right)} + 1} \right) e^{-\left(\nu_4 p_{1, 1}^{\rm DG} \right)}}{\left( \nu_3 e^{-\left( \nu_4 p_{1, 1}^{\rm DG} \right)} + 1 \right)^2} \notag \\
    &+ \frac{ 2\nu_3^2 \nu_4^2\left(\nu_2 p_{1, 1}^{\rm DG} + \nu_1 \right) \left( 1 - \frac{1}{\nu_5 e^{-\left(\nu_6 l_{1, 1}\right)} + 1} \right) e^{-\left(2\nu_4 p_{1, 1}^{\rm DG} \right)}}{\left( \nu_3 e^{-\left( \nu_4 p_{1, 1}^{\rm DG} \right)} + 1 \right)^3} \notag \\
    &- \frac{ \nu_3\nu_4^2 \left( \nu_2 p_{1, 1}^{\rm DG} + \nu_1 \right) \left( 1 - \frac{1}{\nu_5 e^{-\left(\nu_6 l_{1, 1}\right)} + 1} \right) e^{-\left(\nu_4 p_{1, 1}^{\rm DG} \right)}}{\left( \nu_3 e^{-\left( \nu_4 p_{1, 1}^{\rm DG} \right)} + 1 \right)^2} \notag \\
    \frac{\partial^2 \mathcal{C}_{1}^{\text{eens}}\left( \bm{v}_{\text{b}1}\right)}{\partial p_{1, 1}^{\rm DG} \partial l_{1, 1}} &= \frac{\partial^2 \mathcal{C}_{1}^{\text{eens}}\left( \bm{v}_{\text{b}1}\right)}{\partial l_{1, 1} \partial p_{1, 1}^{\rm DG}} \notag \\
    &= \frac{ \nu_2 \nu_5 \nu_6 \left( 1 - \frac{1}{\nu_3 e^{-\left(\nu_4 p_{1, 1}^{\rm DG}\right)} + 1} \right) e^{-\left(\nu_6 l_{1, 1} \right)}}{\left( \nu_5 e^{-\left( \nu_6 l_{1, 1} \right)} + 1 \right)^2} \notag \\
    &-\frac{ \nu_3 \nu_4 \nu_5 \nu_6 \left(\nu_2 p_{1, 1}^{\rm DG} + \nu_1\right) e^{-\left( \nu_4 p_{1, 1}^{\rm DG} \right)} e^{-\left(\nu_6 l_{1, 1}\right)} }{\left( \nu_3 e^{-\left( \nu_4 p_{1, 1}^{\rm DG} \right)} + 1 \right)^2 \left( \nu_5 e^{-\left( \nu_6 l_{1, 1} \right)} + 1 \right)^2} \notag \\
    \frac{\partial^2 \mathcal{C}_{1}^{\text{eens}}\left( \bm{v}_{\text{b}1}\right)}{\partial {l_{1, 1}}^2} &= \frac{ 2\nu_5^2 \nu_6^2\left(\nu_2 p_{1, 1}^{\rm DG} + \nu_1\right) \left( 1 - \frac{1}{\nu_3 e^{-\left(\nu_4 p_{1, 1}^{\rm DG}\right)} + 1} \right) e^{-\left(2\nu_6 l_{1, 1}\right)}}{\left( \nu_5 e^{-\left(\nu_6 l_{1, 1} \right)} + 1 \right)^3} \notag \\
    &- \frac{ \nu_5\nu_6^2(\nu_2 p_{1, 1}^{\rm DG} + \nu_1) \left( 1 - \frac{1}{\nu_3 e^{-\left(\nu_4 p_{1, 1}^{\rm DG}\right)} + 1} \right) e^{-\left(\nu_6 l_{1, 1} \right)}}{\left( \nu_5 e^{-\left( \nu_6 l_{1, 1} \right)} + 1 \right)^2}. \notag 
\end{align*}
It is now straightforward to find a parametrization such that Hessian \eqref{eq:hessian_r_1_1} is not positive (negative) semidefinite. For example setting
$p_{1, 1}^{\rm DG} = 0.1$, $l_{1, 1} = 1$ and choosing $\mathbf{x} = [0.1, 0.4]^T$ we have $\mathbf{x}^{T}H \mathbf{x} < 0$ (setting
$p_{1, 1}^{\rm DG} = 0.1$, $l_{1, 1} = 1$ and choosing $\mathbf{x} = [0.4, 0.1]^T$ \mbox{we have $\mathbf{x}^{T}H \mathbf{x} > 0$).}

To establish the claim that $\mathcal{C}^{\text{eens}}$ is non-convex (non-concave), we first recall the definition of $\mathcal{C}^{\text{eens}}$ as the sum across all buses at every time step. Given that the domain of all variables includes $0$'s, we can set all variables $p_{i, t}^{\diamond}$ for $i \in \mathcal{N}$, $t \in \mathcal{T}$, $\diamond = \{\text{DG}, \text{(B,c)}, \text{(B,d)}, \text{DR}\}$ and $l_{j, t}$ for $j \in \mathcal{N}^{+}$, $t \in \mathcal{T}$ to $0$, except for $p_{1,1}^{\text{DG}}$ and $l_{1,1}$. Under this configuration, we have that $\mathcal{C}^{\text{eens}} = \mathcal{C}_{1}^{\text{eens}}\left( \bm{v}_{\text{b}1}\right) + \nu_7$, where $\nu_7$ is a constant term. This demonstrates the non-convexity (non-concavity) of $\mathcal{C}^{\text{eens}}$ and confirms the claim.

\section{Proof of Claim \ref{claim:add_var_equal}} \label{appendix:proof_add_var_equal}

\begin{proposition}\label{prop:aux_val_equality}
    For a function $f = f(\bm{t})$ that can be written as $\tilde{f} = \tilde{f}\left( \bm{t}, \bm{x} \right)$, where $\bm{t} = \left(  t_1, t_2, \cdots, t_m \right)^T$, $\bm{x} = \left( x_1, x_2, \cdots, x_n \right)$. And $x_i$'s are linear combinations of $\bm{t}$, $i = 1, \cdots, n$, i.e.:
    \begin{equation} \label{eq:linear_trans}
        x_i = c_{i, 1} t_1 + c_{i, 2} t_2 + \cdots c_{i, m} t_m + c_{i, 0} = \bm{c}_i^T \bm{t} + c_{i,0}
    \end{equation}
    $\bm{c}_i = \left(c_{i, 1}, c_{i, 2}, \cdots, c_{i, m} \right)^T$, $i = 1, \cdots, n$ are constant vectors. The first order Taylor expansion $\mathfrak{T}\left( f | \bm{t}^{\star}\right)$ of $f$ w.r.t. variables $\bm{t}$ 
    at the point $\bm{t}^{\star}$ is equal to the first order Taylor expansion of $\tilde{f}$ w.r.t. variables $(\bm{t},\bm{x})$ at the point $(\bm{t}^{\star}, \bm{x}^{\star})$:
    \begin{equation*}
        \mathfrak{T}\big( \tilde{f} | (\bm{t}^{\star}, \bm{x}^{\star}) \big) = \mathfrak{T}\left( f | \bm{t}^{\star} \right).
    \end{equation*}
\end{proposition}

\begin{proof}
We have:
    \begin{align}
        &\mathfrak{T}\left( f | \bm{t}^{\star} \right) = \sum_{j = 1}^{j = m} \frac{\partial f}{\partial t_j}\left( \bm{t}^{\star} \right)  \left( t_{j} - t_{j}^{\star} \right) + f\left( \bm{t}^{\star} \right) 
        \notag \\
        &\mathfrak{T}\left( \tilde{f} | \bm{t}^{\star}, \bm{x}^{\star}  \right) \notag \\
        =& \sum_{j = 1}^{j = m} \frac{\partial \tilde{f}}{\partial t_j}\left( \bm{t}^{\star}, \bm{x}^{\star} \right) \left( t_j - t_{j}^{\star} \right) + \sum_{i = 1}^{i = n} \frac{\partial \tilde{f}}{\partial x_i}\left( \bm{t}^{\star}, \bm{x}^{\star} \right) \left( x_i - x_{i}^{\star} \right) + \tilde{f}\left( \bm{t}^{\star}, \bm{x}^{\star} \right) \label{eq:aux_taylor}
    \end{align}
    Further, per definition of $\tilde{f}$: 
    \begin{equation*}
        \tilde{f} = \tilde{f}\left[ l_k \left( \bm{t}\right), k = 1, \cdots, m+n \right]
    \end{equation*}
    where we introduce $l_k\left( \bm{t}\right) = \bm{c}_k^T \bm{t} + c_{k, 0}$ as linear mappings in addition to \eqref{eq:linear_trans}.
    \begin{align*}
        c_{k, j} = \begin{cases}
            0 & \text{if } k \neq j \cap k \leq m \\
                1 & \text{if } k = j \cap k \leq m \\
                c_{*, *} &  \text{otherwise}
        \end{cases}.
    \end{align*}
    where $k = 1, \cdots, m+n$; $j = 1, \cdots, m$; $c_{\ast, \ast}$ and $c_{k, 0}$ are any constants. \\
    
    \noindent Per the chain rule we have:
    \begin{align} \label{eq:ori_aux_eq}
        &\mathfrak{T}\left( \tilde{f} | \bm{t}^{\star}, \bm{x}^{\star} \right) =\sum_{j = 1}^{j = m} \sum_{k = 1}^{k = m+n}\frac{\partial \tilde{f}}{\partial l_{k}} \frac{\partial l_{k}}{\partial t_j}\left( \bm{t}^{\star}, \bm{x}^{\star} \right) \left( t_j - t_j^{\star} \right) + \tilde{f}\left( \bm{t}^{\star}, \bm{x}^{\star} \right) \notag \\
        =& \sum_{j = 1}^{j = m} \left[ \sum_{k = 1}^{k = m}\frac{\partial \tilde{f}}{\partial l_{k}} \frac{\partial l_{k}}{\partial t_j}\left( \bm{t}^{\star}, \bm{x}^{\star} \right) \left( t_j - t_j^{\star} \right) + \sum_{k = m+1}^{k = m+n}\frac{\partial \tilde{f}}{\partial l_{k}} \frac{\partial l_{k}}{\partial t_j}\left( \bm{t}^{\star}, \bm{x}^{\star} \right) \left( t_j - t_j^{\star}\right) \right] \notag \\
        +& \tilde{f} \left( \bm{t}^{\star}, \bm{x}^{\star} \right).
    \end{align}
    Note that ${\partial l_j}/{\partial t_j} = 1$ and ${\partial l_k}/{\partial t_j} = 0$ for $k \neq j$. Thus, the first term in the R.H.S. of \eqref{eq:ori_aux_eq} is:
    \begin{equation*}
        \sum_{j = 1}^{j = m}\frac{\partial \tilde{f}}{\partial t_j} \left( \bm{t}^{\star}, \bm{x}^{\star} \right) \left( t_j - t_j^{\star} \right).
    \end{equation*}
    Because $\tilde{f}$ is a reformulation of $f$: the expansion point $\left( \bm{t}^{\star}, \bm{x}^{\star} \right) = \left( \bm{t}^{\star} \right)$ and $\tilde{f}\left( \bm{t}^{\star}, \bm{x}^{\star} \right) = f(\bm{t}^{\star})$. Therefore, the first term of \eqref{eq:aux_taylor} is equivalent to the first term of \eqref{eq:ori_aux_eq}. Since $x_i$ are linear combinations of $t_1, t_2, \cdots, t_,$, the second term in the R.H.S. of \eqref{eq:ori_aux_eq} can be written as:
    \begin{equation*}
        \sum_{j = 1}^{j = m}\sum_{k = m+1}^{k = m+n}\frac{\partial \tilde{f}}{\partial l_{k}} \frac{\partial l_{k}}{\partial t_j}\left( \bm{t}^{\star}, \bm{x}^{\star} \right) \left( t_j - t_j^{\star} \right) = \sum_{j = 1}^{j = m}\sum_{i = 1}^{i = n} (c_{i, j}) \frac{\partial \tilde{f}}{\partial x_{i}} (\bm{t}^{\star}, \bm{x}^{\star}) (t_j - t_j^{\star}).
    \end{equation*}
    Finally, the remaining summation terms in \eqref{eq:aux_taylor} can be written as
    \begin{align*}
        & \sum_{i = 1}^{i = n}\frac{\partial \tilde{f}}{\partial x_i}\left( \bm{t}^{\star} \right) \left(\left( c_{i,0} - c_{i, 0} \right) + \sum_{j = 1}^{j = m}c_{i, j} \left(t_{j} - t_{j}^{\star}\right)\right) 
        \\
        =& \sum_{i = 1}^{i = n}\frac{\partial \tilde{f}}{\partial x_i}\left(  \bm{t}^{\star} \right) \left(\sum_{j = 1}^{j = m}c_{i, j} \left(t_{j} - t_{j}^{\star} \right)\right) = \sum_{j = 1}^{j = m}\sum_{i = 1}^{i = n} \frac{\partial \tilde{f}}{\partial x_{i}} (c_{i, j}) (\bm{t}^{\star}) (t_j - t_j^{\star}),
    \end{align*}
    which is equivalent to the second term in \eqref{eq:aux_taylor}. 
\end{proof}
Claim \ref{claim:add_var_equal} follows directly from Proposition~\ref{prop:aux_val_equality} and the observation that in an optimal solution of CRM $\tilde{p}_{0,t}^b$ will be equal to either \eqref{eq:addition_var_sub_const_p} or \eqref{eq:addition_var_sub_const_n} and   $\tilde{p}_{i,t}^b$ will be equal to either \eqref{eq:addition_var_bu_const_p} or \eqref{eq:addition_var_bu_const_n}, respectively. $\blacksquare$

\newpage
\section{IEEE 33 Bus Test System Layout} \label{appendix:sys_layout}
\begin{enumerate}
    \item \textbf{DGs}: Bus 15, Bus 16, Bus 17, Bus 18, Bus 21, Bus 23, Bus 24, Bus 26, Bus 30
    \item \textbf{BESSs}: Bus 17, Bus 18, Bus 23, Bus 24
    \item \textbf{DR}: Bus 2, Bus 3, Bus 5
\end{enumerate}
\begin{figure}[htbp]
\centering
\begin{tikzpicture}[>={Latex[width=2mm,length=2mm]},
                    bus/.style={circle, draw, text centered, minimum size=0.75cm, inner sep=2pt},
                    unit/.style={align=center, font=\scriptsize, execute at begin node=\setlength{\baselineskip}{10pt}},
                    scale=0.5, every node/.style={scale=0.5}]

  \node[bus] (bus0) {Substation};
  \node[bus, below=0.4cm of bus0] (bus1) {1};
  \node[bus, below=0.4cm of bus1] (bus2) {2};
  \node[bus, below=0.4cm of bus2] (bus3) {3};
  \node[bus, below=0.4cm of bus3] (bus4) {4};
  \node[bus, below=0.4cm of bus4] (bus5) {5};
  \node[bus, below=0.4cm of bus5] (bus6) {6};
  \node[bus, below=0.4cm of bus6] (bus7) {7};
  \node[bus, below=0.4cm of bus7] (bus8) {8};
  \node[bus, below=0.4cm of bus8] (bus9) {9};
  \node[bus, below=0.4cm of bus9] (bus10) {10};
  \node[bus, below=0.4cm of bus10] (bus11) {11};
  \node[bus, below=0.4cm of bus11] (bus12) {12};
  \node[bus, below=0.4cm of bus12] (bus13) {13};
  \node[bus, below=0.4cm of bus13] (bus14) {14};
  \node[bus, below=0.4cm of bus14] (bus15) {15};
  \node[bus, below=0.4cm of bus15] (bus16) {16};
  \node[bus, below=0.4cm of bus16] (bus17) {17};

  \node[bus, left=3.5cm of bus3] (bus22) {22};
  \node[bus, below=0.4cm of bus22] (bus23) {23};
  \node[bus, below=0.4cm of bus23] (bus24) {24};

  \node[bus, left=2cm of bus6] (bus25) {25};
  \node[bus, below=0.4cm of bus25] (bus26) {26};
  \node[bus, below=0.4cm of bus26] (bus27) {27};
  \node[bus, below=0.4cm of bus27] (bus28) {28};
  \node[bus, below=0.4cm of bus28] (bus29) {29};
  \node[bus, below=0.4cm of bus29] (bus30) {30};
  \node[bus, below=0.4cm of bus30] (bus31) {31};
  \node[bus, below=0.4cm of bus31] (bus32) {32};

  \node[bus, right=3cm of bus2] (bus18) {18};
  \node[bus, below=0.4cm of bus18] (bus19) {19};
  \node[bus, below=0.4cm of bus19] (bus20) {20};
  \node[bus, below=0.4cm of bus20] (bus21) {21};

  \draw[->] (bus0) -- (bus1);
  \draw[->] (bus1) -- (bus2);
  \draw[->] (bus2) -- (bus3);
  \draw[->] (bus3) -- (bus4);
  \draw[->] (bus4) -- (bus5);
  \draw[->] (bus5) -- (bus6);
  \draw[->] (bus6) -- (bus7);
  \draw[->] (bus7) -- (bus8);
  \draw[->] (bus8) -- (bus9);
  \draw[->] (bus9) -- (bus10);
  \draw[->] (bus10) -- (bus11);
  \draw[->] (bus11) -- (bus12);
  \draw[->] (bus12) -- (bus13);
  \draw[->] (bus13) -- (bus14);
  \draw[->] (bus14) -- (bus15);
  \draw[->] (bus15) -- (bus16);
  \draw[->] (bus16) -- (bus17);
  
  \draw[->] (bus2) -| (bus22);
  \draw[->] (bus22) -- (bus23);
  \draw[->] (bus23) -- (bus24);

  \draw[->] (bus5) -| (bus25);
  \draw[->] (bus25) -- (bus26);
  \draw[->] (bus26) -- (bus27);
  \draw[->] (bus27) -- (bus28);
  \draw[->] (bus28) -- (bus29);
  \draw[->] (bus29) -- (bus30);
  \draw[->] (bus30) -- (bus31);
  \draw[->] (bus31) -- (bus32);

  \draw[->] (bus1) -| (bus18);
  \draw[->] (bus18) -- (bus19);
  \draw[->] (bus19) -- (bus20);
  \draw[->] (bus20) -- (bus21);

  \node[unit, right=0.05cm of bus2]{DR};
  \node[unit, right=0.05cm of bus3]{DR};
  \node[unit, right=0.05cm of bus5]{DR};
  \node[unit, right=0.05cm of bus15]{DG};
  \node[unit, right=0.05cm of bus16]{DG};
  \node[unit, right=0.05cm of bus17]{DG\\Battery};
  \node[unit, right=0.05cm of bus18]{DG\\Battery};
  \node[unit, right=0.05cm of bus21]{DG};
  \node[unit, left=0.05cm of bus23]{DG\\Battery};
  \node[unit, left=0.05cm of bus24]{DG\\Battery};
  \node[unit, right=0.05cm of bus26]{DG};
  \node[unit, right=0.05cm of bus30]{DG};
  
\end{tikzpicture}
\caption{33 Bus Power System Radial Network}
\label{fig:33_bus_power_radial_system}
\end{figure}
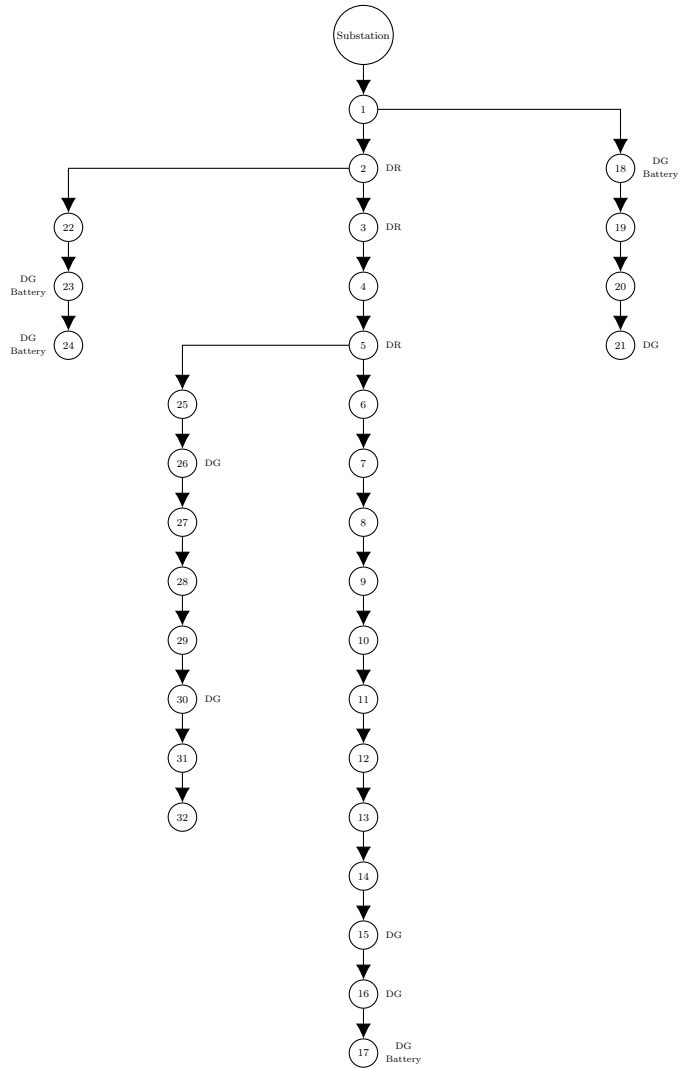

\newpage
\section{Parameter Values} \label{appendix:para_values}
\begin{table}[htbp]
    \centering
    \resizebox{\textwidth}{!}{
        \begin{tabular}{cc}
            \hline
            \textbf{Parameter} & \textbf{Value} \\
            \hline
            $C_t^{\rm sub}$, $\bm{C_{i, t}} = (c_{i, t}^{\rm DG}, c_{i, t}^{\rm B, c}, c_{i, t}^{\rm B, d}, c_{i, t}^{\rm DR})$; $i \in \mathcal{N}^+$, $t \in \mathcal{T}$ & $\$50$, $(\$8, -\$15, \$28, \$100)$  \\
            $P_{i}^{\text{DG, min}}$, $i = 15, 16, 17, 18, 21, 23, 24, 26, 30$ & $(0.000, 0.000, 0.000, 0.000, 0.000, 0.000, 0.000, 0.000, 0.000)$ p.u. \\
            $Q_{i}^{\text{DG, min}}$, $i = 15, 16, 17, 18, 21, 23, 24, 26, 30$ & $(0.000, 0.000, 0.000, 0.000, 0.000, 0.000, 0.000, 0.000, 0.000)$ p.u. \\
            $P_{i}^{\text{DG, max}}$, $i = 15, 16, 17, 18, 21, 23, 24, 26, 30$ & $(0.400, 0.240, 0.100, 0.100, 0.100, 0.220, 0.400, 0.150, 0.150)$ p.u. \\
            $Q_{i}^{\text{DG, max}}$, $i = 15, 16, 17, 18, 21, 23, 24, 26, 30$ & $(0.600, 0.360, 0.150, 0.150, 0.150, 0.330, 0.600, 0.225, 0.225)$ p.u. \\
            $P_{i}^{\text{B, min}}$, $Q_{i}^{\text{B, min}}$; $i = 17, 18, 23, 24$ & $(0.000, 0.000, 0.000, 0.000)$, $(0.000, 0.000, 0.000, 0.000)$p.u. \\
            $P_{i}^{\text{B, max}}$, $Q_{i}^{\text{B, max}}$; $i = 17, 18, 23, 24$ & $(0.240, 0.240, 0.240, 0.240)$, $(0.360, 0.360, 0.360, 0.360)$ p.u. \\
            $P_{i}^{\text{DR, min}}$, $Q_{i}^{\text{DR, min}}$; $i = 2, 3, 5$ & $(0.000, 0.000, 0.000)$, $(0.000, 0.000, 0.000)$ p.u. \\
            $P_{i}^{\text{DR, max}}$, $Q_{i}^{\text{DR, max}}$; $i = 2, 3, 5$ & $(0.240, 0.240, 0.240)$, $(0.360, 0.360, 0.360)$ p.u. \\
            $\delta_i$, $\eta_i^{\rm B, c}$, $\eta_i^{\rm B, d}$; $i = 17, 18, 23, 24$ & $(0\%, 0\%, 0\%, 0\%)$, $(100\%, 100\%, 100\%, 100\%)$, $(100\%, 100\%, 100\%, 100\%)$ \\
            ${\rm {SOC}}_{i}^{\rm B, min}$, ${\rm {SOC}}_{i}^{\rm B, max}$; $i = 17, 18, 23, 24$ & $(0\%, 0\%, 0\%, 0\%)$, $(100\%, 100\%, 100\%, 100\%)$ \\
            $w_{\rm sub}$, $\bm{w_{{b}_i}} = (w_{i}^{\rm c}, w_{i}^{\rm DG}, w_{i}^{\rm B, c}, w_{i}^{\rm B, d}, w_{i}^{\rm DR})$; $i \in \mathcal{N}^{+}$ & $\$1 \times 10^{5}$, $(\$1\times{10}^5, \$2\times{10}^4, \$2\times{10}^4, \$2\times{10}^4, \$2\times{10}^4)$ \\
            $\rm{AT}_t$; $t \in \mathcal{T}$ & $(21.5, 19.3, 17.9, 16.3, 15.2, 15.7, 18.2, 22.0, 23.8, 25.3, 26.1, 25.8)$ $\text{C}^{\circ}$ \\
            $V_{0}^{\text{con}}$, $V_{i}^{\text{min}}$, $V_{i}^{\text{max}}$, $G_j$, $B_j$; $\forall i \in \mathcal{N}^{+}$, $\forall j \in \mathcal{N}$ & $1.03, 0.9, 1.1, 0, 0$ p.u. \\
            \hline
        \end{tabular}
    }
    \caption{CM And CRM Parameter Values}
    \label{tab:para_val}
\end{table}

\begin{table}[htbp]
    \centering
    \resizebox{0.98\textwidth}{!}{
        \begin{tabular}{c||ccc||ccc}
            \hline
            \textbf{Bus/Line} $i$, $i \in \mathcal{N}$ & $\lambda_{b, i}$ & $\beta_{1, i}^{b}$ & $\beta_{2, i}^{b}$ & $\lambda_{l, i}$ & $\beta_{1, i}^{l}$ & $\beta_{2, i}^{l}$ \\ \hline
            0  & 1156113700 & 1.8307053 & 0.24685119 & -          & -           & -          \\ 
            1  & 268660350  & 0.30390543 & 0.30673215 & 460174750  & 0.9500556   & 0.24567656 \\ 
            2  & 298944960  & 0.2884359  & 0.31724855 & 344302800  & 1.0492594   & 0.25378057 \\ 
            3  & 242551490  & 0.44588602 & 0.311652   & 443482800  & 1.4735763   & 0.2638999  \\ 
            4  & 275446140  & 0.18725057 & 0.3119664  & 342151580  & 1.5987376   & 0.272045   \\ 
            5  & 262219680  & 0.17005566 & 0.30642742 & 318478430  & 1.7211572   & 0.26812083 \\ 
            6  & 174550750  & 0.70556873 & 0.2923882  & 196448900  & 2.2786012   & 0.31977552 \\ 
            7  & 623449400  & 0.33577964 & 0.32093257 & 272240220  & 0.9843094   & 0.3327228  \\ 
            8  & 177938450  & 0.22490926 & 0.29845428 & 487715700  & 4.2256393   & 0.34695134 \\ 
            9  & 188718510  & 0.21641213 & 0.3026553  & 437785540  & 6.201332    & 0.3500279  \\ 
            10 & 269642400  & 0.13316843 & 0.30742267 & 826891460  & 7.019942    & 0.362927   \\ 
            11 & 308211000  & 0.19764385 & 0.3135902  & 1070140400 & 5.5902643   & 0.36101142 \\ 
            12 & 265625570  & 0.18195678 & 0.3100992  & 1477688300 & 6.2913237   & 0.3642413  \\ 
            13 & 315843360  & 0.39239794 & 0.3171655  & 1354310000 & 5.034037    & 0.35175356 \\ 
            14 & 224416740  & 0.23627062 & 0.30976433 & 2211711700 & 5.59901     & 0.34484866 \\ 
            15 & 563844740  & 2.2013748  & 0.31617218 & 3357516000 & 5.608351    & 0.3432856  \\ 
            16 & 387065000  & 1.2012444  & 0.31643263 & 256633400  & 1.7178524   & 0.3328324  \\ 
            17 & 288559940  & 1.0218177  & 0.31374693 & 269417760  & 0.0966832   & 0.3339903  \\ 
            18 & 229609400  & 1.2618285  & 0.31427383 & 181219700  & 1.3357927   & 0.32394075 \\ 
            19 & 143127070  & 0.31744176 & 0.28802255 & 226100530  & 0.46914005  & 0.33014536 \\ 
            20 & 193695630  & 0.34026065 & 0.30224663 & 222501800  & 0.11520606  & 0.32999086 \\ 
            21 & 272977020  & 0.9444536  & 0.30839005 & 300525120  & -0.05748484 & 0.33137864 \\ 
            22 & 246455710  & 0.39065266 & 0.31751642 & 246405870  & 2.1500254   & 0.32351357 \\ 
            23 & 221162200  & 0.8296937  & 0.28750572 & 219469570  & 1.6886002   & 0.32250196 \\ 
            24 & 421450500  & 3.4731283  & 0.32212225 & 221565440  & 1.151456    & 0.32871062 \\ 
            25 & 265509070  & 0.18225878 & 0.30883035 & 548647550  & 3.6369984   & 0.2874418  \\ 
            26 & 230688850  & 0.6348009  & 0.29217732 & 364029570  & 3.6293347   & 0.28503045 \\ 
            27 & 307971230  & 0.1981767  & 0.31558347 & 916634240  & 2.8616748   & 0.29052076 \\ 
            28 & 275704770  & 0.35590678 & 0.30712327 & 648119200  & 3.3561606   & 0.28549427 \\ 
            29 & 281854530  & 0.5795351  & 0.30781204 & 489752540  & 3.7715204   & 0.29009366 \\ 
            30 & 420545500  & 1.5111128  & 0.32065573 & 265131570  & 1.0008259   & 0.331785   \\ 
            31 & 214335260  & 0.758895   & 0.3024018  & 233778380  & 1.2644929   & 0.32463384 \\ 
            32 & 312010700  & 0.14149183 & 0.30900797 & 214972750  & -0.08237231 & 0.32953462 \\ 
            \hline
        \end{tabular}
    }
    \caption{Logistics Regression Model Parameter For Buses and Lines}
    \label{tab:para_logistic}
\end{table}

\newpage
\section{Objective Values} \label{appendix:obj_values}
The experiments were conducted on a system running Windows $10$, equipped with an $13^{th}$ Gen Inter(R) Core(TM) i$7$-$1360$P.
\begin{table}[htbp]
    \centering
    \resizebox{0.98\textwidth}{!}{
        \begin{tabular}{ccccccc}
            \hline
            \textbf{Iteration} & \textbf{CRM} (\$) & \textbf{CRM-APPX} (\$) & $\epsilon_{1}$ & $\epsilon_{2}$ & $\epsilon_{3}$ & \textbf{Comp. Time (Sec.)} \\
            \hline
            0  & 956.1050    & 956.1050    & -          & -          & -  & 10.0695          \\ 
            1  & 370712.8418 & 370496.6175 & -          & 15750.1645 & -        & 1920.1130  \\ 
            2  & 358365.1071 & 358222.7163 & 1.1884     & 12503.1271 & 0.0419   & 1932.1140  \\ 
            3  & 348404.2667 & 348308.1215 & 0.9920     & 10067.8763 & 0.0345   & 1913.2482  \\ 
            4  & 340290.2623 & 340225.5969 & 0.8269     & 8187.8056  & 0.0286   & 1912.5347  \\ 
            5  & 333626.9679 & 333583.8371 & 0.6927     & 6714.1001  & 0.0238   & 1911.5029  \\ 
            6  & 328117.7579 & 328089.3415 & 0.5815     & 5544.0868  & 0.0200   & 1944.9554  \\ 
            7  & 323537.0030 & 323518.4921 & 0.4901     & 4604.7458  & 0.0168   & 1953.1037  \\ 
            8  & 319710.9668 & 319699.0511 & 0.4135     & 3842.6189  & 0.0142   & 1954.5051  \\ 
            9  & 316501.2830 & 316493.6924 & 0.3495     & 3221.2426  & 0.0120   & 1969.4355  \\ 
            10 & 313800.6919 & 313795.9107 & 0.2956     & 2708.7347  & 0.0101   & 1971.3271  \\ 
            11 & 311593.8632 & 311591.0762 & 0.2502     & 2212.1642  & 0.0086   & 1966.0825  \\ 
            12 & 309748.2508 & 309746.6010 & 0.1935     & 1849.3361  & 0.0071   & 1955.1917  \\ 
            13 & 308186.5547 & 308185.5645 & 0.1605     & 1564.4392  & 0.0060   & 1935.7529  \\ 
            14 & 306863.5288 & 306862.9400 & 0.1360     & 1325.0978  & 0.0051   & 2033.7460  \\ 
            15 & 305741.7644 & 305741.4179 & 0.1153     & 1123.3664  & 0.0043   & 1911.3439  \\ 
            16 & 304790.0182 & 304789.8167 & 0.0978     & 953.0115   & 0.0037   & 1913.8146  \\ 
            17 & 303981.0040 & 303980.8894 & 0.0831     & 810.0296   & 0.0031   & 1925.7290  \\ 
            18 & 303294.9766 & 303294.9119 & 0.0714     & 686.8544   & 0.0027   & 1938.8149  \\ 
            19 & 302712.3328 & 302712.2975 & 0.0598     & 583.3270   & 0.0023   & 1933.4018  \\ 
            20 & 302217.3771 & 302217.3586 & 0.0509     & 495.5236   & 0.0019   & 1928.3233  \\ 
            21 & 301796.8510 & 301796.8420 & 0.0432     & 421.0017   & 0.0016   & 1935.1277  \\ 
            22 & 301439.5258 & 301439.5219 & 0.0367     & 357.7250   & 0.0014   & 1941.6433  \\ 
            23 & 301135.8713 & 301135.8702 & 0.0312     & 303.9918   & 0.0012   & 1935.4866  \\ 
            24 & 300877.8191 & 300877.8192 & 0.0265     & 258.3376   & 0.0010   & 1944.5957  \\ 
            25 & 300658.5111 & 300658.5118 & 0.0225     & 219.5498   & 0.0009   & 1932.6275  \\ 
            26 & 300472.1048 & 300472.1056 & 0.0191     & 186.6114   & 0.0007   & 1903.0784  \\ 
            27 & 300313.6920 & 300313.6929 & 0.0163     & 158.5870   & 0.0006   & 1926.1540  \\ 
            28 & 300179.0551 & 300179.0558 & 0.0138     & 134.7849   & 0.0005   & 1963.3941  \\ 
            29 & 300064.6003 & 300064.6009 & 0.0117     & 114.5801   & 0.0004   & 1986.8101  \\ 
            30 & 299967.3305 & 299967.3309 & 0.0100     & 97.3767    & 0.0004   & 1948.3431  \\ 
            31 & 299884.6557 & 299884.6561 & 0.0085     & 82.7656    & 0.0003   & 1924.9493  \\ 
            32 & 299814.3811 & 299814.3814 & 0.0072     & 70.3516    & 0.0003   & 1939.8494  \\ 
            33 & 299755.1160 & 299755.1162 & 0.0061     & 59.3488    & 0.0002   & 1944.3648  \\ 
            34 & 299704.1494 & 299704.1496 & 0.0062     & 51.0108    & 0.0002   & 1942.8310  \\ 
            35 & 299660.9327 & 299660.9327 & 0.0044     & 43.2579    & 0.0002   & 2118.2168  \\ 
            36 & 299624.2656 & 299624.2658 & 0.0038     & 36.7070    & 0.0001   & 2089.3453  \\ 
            37 & 299593.0856 & 299593.0857 & 0.0032     & 31.2142    & 0.0001   & 1953.7750  \\ 
            38 & 299566.5859 & 299566.5860 & 0.0027     & 26.5288    & 0.0001   & 1973.3946  \\ 
            39 & 299544.0586 & 299544.0586 & 0.0023     & 22.5520    & 8.85E-05 & 2066.2396  \\ 
            40 & 299524.9668 & 299524.9668 & 0.0020     & 19.1173    & 7.52E-05 & 1973.9318  \\ 
            41 & 299508.5901 & 299508.5901 & 0.0018     & 16.3924    & 6.38E-05 & 1996.2207  \\ 
            42 & 299495.8235 & 299495.8236 & 0.0014     & 12.7977    & 5.46E-05 & 2431.9818  \\ 
            43 & 299483.5762 & 299483.5762 & 0.0023     & 12.2566    & 4.27E-05 & 2143.4372  \\ 
            44 & 299474.6961 & 299474.6960 & 0.0010     & 8.9116     & 4.08E-05 & 2004.5153  \\ 
            45 & 299466.2461 & 299466.2461 & 0.0022     & 8.4644     & 2.97E-05 & 2034.9476  \\ 
            46 & 299459.9347 & 299459.9347 & 0.0011     & 6.3344     & 2.82E-05 & 2038.9545  \\ 
            47 & 299454.5197 & 299454.5196 & 0.0017     & 5.4340     & 2.11E-05 & 2029.5761  \\ 
            48 & 299449.2458 & 299449.2457 & 0.0015     & 5.2806     & 1.81E-05 & 1977.5771  \\ 
            \hline
        \end{tabular}
    }
    \caption{Objective Values And Computation Time For Each Iteration}
    \label{tab:obj_time_iteration}
\end{table}

\newpage
\printbibliography[title={References}]

\end{document}

%% file: tikzfig_radial_network.tex
\begin{tikzpicture}[
        every node/.style={font=\small},
        bus/.style={draw, fill=black, thick},
        bus_b/.style={draw=blue, fill=blue, thick},
        block_c/.style={draw, fill=gray!20, minimum width=0.5cm, minimum height=4cm},
        block_a/.style={draw, fill=gray!20, minimum width=0.5cm, minimum height=1.5cm},
        bus_type/.style={draw=blue, fill=blue, minimum width=0.5cm, minimum height=0.25cm, text centered, text=white, align=center, font=\scriptsize, inner sep=0},
        substation/.style={draw, fill=black, thick, circle, minimum size=0.5cm, text=white, align=center, font=\scriptsize},
        set/.style={draw, fill=white, thick, circle, minimum size=0.4cm, text=black, align=center, font=\scriptsize, inner sep=0},
        bus_ori/.style={draw=blue, fill=white, thick, circle, minimum size=0.4cm, text=blue, align=center, font=\tiny, inner sep=0},
        arrow/.style={-Latex, thick},
        arrow_in/.style={-Stealth, thick, draw = blue},
        arrow_out/.style={-Stealth, thick, draw = purple},
        label/.style={font=\small}
        ]

        \pgfdeclarelayer{background}
        \pgfdeclarelayer{foreground}
        \pgfsetlayers{background,main,foreground}

        \node[label] at (-0.5, -2.85) {upstream};
        \node[label] at (4.5, -2.85) {downstream};
        \draw[arrow, thick, dash dot] (2.01, -2.5) -- (0, -2.5);
        \draw[arrow, thick, dash dot] (2.01, -2.5) -- (4.0, -2.5);
        \draw[dotted] (2.01, -2.3) -- (2.01, -2.7);

        \begin{pgfonlayer}{foreground}
            \draw[thick] (-2.7, 0) -- (-2.2, 0);
            \draw[thick, dashed] (-2.0, 0) -- (-1.4, 0);
            \draw[thick] (-1.2, 0) -- (2.5, 0);
            \draw[thick] (2.07, 0.25) -- (2.51, 0.25);
            \draw[thick] (2.5, 0.25) -- (4, 1.25);
            \draw[thick] (3.99, 1.25) -- (6.5, 1.25);
            \draw[thick, dashed] (6.5, 1.25) -- (7.5, 1.25);
            \draw[thick, dashed] (2.5, 0) -- (3.75, 0);
            \draw[thick] (2.07, -0.25) -- (2.51, -0.25);
            \draw[thick] (2.5, -0.25) -- (4, -1.25);
            \draw[thick] (3.99, -1.25) -- (6.5, -1.25);
            \draw[thick, dashed] (6.5, -1.25) -- (7.5, -1.25);
            \draw[thick, dotted] (4.55, -0.4) -- (4.55, 0.4);
        \end{pgfonlayer}

        \node[block_c, anchor=west] (block) at (4.3, 0) {};
        \node[block_a, anchor=west] (block) at (-0.75, 0) {};

        \node[substation] at (-3, 0){\textbf{Sub}};
        \node[bus, anchor=west, minimum width=0.1cm, minimum height=1.0cm, inner sep=0] (bus) at (-0.55, 0){};
        \node[bus_b, anchor=west, minimum width=0.1cm, minimum height=1.5cm, inner sep=0] (bus) at (1.95, 0){};
        \node[bus, anchor=west, minimum width=0.1cm, minimum height=1.0cm, inner sep=0] (bus) at (4.5, -1.25){};
        \node[bus, anchor=west, minimum width=0.1cm, minimum height=1.0cm, inner sep=0] (bus) at (4.5, 1.25){};
        \node[bus, anchor=west, minimum width=0.1cm, minimum height=1.0cm, inner sep=0] (bus) at (6.0, 1.25){};

        \node[bus_type] at (0.15, 1.75){\textbf{DG}};
        \node[bus_type] at (1, 1.75){\textbf{DR}};
        \node[bus_type] at (2, 1.75){\textbf{BESS}};

        \begin{pgfonlayer}{foreground}
            \node[set] at (-0.75, -0.75){$\mathcal{A}_i$};
            \node[bus_ori] at (1.6, -0.75){$\textbf{Bus}_i$};
            \node[set] at (4.25, -2){$\mathcal{S}c_i$};
        \end{pgfonlayer}

        \draw[arrow_in, dotted] (0.5, 0.25) -- (1.94, 0.25);
        \draw[thick, dotted, purple] (2.1, 0.125) -- (2.51, 0.125);
        \draw[arrow_out, dotted] (3.99, 1.125) -- (4.5, 1.125);
        \draw[thick, dotted, purple] (2.51, 0.125) -- (4.0, 1.125);
        \draw[thick, dotted, purple] (2.1, -0.125) -- (2.51, -0.125);
        \draw[arrow_out, dotted] (3.99, -1.125) -- (4.5, -1.125);
        \draw[thick, dotted, purple] (2.51, -0.125) -- (4.0, -1.125);

        \draw[arrow_in, dotted] (1.96, 1.60) -- (1.96, 0.75);
        \draw[arrow_out, dotted] (2.04, 0.75) -- (2.04, 1.60);
        \draw[thick, dotted, blue] (0.15, 0.45) -- (0.15, 1.75);
        \draw[arrow_in, dotted] (0.15, 0.45) -- (1.96, 0.45);
        \draw[thick, dotted, blue] (1, 0.65) -- (1, 1.75);
        \draw[arrow_in, dotted] (1, 0.65) -- (1.96, 0.65);

        \draw[arrow_out, dotted] (3, -0.5) -- (3, -1.5);
        \draw[arrow_out, dotted] (3.25, 0.6) -- (3.25, -1.5);
        \draw[arrow_out, dotted] (2.01, -0.75) -- (2.01, -1.5);
        \draw[thick, dotted, purple] (0.75, -0.125) -- (1.95, -0.125);
        \draw[arrow_out, dotted] (0.75, -0.125) -- (0.75, -1.5);

        \node[label] at (-0.2, 1.35) {\textcolor{blue}{$p_i^{\text{DG}}$}};
        \node[label] at (0.675, 1.35) {\textcolor{blue}{$p_i^{\text{DR}}$}};
        \node[label] at (1.65, 1.35) {\textcolor{blue}{$p_i^{\text{B,d}}$}};
        \node[label] at (0.3, 0.25) {\textcolor{blue}{$f_i^{p}$}};
        \node[label] at (2.45, 1.35) {\textcolor{purple}{$p_i^{\text{B,c}}$}};
        \node[label] at (4, 0.8) {\textcolor{purple}{$f_{\mathcal{S}c_i}^{p}$}};
        \node[label] at (4, -0.8) {\textcolor{purple}{$f_{\mathcal{S}c_i}^{p}$}};
        \node[label] at (0.75, -1.7) {\textcolor{purple}{$p_{i}^{c}$}};
        \node[label] at (2.0, -1.7) {\textcolor{purple}{$G_{i}v_{i}$}};
        \node[label] at (3.3, -1.7) {\textcolor{purple}{$l_{\mathcal{S}c_i}R_{\mathcal{S}c_i}$}};
    \end{tikzpicture}

%% file: tikzfig_reliability_model.tex
\begin{tikzpicture}[
        every node/.style={font=\small},
        bus/.style={draw, fill=black, thick},
        bus_b/.style={draw=blue, fill=blue, thick},
        substation/.style={draw, fill=black, thick, circle, minimum size=0.5cm, text=white, align=center, font=\scriptsize},
        arrow/.style={-Latex, thick},
        label/.style={font=\small},
        label_tiny/.style={font=\scriptsize, fill = lightblue, inner sep = 0},
        label_prob_r/.style={font=\small, text=darkred},
        label_prob_g/.style={font=\small, text=darkgreen}
        ]

        \pgfdeclarelayer{background}
        \pgfdeclarelayer{foreground}
        \pgfsetlayers{background,main,foreground}

        \node[label] at (2.2, -2.35) {upstream};
        \node[label] at (7.5, -2.35) {downstream};
        \draw[arrow, thick, dash dot] (6, -2) -- (2, -2);
        \draw[arrow, thick, dash dot] (6, -2) -- (7.5, -2);
        \draw[dotted] (6.06, -1.8) -- (6.06, -2.2);

        \begin{pgfonlayer}{background}
            \draw[thick] (-3, 0) -- (-2.15, 0);
            \draw[thick, dashed] (-1.925, 0) -- (-1, 0);
            \draw[thick] (-0.75, 0) -- (6.75, 0);
            \draw[thick, dashed] (6.75, 0) -- (8.5, 0);
        \end{pgfonlayer}

        \begin{pgfonlayer}{main}
            \node[substation] at (-3, 0){\textbf{Sub}};
            \node[bus, anchor=west, minimum width=0.1cm, minimum height=0.75cm, inner sep=0] (bus) at (0, 0){};
            \node[bus, anchor=west, minimum width=0.1cm, minimum height=0.75cm, inner sep=0] (bus) at (3, 0){};
            \node[bus_b, anchor=west, minimum width=0.1cm, minimum height=1cm, inner sep=0] (bus) at (6, 0){};

            \node[label_tiny] at (0, -0.625) {\textcolor{black}{$\text{Bus}_{i-2}$}};
            \node[label_tiny] at (3, -0.625) {\textcolor{black}{$\text{Bus}_{i-1}$}};
            \node[label_tiny] at (6.1, -0.75) {\textcolor{blue}{$\text{Bus}_{i}$}};
        \end{pgfonlayer}

        \begin{pgfonlayer}{foreground}
            \draw[line width=0.5mm, draw = red] (-3, 0.25) -- (-2.75, 0.5);
            \draw[line width=0.5mm, draw = red] (-3, 0.5) -- (-2.75, 0.25);
            \draw[line width=0.5mm, draw = red] (0, 0.25) -- (0.25, 0.5);
            \draw[line width=0.5mm, draw = red] (0, 0.5) -- (0.25, 0.25);
            \draw[line width=0.5mm, draw = red] (3, 0.25) -- (3.25, 0.5);
            \draw[line width=0.5mm, draw = red] (3, 0.5) -- (3.25, 0.25);
            \draw[line width=0.5mm, draw = red] (6, 0.375) -- (6.25, 0.625);
            \draw[line width=0.5mm, draw = red] (6, 0.625) -- (6.25, 0.374);

            \draw[line width=0.5mm, draw = red] (1.75, -0.125) -- (2, 0.125);
            \draw[line width=0.5mm, draw = red] (1.75, 0.125) -- (2, -0.125);
            \draw[line width=0.5mm, draw = red] (4.75, -0.125) -- (5, 0.125);
            \draw[line width=0.5mm, draw = red] (4.75, 0.125) -- (5, -0.125);
        \end{pgfonlayer}

        \begin{pgfonlayer}{foreground}
            \draw[line width=0.5mm, draw = green] (-3.125, -0.375) -- (-3, -0.5);
            \draw[line width=0.5mm, draw = green] (-3.025, -0.5) -- (-2.875, -0.25);
            \draw[line width=0.5mm, draw = green] (-0.125, -0.375) -- (0, -0.5);
            \draw[line width=0.5mm, draw = green] (-0.025, -0.5) -- (0.125, -0.25);
            \draw[line width=0.5mm, draw = green] (2.875, -0.375) -- (3, -0.5);
            \draw[line width=0.5mm, draw = green] (2.975, -0.5) -- (3.125, -0.25);
            \draw[line width=0.5mm, draw = green] (5.875, -0.45) -- (6, -0.575);
            \draw[line width=0.5mm, draw = green] (5.975, -0.575) -- (6.125, -0.325);

            \draw[line width=0.5mm, draw = green] (1.125, 0) -- (1.25, -0.125);
            \draw[line width=0.5mm, draw = green] (1.225, -0.125) -- (1.375, 0.125);
            \draw[line width=0.5mm, draw = green] (4.125, 0) -- (4.25, -0.125);
            \draw[line width=0.5mm, draw = green] (4.225, -0.125) -- (4.375, 0.125);
        \end{pgfonlayer}

        \draw[thick, draw = darkred] (-2.875, 0.375) -- (-2.875, 1);
        \draw[thick, draw = darkred] (-2.875, 1) -- (-2.5, 1);
        \node[label_prob_r] at (-2.1, 1) {$\text{Pr}_{0,t}^{\text{b}}$};
        \draw[thick, draw = darkred] (0.125, 0.375) -- (0.125, 1);
        \draw[thick, draw = darkred] (0.125, 1) -- (0.5, 1);
        \node[label_prob_r] at (1, 1) {$\text{Pr}_{i \!- \!2,t}^{\text{b}}$};
        \draw[thick, draw = darkred] (3.125, 0.375) -- (3.125, 1);
        \draw[thick, draw = darkred] (3.125, 1) -- (3.5, 1);
        \node[label_prob_r] at (4, 1) {$\text{Pr}_{i \!- \!1,t}^{\text{b}}$};
        \draw[thick, draw = darkred] (6.125, 0.5) -- (6.125, 1);
        \draw[thick, draw = darkred] (6.125, 1) -- (6.5, 1);
        \node[label_prob_r] at (7, 1) {$\text{Pr}_{i,t}^{\text{b}}$};

        \draw[thick, draw = darkred] (1.875, 0) -- (1.875, 0.3);
        \node[label_prob_r] at (1.875, 0.5) {$\text{Pr}_{i \!- \!1,t}^{\text{l}}$};
        \draw[thick, draw = darkred] (4.875, 0) -- (4.875, 0.3);
        \node[label_prob_r] at (4.875, 0.5) {$\text{Pr}_{i,t}^{\text{l}}$};

        \draw[thick, draw = darkgreen] (-0.2, -0.4) -- (-1, -0.4);
        \draw[thick, draw = darkgreen] (-1, -0.4) -- (-1, -1);
        \node[label_prob_g] at (-1, -1.2) {$\left( \!1 \! \!- \! \!\text{Pr}_{i \!- \!2,t}^{\text{b}} \!\right)$};
        \draw[thick, draw = darkgreen] (2.8, -0.4) -- (2, -0.4);
        \draw[thick, draw = darkgreen] (2, -0.4) -- (2, -1);
        \node[label_prob_g] at (2, -1.2) {$\left( \!1 \! \!- \! \!\text{Pr}_{i \!- \!1,t}^{\text{b}} \!\right)$};
        \draw[thick, draw = darkgreen] (5.8, -0.5) -- (5, -0.5);
        \draw[thick, draw = darkgreen] (5, -0.5) -- (5, -1);
        \node[label_prob_g] at (5, -1.2) {$\left( \! 1  \!\! -  \!\! \text{Pr}_{i,t}^{\text{b}}\!\right)$};

        \node[label_prob_g] at (1.2, -0.375) {$\left( \!1 \! \!- \!\!\text{Pr}_{i \!- \!1,t}^{\text{l}} \!\right)$};
        \node[label_prob_g] at (4.2, -0.375) {$\left( \!1 \! \!- \! \!\text{Pr}_{i,t}^{\text{l}} \!\right)$};
        
    \end{tikzpicture}

%% file: tikzfig_radial_network_working_a.tex
        \begin{tikzpicture}[
        every node/.style={font=\small},
        bus/.style={draw, fill=black, thick},
        bus_b/.style={draw=darkgreen, fill=darkgreen, thick},
        substation/.style={draw, fill=black, thick, circle, minimum size=0.5cm, text=white, align=center, font=\scriptsize},
        arrow/.style={-Latex, thick},
        label/.style={font=\small},
        label_tiny/.style={font=\scriptsize, fill = lightblue, inner sep = 0},
        label_prob_r/.style={font=\small, text=darkred},
        label_prob_g/.style={font=\small, text=darkgreen}
        ]

        \pgfdeclarelayer{background}
        \pgfdeclarelayer{foreground}
        \pgfsetlayers{background,main,foreground}

        \begin{pgfonlayer}{background}
            \draw[thick, draw = darkgreen] (-3, 0) -- (-2.35, 0);
            \draw[thick, dashed, draw = darkgreen] (-2.35, 0) -- (-1.85, 0);
            \draw[thick, draw = darkgreen] (-1.85, 0) -- (-1.05, 0);
            \draw[thick, dashed, draw = darkgreen] (-1.05, 0) -- (-0.4, 0);
            \draw[thick, draw = darkgreen] (-0.4, 0) -- (0.5, 0);
            \draw[thick, dashed, draw = darkgreen] (0.5, 0) -- (1.15, 0);
            \draw[thick, draw = darkgreen] (1.15, 0) -- (2, 0);
            \draw[thick, dashed, draw = darkgreen] (2, 0) -- (2.5, 0);
        \end{pgfonlayer}

        \begin{pgfonlayer}{main}
            \node[substation] at (-3, 0){\textbf{Sub}};
            \node[bus, anchor=west, minimum width=0.1cm, minimum height=0.75cm, inner sep=0] (bus) at (-1.5, 0){};
            \node[bus, anchor=west, minimum width=0.1cm, minimum height=0.75cm, inner sep=0] (bus) at (0, 0){};
            \node[bus_b, anchor=west, minimum width=0.1cm, minimum height=1cm, inner sep=0] (bus) at (1.5, 0){};

            \node[label_tiny] at (1.6, -0.75) {\textcolor{blue}{$\text{Bus}_{i}$}};
        \end{pgfonlayer}


        \begin{pgfonlayer}{foreground}
            \draw[line width=0.5mm, draw = green] (-3.125, -0.375) -- (-3, -0.5);
            \draw[line width=0.5mm, draw = green] (-3.025, -0.5) -- (-2.875, -0.25);
            \draw[line width=0.5mm, draw = green] (-1.625, -0.375) -- (-1.5, -0.5);
            \draw[line width=0.5mm, draw = green] (-1.525, -0.5) -- (-1.375, -0.25);
            \draw[line width=0.5mm, draw = green] (-0.125, -0.375) -- (0, -0.5);
            \draw[line width=0.5mm, draw = green] (-0.025, -0.5) -- (0.125, -0.25);
            \draw[line width=0.5mm, draw = green] (1.375, -0.45) -- (1.5, -0.575);
            \draw[line width=0.5mm, draw = green] (1.475, -0.575) -- (1.625, -0.325);

            \draw[line width=0.5mm, draw = green] (-1.875, 0) -- (-1.75, -0.125);
            \draw[line width=0.5mm, draw = green] (-1.775, -0.125) -- (-1.625, 0.125);
            \draw[line width=0.5mm, draw = green] (-0.375, 0) -- (-0.25, -0.125);
            \draw[line width=0.5mm, draw = green] (-0.275, -0.125) -- (-0.125, 0.125);
            \draw[line width=0.5mm, draw = green] (1.125, 0) -- (1.25, -0.125);
            \draw[line width=0.5mm, draw = green] (1.225, -0.125) -- (1.375, 0.125);
        \end{pgfonlayer}
        \end{tikzpicture}

%% file: tikzfig_radial_network_working_b.tex
\begin{tikzpicture}[
        every node/.style={font=\small},
        bus/.style={draw, fill=black, thick},
        bus_b/.style={draw=darkgreen, fill=darkgreen, thick},
        substation/.style={draw, fill=black, thick, circle, minimum size=0.5cm, text=white, align=center, font=\scriptsize},
        arrow/.style={-Latex, thick},
        label/.style={font=\small},
        label_tiny/.style={font=\scriptsize, fill = lightblue, inner sep = 0},
        label_prob_r/.style={font=\small, text=darkred},
        label_prob_g/.style={font=\small, text=darkgreen}
        ]

        \pgfdeclarelayer{background}
        \pgfdeclarelayer{foreground}
        \pgfsetlayers{background,main,foreground}

        \begin{pgfonlayer}{background}
            \draw[thick, draw = darkgreen] (-3, 0) -- (-2.35, 0);
            \draw[thick, dashed, draw = darkgreen] (-2.35, 0) -- (-1.85, 0);
            \draw[thick, draw = darkgreen] (-1.85, 0) -- (-1.05, 0);
            \draw[thick, dashed, draw = darkgreen] (-1.05, 0) -- (-0.4, 0);
            \draw[thick, draw = darkgreen] (-0.4, 0) -- (0.5, 0);
            \draw[thick, dashed, draw = darkgreen] (0.5, 0) -- (1.15, 0);
            \draw[thick, draw = darkgreen] (1.15, 0) -- (2, 0);
            \draw[thick, dashed, draw = darkgreen] (2, 0) -- (2.5, 0);
        \end{pgfonlayer}

        \begin{pgfonlayer}{main}
            \node[substation] at (-3, 0){\textbf{Sub}};
            \node[bus, anchor=west, minimum width=0.1cm, minimum height=0.75cm, inner sep=0] (bus) at (-1.5, 0){};
            \node[bus, anchor=west, minimum width=0.1cm, minimum height=0.75cm, inner sep=0] (bus) at (0, 0){};
            \node[bus_b, anchor=west, minimum width=0.1cm, minimum height=1cm, inner sep=0] (bus) at (1.5, 0){};

            \node[label_tiny] at (1.6, -0.75) {\textcolor{blue}{$\text{Bus}_{i}$}};
        \end{pgfonlayer}

        \begin{pgfonlayer}{foreground}
            \draw[line width=0.5mm, draw = red] (-1.5, 0.25) -- (-1.25, 0.5);
            \draw[line width=0.5mm, draw = red] (-1.5, 0.5) -- (-1.25, 0.25);
            \draw[line width=0.5mm, draw = red] (0, 0.25) -- (0.25, 0.5);
            \draw[line width=0.5mm, draw = red] (0, 0.5) -- (0.25, 0.25);
        \end{pgfonlayer}

        \begin{pgfonlayer}{foreground}
            \draw[line width=0.5mm, draw = green] (-3.125, -0.375) -- (-3, -0.5);
            \draw[line width=0.5mm, draw = green] (-3.025, -0.5) -- (-2.875, -0.25);
            \draw[line width=0.5mm, draw = green] (1.375, -0.45) -- (1.5, -0.575);
            \draw[line width=0.5mm, draw = green] (1.475, -0.575) -- (1.625, -0.325);

            \draw[line width=0.5mm, draw = green] (-1.875, 0) -- (-1.75, -0.125);
            \draw[line width=0.5mm, draw = green] (-1.775, -0.125) -- (-1.625, 0.125);
            \draw[line width=0.5mm, draw = green] (-0.375, 0) -- (-0.25, -0.125);
            \draw[line width=0.5mm, draw = green] (-0.275, -0.125) -- (-0.125, 0.125);
            \draw[line width=0.5mm, draw = green] (1.125, 0) -- (1.25, -0.125);
            \draw[line width=0.5mm, draw = green] (1.225, -0.125) -- (1.375, 0.125);
        \end{pgfonlayer}
        \end{tikzpicture}

%% file: tikzfig_radial_network_failure_a.tex
\begin{tikzpicture}[
        every node/.style={font=\small},
        bus/.style={draw, fill=black, thick},
        bus_b/.style={draw=darkred, fill=darkred, thick},
        substation/.style={draw, fill=black, thick, circle, minimum size=0.5cm, text=white, align=center, font=\scriptsize},
        arrow/.style={-Latex, thick},
        label/.style={font=\small},
        label_tiny/.style={font=\scriptsize, fill = lightblue, inner sep = 0},
        label_prob_r/.style={font=\small, text=darkred},
        label_prob_g/.style={font=\small, text=darkgreen}
        ]

        \pgfdeclarelayer{background}
        \pgfdeclarelayer{foreground}
        \pgfsetlayers{background,main,foreground}

        \begin{pgfonlayer}{background}
            \draw[thick, draw = darkgreen] (-3, 0) -- (-2.35, 0);
            \draw[thick, dashed, draw = darkgreen] (-2.35, 0) -- (-1.85, 0);
            \draw[thick, draw = darkgreen] (-1.85, 0) -- (-1.05, 0);
            \draw[thick, dashed, draw = darkgreen] (-1.05, 0) -- (-0.4, 0);
            \draw[thick, draw = darkgreen] (-0.4, 0) -- (0.5, 0);
            \draw[thick, dashed, draw = darkred] (0.5, 0) -- (1.15, 0);
            \draw[thick, draw = darkred] (1.15, 0) -- (2, 0);
            \draw[thick, dashed, draw = darkred] (2, 0) -- (2.5, 0);
        \end{pgfonlayer}

        \begin{pgfonlayer}{main}
            \node[substation] at (-3, 0){\textbf{Sub}};
            \node[bus, anchor=west, minimum width=0.1cm, minimum height=0.75cm, inner sep=0] (bus) at (-1.5, 0){};
            \node[bus, anchor=west, minimum width=0.1cm, minimum height=0.75cm, inner sep=0] (bus) at (0, 0){};
            \node[bus_b, anchor=west, minimum width=0.1cm, minimum height=1cm, inner sep=0] (bus) at (1.5, 0){};

            \node[label_tiny] at (1.6, -0.75) {\textcolor{blue}{$\text{Bus}_{i}$}};
        \end{pgfonlayer}


        \begin{pgfonlayer}{foreground}
            \draw[line width=0.5mm, draw = green] (-3.125, -0.375) -- (-3, -0.5);
            \draw[line width=0.5mm, draw = green] (-3.025, -0.5) -- (-2.875, -0.25);
            \draw[line width=0.5mm, draw = green] (-1.625, -0.375) -- (-1.5, -0.5);
            \draw[line width=0.5mm, draw = green] (-1.525, -0.5) -- (-1.375, -0.25);
            \draw[line width=0.5mm, draw = green] (-0.125, -0.375) -- (0, -0.5);
            \draw[line width=0.5mm, draw = green] (-0.025, -0.5) -- (0.125, -0.25);
            \draw[line width=0.5mm, draw = green] (1.375, -0.45) -- (1.5, -0.575);
            \draw[line width=0.5mm, draw = green] (1.475, -0.575) -- (1.625, -0.325);

            \draw[line width=0.5mm, draw = green] (-1.875, 0) -- (-1.75, -0.125);
            \draw[line width=0.5mm, draw = green] (-1.775, -0.125) -- (-1.625, 0.125);
            \draw[line width=0.5mm, draw = green] (-0.375, 0) -- (-0.25, -0.125);
            \draw[line width=0.5mm, draw = green] (-0.275, -0.125) -- (-0.125, 0.125);

            \draw[line width=0.5mm, draw = red] (1.15, -0.125) -- (1.4, 0.125);
            \draw[line width=0.5mm, draw = red] (1.15, 0.125) -- (1.4, -0.125);
        \end{pgfonlayer}
    \end{tikzpicture}

%% file: tikzfig_radial_network_failure_b.tex
\begin{tikzpicture}[
        every node/.style={font=\small},
        bus/.style={draw, fill=black, thick},
        bus_b/.style={draw=darkred, fill=darkred, thick},
        substation/.style={draw, fill=black, thick, circle, minimum size=0.5cm, text=white, align=center, font=\scriptsize},
        arrow/.style={-Latex, thick},
        label/.style={font=\small},
        label_tiny/.style={font=\scriptsize, fill = lightblue, inner sep = 0},
        label_prob_r/.style={font=\small, text=darkred},
        label_prob_g/.style={font=\small, text=darkgreen}
        ]

        \pgfdeclarelayer{background}
        \pgfdeclarelayer{foreground}
        \pgfsetlayers{background,main,foreground}

        \begin{pgfonlayer}{background}
            \draw[thick] (-3, 0) -- (-2.35, 0);
            \draw[thick, dashed] (-2.35, 0) -- (-1.85, 0);
            \draw[thick] (-1.85, 0) -- (-1.05, 0);
            \draw[thick, dashed] (-1.05, 0) -- (-0.4, 0);
            \draw[thick] (-0.4, 0) -- (0.5, 0);
            \draw[thick, dashed] (0.5, 0) -- (1.15, 0);
            \draw[thick] (1.15, 0) -- (2, 0);
            \draw[thick, dashed] (2, 0) -- (2.5, 0);
        \end{pgfonlayer}

        \begin{pgfonlayer}{main}
            \node[substation] at (-3, 0){\textbf{Sub}};
            \node[bus, anchor=west, minimum width=0.1cm, minimum height=0.75cm, inner sep=0] (bus) at (-1.5, 0){};
            \node[bus, anchor=west, minimum width=0.1cm, minimum height=0.75cm, inner sep=0] (bus) at (0, 0){};
            \node[bus_b, anchor=west, minimum width=0.1cm, minimum height=1cm, inner sep=0] (bus) at (1.5, 0){};

            \node[label_tiny] at (1.6, -0.75) {\textcolor{blue}{$\text{Bus}_{i}$}};
        \end{pgfonlayer}

        \begin{pgfonlayer}{foreground}
            \draw[line width=0.5mm, draw = red] (1.5, 0.375) -- (1.75, 0.625);
            \draw[line width=0.5mm, draw = red] (1.5, 0.625) -- (1.75, 0.374);
        \end{pgfonlayer}

        \begin{pgfonlayer}{foreground}
            \draw[line width=0.5mm, draw = green] (-3.125, -0.375) -- (-3, -0.5);
            \draw[line width=0.5mm, draw = green] (-3.025, -0.5) -- (-2.875, -0.25);
            \draw[line width=0.5mm, draw = green] (-1.625, -0.375) -- (-1.5, -0.5);
            \draw[line width=0.5mm, draw = green] (-1.525, -0.5) -- (-1.375, -0.25);
            \draw[line width=0.5mm, draw = green] (-0.125, -0.375) -- (0, -0.5);
            \draw[line width=0.5mm, draw = green] (-0.025, -0.5) -- (0.125, -0.25);

            \draw[line width=0.5mm, draw = green] (-1.875, 0) -- (-1.75, -0.125);
            \draw[line width=0.5mm, draw = green] (-1.775, -0.125) -- (-1.625, 0.125);
            \draw[line width=0.5mm, draw = green] (-0.375, 0) -- (-0.25, -0.125);
            \draw[line width=0.5mm, draw = green] (-0.275, -0.125) -- (-0.125, 0.125);
            \draw[line width=0.5mm, draw = green] (1.125, 0) -- (1.25, -0.125);
            \draw[line width=0.5mm, draw = green] (1.225, -0.125) -- (1.375, 0.125);

        \end{pgfonlayer}
        \end{tikzpicture}

%% file: tikzfig_scp_iterative_approach.tex
\begin{tikzpicture}[
        every node/.style={font=\small},
        process/.style={draw=gray!20, fill=gray!20, minimum width=1.8cm, minimum height=0.8cm, inner sep=0, align=center},
        process_crm/.style={draw=gray!30, fill=gray!30, minimum width=1.8cm, minimum height=0.4cm, inner sep=0, align=center},
        arrow/.style={-Latex, thick},
        label/.style={font=\small, text=darkred}
        ]

        \pgfdeclarelayer{background}
        \pgfdeclarelayer{foreground}
        \pgfsetlayers{background,main,foreground}

        \begin{pgfonlayer}{main}
            \node[process, anchor=west] (block) at (0, 0.15) {Initial\\guess};
            \node[process_crm, anchor=west] (block) at (0, -0.45) {CM};
            \node[process, anchor=west] (block) at (3.5, 0.15) {non-convex\\approx};
            \node[process_crm, anchor=west] (block) at (3.5, -0.45) {CRM-APPX};
            \node[process, anchor=west] (block) at (7, 0) {Convergence\\check};
            \node[process, anchor=west] (block) at (10.5, 0) {End\\Output};
        \end{pgfonlayer}

        \begin{pgfonlayer}{background}
            \draw[arrow] (0.9, 0) -- (3.5, 0);
            \draw[arrow] (4.4, 0) -- (7, 0);
            \draw[arrow] (7.9, 0) -- (10.5, 0);
            \draw[thick] (7.9, 0) -- (7.9, 1.0);
            \draw[thick] (4.4, 1.0) -- (7.9, 1.0);
            \draw[arrow] (4.4, 1.0) -- (4.4, 0.6);
        \end{pgfonlayer}

        \begin{pgfonlayer}{main}
            \node[label] at (9.1, 0.15) {Yes};
            \node[label] at (8.2, 0.6) {No};
        \end{pgfonlayer}
        
    \end{tikzpicture}

%% file: ref.bib
@article{izadi2019optimal,
  title={Optimal placement of protective and controlling devices in electric power distribution systems: A MIP model},
  author={Izadi, Milad and Safdarian, Amir and Moeini-Aghtaie, Moein and Lehtonen, Matti},
  journal={IEEE Access},
  volume={7},
  pages={122827--122837},
  year={2019},
  publisher={IEEE}
}

@article{billinton1996optimal,
  title={Optimal switching device placement in radial distribution systems},
  author={Roy Billinton and Satish Jonnavithula},
  journal={IEEE transactions on power delivery},
  volume={11},
  number={3},
  pages={1646--1651},
  year={1996},
  publisher={IEEE}
}

@techreport{njbpu2022gridmod,
    author = {{Guidehouse Inc}},
    title = {Grid Modernization Study: New Jersey Board of Public Utilities},
    institution = {New Jersey Board of Public Utilities},
    year = {2022},
    url={https://www.nj.gov/bpu/pdf/reports/NJBPU%20Grid%20Modernization%20Final%20Report.pdf}
}

@article{andrews2022data,
  title={Data-driven examination of the impact energy efficiency has on demand response capabilities in institutional buildings},
  author={Andrews, Abigail and Roth, Jonathan and Jain, Rishee K and Mathieu, Johanna L},
  journal={Journal of Engineering for Sustainable Buildings and Cities},
  volume={3},
  number={2},
  pages={024501},
  year={2022},
  publisher={American Society of Mechanical Engineers}
}

@article{wang2017robust,
  title={Robust line hardening strategies for improving the resilience of distribution systems with variable renewable resources},
  author={Wang, Xu and Li, Zhiyi and Shahidehpour, Mohammad and Jiang, Chuanwen},
  journal={IEEE Transactions on Sustainable Energy},
  volume={10},
  number={1},
  pages={386--395},
  year={2017},
  publisher={IEEE}
}

@article{kim2018enhancing,
  title={Enhancing distribution system resilience with mobile energy storage and microgrids},
  author={Kim, Jip and Dvorkin, Yury},
  journal={IEEE Transactions on Smart Grid},
  volume={10},
  number={5},
  pages={4996--5006},
  year={2018},
  publisher={IEEE}
}

@article{baran1989optimal,
  title={Optimal capacitor placement on radial distribution systems},
  author={Baran, Mesut E and Wu, Felix F},
  journal={IEEE Transactions on power Delivery},
  volume={4},
  number={1},
  pages={725--734},
  year={1989},
  publisher={IEEE}
}

@book{griffith2022electrify,
  title={Electrify: An optimist’s playbook for our clean energy future},
  author={Griffith, Saul},
  year={2022},
  publisher={MIT Press}
}

@ARTICLE{thurner2018pandapower,
    author={L. Thurner and A. Scheidler and F. Sch{\"a}fer and J. Menke and J. Dollichon and F. Meier and S. Meinecke and M. Braun},
    journal={IEEE Transactions on Power Systems},
    title={pandapower — An Open-Source Python Tool for Convenient Modeling, Analysis, and Optimization of Electric Power Systems},
    year={2018},
    month={11},
    volume={33},
    number={6},
    pages={6510-6521},
    doi={10.1109/TPWRS.2018.2829021},
    ISSN={0885-8950}}

@INCOLLECTION{boyd_co_3,
   author = {Boyd, Stephen and Vandenberghe, Lieven},
   title={Convex Sets},
   booktitle = {Convex Optimization},
   publisher={Cambridge University Press},
   year={2004},
   chapter = {3},
   pages = {67-125},
   address = {Cambridge}
}

@INCOLLECTION{hoff_afcbsm_10,
   author={Peter D. Hoff},
   title = {Nonconjugate priors and Metropolis-Hastings algorithms},
   booktitle = {A First Course in Bayesian Statistical Methods},
   publisher={Springer},
   year={2009},
   chapter = {10},
   pages = {171-193},
   address = {New York},
   edition={1}
}

@INCOLLECTION{gelman_hmcmc_5,
   author={Steve Brooks and Andrew Gelman and Galin L. Jones and Xiao-Li Meng},
   title = {MCMC Using Hamiltonian Dynamics},
   booktitle = {Handbook of Markov Chain Monte Carlo},
   publisher={Chapman \& Hall/CRC},
   year={2011},
   chapter = {5},
   pages = {113-162},
   address = {New York},
   edition={1st ed.}
}

@INCOLLECTION{martin_bmacip_11,
   author = {Osvaldo A. Martin and Ravin Kumar and Junpeng Lao},
   title = {Appendiceal Topics},
   booktitle = {Bayesian Modeling and Computation in Python},
   publisher = {Chapman \& Hall/CRC},
   month = {12},
   year = {2021},
   chapter = {11},
   pages = {323-382},
   address = {Boca Raton}
}

@INCOLLECTION{mccullagh_glm_4,
   author = {P.McCullagh and J.A. Nelder},
   title = {Binary Data},
   booktitle = {Generalized Linear Model},
   publisher = {Routledge},
   year = {1989},
   chapter = {4},
   pages = {98-148},
   address = {New York},
   edition = {2nd ed.}
}

@INCOLLECTION{richard_epdr_3,
  author    = {Richard E. Brown},
  title     = {Interruption Causes},
  booktitle = {Electric Power Distribution Reliability},
  publisher = {CRC Press},
  year      = {2009},
  chapter   = {3},
  pages     = {107-161},
  address   = {Boca Raton},
  edition   = {2nd ed.}
}

@INCOLLECTION{richard_epdr_4,
  author    = {Richard E. Brown},
  title     = {Component Modeling},
  booktitle = {Electric Power Distribution Reliability},
  publisher = {CRC Press},
  year      = {2009},
  chapter   = {4},
  pages     = {163-190},
  address   = {Boca Raton},
  edition   = {2nd ed.}
}

@INCOLLECTION{pham_sre_1,
  author    = {Hoang Pham},
  title     = {Basic Probability, Statistics, and Reliability},
  booktitle = {Statistical Reliability Engineering: Methods, Models and Applications},
  publisher = {Springer Cham},
  year      = {2021},
  chapter   = {1},
  pages     = {1-65},
  address   = {Switzerland},
  edition   = {1st ed.}
}

@INCOLLECTION{schittkowski1995sequential,
   author="Schittkowski, K. and Zillober, C.",
   editor="Marti, Kurt and Kall, Peter",
   title="Sequential Convex Programming Methods",
   bookTitle="Stochastic Programming: Numerical Techniques and Engineering Applications",
   year="1995",
   publisher="Springer Berlin Heidelberg",
   address="Berlin, Heidelberg",
   pages="123--141"
}

@ARTICLE{cox1972,
    author = {Cox, D. R.},
    title = {Regression Models and Life-Tables.},
    journal = {Journal of the Royal Statistical Society. Series B (Methodological)},
    year = {1972},
    volume = {34},
    number = {2},
    pages = {pp. 187–220}
}

@ARTICLE{duane1987hybrid,
   title = {Hybrid Monte Carlo},
   journal = {Physics Letters B},
   volume = {195},
   number = {2},
   pages = {216-222},
   year = {1987},
   author = {Simon Duane and A.D. Kennedy and Brian J. Pendleton and Duncan Roweth}
}

@ARTICLE{guanyang2022exact,
   author = {Guanyang Wang},
   title = {{Exact convergence analysis of the independent Metropolis-Hastings algorithms}},
   volume = {28},
   journal = {Bernoulli},
   number = {3},
   publisher = {Bernoulli Society for Mathematical Statistics and Probability},
   pages = {2012-2033},
   year = {2022}
}

@ARTICLE{john2023fifty,
    author = {John D. Kalbfleisch and Douglas E. Schaubel},
    title = {Fifty Years of the Cox Model},
    journal = {Annual Review of Statistics and Its Application},
    year = {2023},
    volume = {10},
    number = {1},
    pages = {1–23}
}

@ARTICLE{abbott1985logistic,
  title={Logistic regression in survival analysis},
  author={Abbott, Robert D},
  journal={American journal of epidemiology},
  volume={121},
  number={3},
  pages={465-471},
  year={1985},
  publisher={Oxford University Press}
}

@ARTICLE{low2014convex,
  title={Convex relaxation of optimal power flow—Part I: Formulations and equivalence},
  author={Low, Steven H},
  journal={IEEE Transactions on Control of Network Systems},
  volume={1},
  number={1},
  pages={15-27},
  year={2014},
  publisher={IEEE}
}

@ARTICLE{heidelberger1995fast,
   author = {Heidelberger, Philip},
   title = {Fast simulation of rare events in queueing and reliability models},
   year = {1995},
   publisher = {Association for Computing Machinery},
   address = {New York, NY, USA},
   volume = {5},
   number = {1},
   journal = {ACM Trans. Model. Comput. Simul.},
   month = {1},
   pages = {43–85}
}

@ARTICLE{jens1993exchangeably,
   author = {Jens Praestgaard and Jon A. Wellner},
   journal = {The Annals of Probability},
   number = {4},
   pages = {2053-2086},
   publisher = {Institute of Mathematical Statistics},
   title = {Exchangeably Weighted Bootstraps of the General Empirical Process},
   volume = {21},
   year = {1993}
}

@ARTICLE{king2001logistic, 
   title={Logistic Regression in Rare Events Data}, volume={9},
   number={2}, 
   journal={Political Analysis}, 
   author={King, Gary and Zeng, Langche}, 
   year={2001}, 
   pages={137–163}
}

@ARTICLE{mieth2018data,
  author={Mieth, Robert and Dvorkin, Yury},
  journal={IEEE Control Systems Letters}, 
  title={Data-Driven Distributionally Robust Optimal Power Flow for Distribution Systems}, 
  year={2018},
  volume={2},
  number={3},
  pages={363-368}
}

@ARTICLE{moein2016,
  author    = {Moein Choobineh and Paulo C. Tabares-Velasco and Salman Mohagheghi},
  title     = {Optimal energy management of a distribution network during the course of a heat wave},
  journal   = {Electric Power Systems Research},
  volume    = {130},
  pages     = {230-240},
  year      = {2016}
}

@ARTICLE{mieth2022risk,
  title={Risk-aware dimensioning and procurement of contingency reserve},
  author={Mieth, Robert and Dvorkin, Yury and Ortega-Vazquez, Miguel A},
  journal={IEEE Transactions on Power Systems},
  volume={38},
  number={2},
  pages={1081-1093},
  year={2022},
  publisher={IEEE}
}

@article{robert2018accelerating,
   author = {Robert, Christian P. and Elvira, Víctor and Tawn, Nick and Wu, Changye},
   title = {Accelerating MCMC algorithms},
   journal = {WIREs Computational Statistics},
   volume = {10},
   number = {5},
   pages = {e1435},
   year = {2018}
}

@ARTICLE{luca2022,
   title = {A reliability-centered methodology for identifying renovation actions for improving resilience against heat waves in power distribution grids},
   journal = {International Journal of Electrical Power \& Energy Systems},
   volume = {137},
   pages = {107813},
   year = {2022},
   author = {Luca Bellani and Michele Compare and Enrico Zio and Alessandro Bosisio and Bartolomeo Greco and Gaetano Iannarelli and Andrea Morotti}
}

@ARTICLE{abdin2019,
   title = {A modeling and optimization framework for power systems design with operational flexibility and resilience against extreme heat waves and drought events},
   journal = {Renewable and Sustainable Energy Reviews},
   volume = {112},
   pages = {706-719},
   year = {2019},
   author = {A.F. Abdin and Y.-P. Fang and E. Zio}
}

@ARTICLE{milad2023,
   title = {Mixed-integer exponential conic optimization for reliability enhancement of power distribution systems},
   journal = {Optimization and Engineering},
   year = {2023},
   author = {Filabadi, Milad Dehghani and Chen, Chen and Conejo, Antonio}
}

@inproceedings{zhang2019data,
  title={Data-driven feature description of heat wave effect on distribution system},
  author={Zhang, Yang and Mazza, Andrea and Bompard, Ettore and Roggero, Emiliano and Galofaro, Giuliana},
  booktitle={Proceedings of the 2019 IEEE Milan PowerTech},
  pages={1-6},
  year={2019},
  organization={IEEE}
}

@ARTICLE{mathaios2016,
  author={Panteli, Mathaios and Pickering, Cassandra and Wilkinson, Sean and Dawson, Richard and Mancarella, Pierluigi},
  journal={IEEE Transactions on Power Systems}, 
  title={Power System Resilience to Extreme Weather: Fragility Modeling, Probabilistic Impact Assessment, and Adaptation Measures}, 
  year={2017},
  volume={32},
  number={5},
  pages={3747-3757}
}

@article{wang2016resilience,
  title={Resilience enhancement with sequentially proactive operation strategies},
  author={Wang, Chong and Hou, Yunhe and Qiu, Feng and Lei, Shunbo and Liu, Kai},
  journal={IEEE Transactions on Power Systems},
  volume={32},
  number={4},
  pages={2847-2857},
  year={2016},
  publisher={IEEE}
}

@article{dall2017chance,
  title={Chance-constrained AC optimal power flow for distribution systems with renewables},
  author={Dall’Anese, Emiliano and Baker, Kyri and Summers, Tyler},
  journal={IEEE Transactions on Power Systems},
  volume={32},
  number={5},
  pages={3427-3438},
  year={2017},
  publisher={IEEE}
}

@article{andrianesis2021optimal,
  title={Optimal distributed energy resource coordination: A decomposition method based on distribution locational marginal costs},
  author={Andrianesis, Panagiotis and Caramanis, Michael and Li, Na},
  journal={IEEE Transactions on Smart Grid},
  volume={13},
  number={2},
  pages={1200-1212},
  year={2021},
  publisher={IEEE}
}

@article{bajpai2016novel,
  title={A novel metric to quantify and enable resilient distribution system using graph theory and choquet integral},
  author={Bajpai, Prabodh and Chanda, Sayonsom and Srivastava, Anurag K},
  journal={IEEE Transactions on Smart Grid},
  volume={9},
  number={4},
  pages={2918--2929},
  year={2016},
  publisher={IEEE}
}

@inproceedings{campbell2012weather,
  title={Weather-related power outages and electric system resiliency},
  author={Campbell, Richard J and Lowry, Sean},
  year={2012},
  organization={Congressional Research Service, Library of Congress Washington, DC}
}

@ARTICLE{hanchen2022,
   title = {A sequentially preventive model enhancing power system resilience against extreme-weather-triggered failures},
   journal = {Renewable and Sustainable Energy Reviews},
   volume = {156},
   pages = {111945},
   year = {2022},
   author = {Hanchen Liu and Chong Wang and Ping Ju and Hongyu Li}
}

@article{mieth2019online,
  title={Online learning for network constrained demand response pricing in distribution systems},
  author={Mieth, Robert and Dvorkin, Yury},
  journal={IEEE Transactions on Smart Grid},
  volume={11},
  number={3},
  pages={2563-2575},
  year={2019},
  publisher={IEEE}
}

@ARTICLE{emmanuel2017,
   title = {Evolution of dispatchable photovoltaic system integration with the electric power network for smart grid applications: A review},
   journal = {Renewable and Sustainable Energy Reviews},
   volume = {67},
   pages = {207-224},
   year = {2017},
   author = {Michael Emmanuel and Ramesh Rayudu}
}

@ARTICLE{yuqing2018,
   title = {Battery energy storage system size determination in renewable energy systems: A review},
   journal = {Renewable and Sustainable Energy Reviews},
   volume = {91},
   pages = {109-125},
   year = {2018},
   author = {Yuqing Yang and Stephen Bremner and Chris Menictas and Merlinde Kay}
}

@ARTICLE{nasif2016,
   title = {Review of control strategies for voltage regulation of the smart distribution network with high penetration of renewable distributed generation},
   journal = {Renewable and Sustainable Energy Reviews},
   volume = {64},
   pages = {582-595},
   year = {2016},
   author = {Nasif Mahmud and A. Zahedi}
}

@ARTICLE{jamshid2013,
   title = {Demand response in smart electricity grids equipped with renewable energy sources: A review},
   journal = {Renewable and Sustainable Energy Reviews},
   volume = {18},
   pages = {64-72},
   year = {2013},
   author = {Jamshid Aghaei and Mohammad-Iman Alizadeh}
}

@ARTICLE{hani2023,
title = {A weather-based power distribution system reliability assessment},
journal = {Alexandria Engineering Journal},
volume = {78},
pages = {256-264},
year = {2023},
author = {Hani A. Aldhubaib and Mohamed {Hassan Ahmed} and Magdy M.A. Salama}
}

@ARTICLE{scikitlearn,
  title={Scikit-learn: Machine Learning in {P}ython},
  author={Pedregosa, F. and Varoquaux, G. and Gramfort, A. and Michel, V. and Thirion, B. and Grisel, O. and Blondel, M. and Prettenhofer, P. and Weiss, R. and Dubourg, V. and Vanderplas, J. and Passos, A. and Cournapeau, D. and Brucher, M. and Perrot, M. and Duchesnay, E.},
  journal={Journal of Machine Learning Research},
  volume={12},
  pages={2825-2830},
  year={2011}
}

@INPROCEEDINGS{eckstrom2022outing,
  title={Outing power outages: real-time and predictive socio-demographic analytics for New York City},
  author={Eckstrom, Samuel and Murphy, Graham and Ye, Eileen and Acharya, Samrat and Mieth, Robert and Dvorkin, Yury},
  booktitle={Proceedings of the 2022 IEEE Power \& Energy Society General Meeting (PESGM)},
  pages={1-5},
  year={2022},
  organization={IEEE}
}

@INPROCEEDINGS{billinton2005consideration,
  author={Billinton, R. and Acharya, J.},
  booktitle={Proceedings of the Canadian Conference on Electrical and Computer Engineering, 2005.}, 
  title={Consideration of multi-state weather models in reliability evaluation of transmission and distribution systems}, 
  year={2005},
  pages={916-922}
}

@INPROCEEDINGS{miguel2016,
  author={Miguel A Ortega-Vazquez},
  booktitle={Proceedings of the 2016 IEEE Power and Energy Society General Meeting (PESGM)}, 
  title={Assessment of N-k contingencies in a probabilistic security-constrained optimal power flow}, 
  year={2016},
  pages={1-5}
}

@INPROCEEDINGS{billinton2006distribution,
  author={Billinton, Roy and Acharya, Janak},
  booktitle={Proceedings of the 41st International Universities Power Engineering Conference}, 
  title={Distribution System Reliability Assessment Incorporating Weather Effects}, 
  year={2006},
  volume={1},
  pages={282-286}
}

@INPROCEEDINGS{brown2004failure,
   author={Brown, R.E.},
   booktitle={Proceedings of the IEEE Power Engineering Society General Meeting, 2004.}, 
   title={Failure rate modeling using equipment inspection data}, 
   year={2004},
   pages={693-700 Vol.1}
}

@ONLINE{copernicus,
   author = {{Copernicus Climate Change}},
   title = "ERA5 hourly data on single levels from 1940 to present",
   year = "2024",
   url = "https://cds-beta.climate.copernicus.eu/",
   addendum = "(accessed: 08.01.2024)"
}

@ONLINE{nyiso,
   author = {{NYISO}},
   title = "NYISO Load Data",
   year = "2024",
   url = "https://www.nyiso.com/load-data",
   addendum = "(accessed: 08.01.2024)"
}

@MISC{tensorflow2015whitepaper,
   title={ {TensorFlow}: Large-Scale Machine Learning on Heterogeneous Systems},
   url={https://www.tensorflow.org/},
   note={Software available from tensorflow.org},
   author={Mart\'{i}n~Abadi and Ashish~Agarwal and Paul~Barham and Eugene~Brevdo and Zhifeng~Chen and Craig~Citro and Greg~S.~Corrado and Andy~Davis and Jeffrey~Dean and Matthieu~Devin and Sanjay~Ghemawat and Ian~Goodfellow and Andrew~Harp and Geoffrey~Irving and Michael~Isard and Yangqing Jia and Rafal~Jozefowicz and Lukasz~Kaiser and Manjunath~Kudlur and Josh~Levenberg and Dandelion~Man\'{e} and Rajat~Monga and Sherry~Moore and Derek~Murray and Chris~Olah and Mike~Schuster and Jonathon~Shlens and Benoit~Steiner and Ilya~Sutskever and Kunal~Talwar and Paul~Tucker and Vincent~Vanhoucke and Vijay~Vasudevan and Fernanda~Vi\'{e}gas and Oriol~Vinyals and Pete~Warden and Martin~Wattenberg and Martin~Wicke and Yuan~Yu and Xiaoqiang~Zheng},
  year={2015},
}

@misc{michael2018conceptual,
   title={A Conceptual Introduction to Hamiltonian Monte Carlo}, 
   author={Michael Betancourt},
   year={2018},
   eprint={1701.02434},
   archivePrefix={arXiv},
   primaryClass={stat.ME},
   url={https://arxiv.org/abs/1701.02434}
}

@misc{yuan2024markov,
   title={Markov chain Monte Carlo without evaluating the target: an auxiliary variable approach}, 
   author={Wei Yuan and Guanyang Wang},
   year={2024},
   eprint={2406.05242},
   archivePrefix={arXiv},
   primaryClass={stat.CO},
   url={https://arxiv.org/abs/2406.05242}
}
